\documentclass[leqno,11pt]{article}
\usepackage[utf8x]{inputenc}
\usepackage{microtype}

\usepackage{amsmath}
\usepackage{amsthm}
\usepackage{amssymb}
\usepackage{mathrsfs} 

\usepackage[T1]{fontenc}

\usepackage{ae,aecompl}
\usepackage{times}

\usepackage[english]{babel}

\usepackage[colorlinks,pagebackref,hypertexnames=false]{hyperref}
 \usepackage[alphabetic,backrefs,msc-links]{amsrefs}
\usepackage{bookmark}

\usepackage{esint} 

\usepackage[yyyymmdd]{datetime}

\newcommand{\defined}[2][\key]{%
  \def\key{#2}%
  \textbf{#2}%
  \index{#1}%
}

\newcommand{\Z}{\bZ}

\newcommand{\R}{\bR}
\newcommand{\C}{\bC}
\renewcommand{\H}{\bH}
\newcommand{\Oc}{\bO}

\newcommand{\su}{\mathfrak{su}}
\newcommand{\so}{\mathfrak{so}}

\newcommand{\SO}{{\rm SO}}

\newcommand{\Spin}{{\rm Spin}}
\newcommand{\SU}{{\rm SU}}
\newcommand{\GL}{\mathrm{GL}}
\newcommand{\Cl}{\mathrm{C\ell}}

\newcommand{\U}{{\rm U}}

\newcommand{\Vect}{\mathrm{Vect}}

\DeclareMathOperator{\Aut}{Aut}

\DeclareMathOperator{\Diff}{Diff}

\DeclareMathOperator{\Hom}{Hom}

\DeclareMathOperator{\Lie}{Lie}

\DeclareMathOperator{\ad}{ad}

\newcommand{\ev}{{\rm ev}}
\newcommand{\odd}{{\rm odd}}

\DeclareMathOperator{\im}{im}

\DeclareMathOperator{\tr}{tr}

\newcommand{\End}{\mathop{\rm End}\nolimits}

\renewcommand{\Im}{\operatorname{Im}}
\renewcommand{\Re}{\operatorname{Re}}
\renewcommand{\det}{\operatorname{det}}

\newcommand{\id}{{\rm id}}

\newcommand{\vol}{{\rm vol}}
\newcommand{\dvol}{{\rm vol}}
\newcommand{\zbar}{\overline z}
\newcommand{\A}{\bA}

\renewcommand{\epsilon}{\varepsilon}

\def\({\left(}
\def\){\right)}
\def\<{\left\langle}
\def\>{\right\rangle}

\newcommand{\co}{\mskip0.5mu\colon\thinspace}

\newcommand{\qand}{\quad\text{and}}

\newcommand{\qandq}{\quad\text{and}\quad}

\usepackage{mathtools}

\DeclarePairedDelimiter{\abs}{\lvert}{\rvert}

\DeclarePairedDelimiter{\set}{\lbrace}{\rbrace}
\DeclarePairedDelimiter{\braket}{\langle}{\rangle}

\newcommand{\inner}[2]{\braket{#1, #2}}
\newcommand{\winner}[2]{\braket{#1 \wedge #2}}

\newcommand{\itref}{\eqref}

\makeatletter
\renewcommand\xleftrightarrow[2][]{%
  \ext@arrow 9999{\longleftrightarrowfill@}{#1}{#2}}
\newcommand\longleftrightarrowfill@{%
  \arrowfill@\leftarrow\relbar\rightarrow}
\makeatother

\newcommand{\Gtwo}{G_2}

\newcommand{\rd}{{\rm d}}

\newcommand{\rG}{{\rm G}}
\newcommand{\rH}{{\rm H}}

\newcommand{\rO}{{\rm O}}

\newcommand{\bA}{\mathbf{A}}

\newcommand{\bC}{\mathbf{C}}

\newcommand{\bE}{\mathbf{E}}

\newcommand{\bH}{\mathbf{H}}

\newcommand{\bO}{\mathbf{O}}

\newcommand{\bR}{\mathbf{R}}

\newcommand{\bZ}{\mathbf{Z}}

\newcommand{\cJ}{\mathcal{J}}

\newcommand{\cL}{\mathcal{L}}

\newcommand{\cP}{\mathcal{P}}

\newcommand{\sA}{\mathscr{A}}

\newcommand{\sE}{\mathscr{E}}
\newcommand{\sF}{\mathscr{F}}
\newcommand{\sG}{\mathscr{G}}
\newcommand{\sH}{\mathscr{H}}

\newcommand{\sL}{\mathscr{L}}

\newcommand{\sS}{\mathscr{S}}

\usepackage{slashed}

\newcommand{\slS}{\slashed S}

\numberwithin{equation}{section}

\renewcommand{\eqref}[1]{\hyperref[#1]{\rm(\ref*{#1})}}

\usepackage{aliascnt}

\def\makeautorefname#1#2{\AtBeginDocument{\expandafter\def\csname#1autorefname\endcsname{#2}}}

\newcommand{\mynewtheorem}[2]{
  \newaliascnt{#1}{equation}          
  \newtheorem{#1}[#1]{#2}
  \aliascntresetthe{#1}
  \makeautorefname{#1}{#2}
}

\newcommand{\mynewproblem}[2]{
  \newaliascnt{#1}{myProblem}          
  \newtheorem{#1}[#1]{#2}
  \aliascntresetthe{#1}
  \makeautorefname{#1}{#2}
}

\mynewtheorem{theorem}{Theorem}
\newtheorem*{theorem*}{Theorem}
\mynewtheorem{prop}{Proposition}
\newtheorem*{prop*}{Proposition}
\newtheorem*{claim}{Claim}
\mynewtheorem{cor}{Corollary}
\mynewtheorem{construction}{Construction}
\mynewtheorem{lemma}{Lemma}
\mynewtheorem{conjecture}{Conjecture}
\mynewtheorem{hyp}{Hypothesis}
\newtheorem{step}{Step}

\numberwithin{substep}{step}
\makeautorefname{step}{Step}
\makeautorefname{substep}{Step}

\numberwithin{subcase}{case}
\makeautorefname{case}{Case}
\makeautorefname{subcase}{case}

\theoremstyle{remark}
\mynewtheorem{remark}{Remark}
\newtheorem*{remark*}{Remark}

\theoremstyle{definition}
\mynewtheorem{definition}{Definition}
\mynewtheorem{example}{Example}
\mynewtheorem{exercise}{Exercise}
\mynewproblem{problem}{Problem}
\mynewtheorem{solution}{Solution}
\mynewtheorem{convention}{Convention}
\newtheorem*{convention*}{Convention}
\newtheorem*{conventions*}{Conventions}
\mynewtheorem{question}{Question}

\makeautorefname{table}{Table}        
\makeautorefname{chapter}{Chapter}
\makeautorefname{section}{Section}
\makeautorefname{subsection}{Section}
\makeautorefname{subsubsection}{Section}
\makeautorefname{footnote}{Footnote}

\newcommand{\tsF}{\widetilde{\mathscr{F}}}   
\newcommand{\tsG}{\widetilde{\mathscr{G}}}   
\newcommand{\tsS}{\widetilde{\mathscr{S}}}   

\newcommand{\CS}{{\mathcal{CS}}}    
\newcommand{\RE}{{\mathrm{Re}}}    
\newcommand{\IM}{{\mathrm{Im}}}

\newcommand{\sym}{{\mathrm{sym}}}    

\newcommand{\tf}{{\widetilde f}}    
\newcommand{\tg}{{\widetilde g}}

\newcommand{\tsi}{{\widetilde\sigma}}

\newcommand{\e}{{\mathbf{e}}}    
\renewcommand{\i}{{\mathbf{i}}}    
\renewcommand{\j}{{\mathbf{j}}}    
\renewcommand{\k}{{\mathbf{k}}}    
\newcommand{\one}{\mathbf{1}}

\def\slashii#1{\setbox0=\hbox{$#1$}             
\dimen0=\wd0                                 
\setbox1=\hbox{\sl/} \dimen1=\wd1            
\ifdim\dimen0>\dimen1                        
\rlap{\hbox to \dimen0{\hfil\sl/\hfil}}   
#1                                        
\else                                        
\rlap{\hbox to \dimen1{\hfil$#1$\hfil}}   
\hbox{\sl/}                               
\fi}                                         %
\def\slashiii#1{\setbox0=\hbox{$#1$}#1\hskip-\wd0\hbox to\wd0{\hss\sl/\/\hss}}
\newcommand{\SPAN}{\mathrm{span}}

\newcommand{\grad}{\mathrm{grad}}       
    
\renewcommand{\Im}{\mathrm{Im\,}}

\renewcommand{\Re}{\mathrm{Re\,}}  
\newcommand{\g}{{\mathfrak{g}}}    
\newcommand{\h}{{\mathfrak{h}}}

\newcommand{\ebar}{{\bar{e}}}
\newcommand{\ubar}{{\bar{u}}}
\newcommand{\vbar}{{\bar{v}}}
\newcommand{\wbar}{{\bar{w}}}
\newcommand{\xbar}{{\bar{x}}}

\newcommand{\eps}{{\varepsilon}}    
\newcommand{\om}{{\omega}}    
\newcommand{\Om}{{\Omega}}    
\newcommand{\Inner}[2]{\left\langle #1, #2\right\rangle}    
\def\NABLA#1{{\mathop{\nabla\kern-.5ex\lower1ex\hbox{$#1$}}}}    
\def\Nabla#1{\nabla\kern-.5ex{}_{#1}}    
\def\Tabla#1{\Tilde\nabla\kern-.5ex{}_{#1}}    
\renewcommand{\Tilde}{\widetilde}

\newcommand{\p}{{\partial}}

\title{Notes on the octonions}
\author{
  Dietmar~A. Salamon\footnote{partially supported
  by the Swiss National Science Foundation}  \\
  ETH Z\"urich
  \and
  Thomas Walpuski \\
  Massachusetts Institute of Technology
}    
\usepackage{enumerate} 
\date{2017-04-04}

\begin{document}    

\maketitle    

\begin{abstract}   
  This is an expository paper. %
  Its purpose is to explain the linear algebra that underlies
  Donaldson--Thomas theory and the geometry of Riemannian manifolds
  with holonomy in $\Gtwo$ and $\Spin(7)$.
\end{abstract}

    
\section{Introduction}
\label{sec:intro}  

In these notes we give an exposition of the structures in linear algebra that 
underlie Donaldson--Thomas theory~\cites{Donaldson1998,Donaldson2009} 
and calibrated geometry~\cites{Harvey1982,Joyce2000}. %
No claim is made to originality. %
All the results and ideas described here (except perhaps
\autoref{thm:CAYLEY}) can be found in the existing literature, notably
in the beautiful paper~\cite{Harvey1982} by Harvey and Lawson. %
Perhaps these notes might be a useful introduction for students who
wish to enter the subject. %

Our emphasis is on characterizing the relevant algebraic
structures---such as cross products, triple cross products, associator
and coassociator brackets, associative, coassocitative, and Cayley
calibrations and subspaces---by their intrinsic properties rather than
by the existence of isomorphisms to the standard structures on the
octonions and the imaginary octonions, although both descriptions are
of course equivalent.

\autoref{sec:cross} deals with cross products and their associative
calibrations. %
It contains a proof that they exist only in dimensions $0$, $1$, $3$,
and $7$.  In \autoref{sec:imO} we discuss nondegenerate $3$--forms on
$7$--dimensional vector spaces (associative calibrations) and explain
how they give rise to unique compatible inner products. %
Additional structures such as associative and coassociative subspaces
and the associator and coassociator brackets are discussed in
\autoref{sec:assoc}. %
These structures are relevant for understanding $\Gtwo$--structures on
$7$--manifolds and the Chern--Simons functional in Donaldson--Thomas
theory. %

The corresponding Floer theory has as its counterpart in linear
algebra the product with the real line. %
This leads to the structure of a normed algebra which only exists in
dimensions $1$, $2$, $4$, and $8$, corresponding to the reals, the
complex numbers, the quaternions, and the octonions. %
These structures are discussed in \autoref{sec:NA}. %
Going from Floer theory to an intrinsic theory for Donaldson-type
invariants of $8$--dimensional $\Spin(7)$--manifolds corresponds to
dropping the space-time splitting. %
The algebraic counterpart of this reduction is to eliminate the choice
of the unit (as well as the product). %
What is left of the algebraic structures is the triple cross product
and its Cayley calibration---a suitable $4$--form on an
$8$--dimensional Hilbert space. %
These structures are discussed in \autoref{sec:TCP}. %
\autoref{sec:cayley} characterizes those $4$--forms on
$8$--dimensional vector spaces (the Cayley-forms) that give rise to
(unique) compatible inner products and hence to triple cross
products. %
The relevant structure groups $\Gtwo$ (in dimension $7$) and
$\Spin(7)$ (in dimension $8$) are discussed in \autoref{sec:G2}
and~\autoref{sec:Spin7}
with a particular emphasis on the splitting of the space 
of alternating multi-linear forms into irreducible representations. %
In \autoref{sec:spin} we examine spin structures in dimensions~$7$
and~$8$. %
\autoref{sec:SU} relates $\SU(3)$ and $\SU(4)$ structures to cross
products and triple cross products and \autoref{sec:DT} gives a brief
introduction to the basic setting of Donaldson--Thomas theory.

\medskip

Here is a brief overview of some of the literature about 
the groups
$\Gtwo$ and $\Spin(7)$. %
The concept of a calibration was introduced in the article of
Harvey--Lawson~\cite{Harvey1982} which also contains definitions of
$\Gtwo$ and $\Spin(7)$ in terms of the octonions.  %
Humphreys~\cite{Humphreys1987}*{Section 19.3} constructs (the Lie
algebra of) $\Gtwo$ from the Dynkin diagram and proves that this
coincides with the definition in terms of the octonions. %
The characterization of $\Gtwo$ and $\Spin(7)$ as the stabilisers of
certain $3$-- and $4$--forms is due to Bonan~\cite{Bonan1966}.

Harvey--Lawson also introduced the (multiple) cross products
and the associator and coassociator brackets. %
The concept of a multiple cross product goes back 
to Eckmann~\cite{Eckmann1943}. %
Building on this work, Whitehead \cite{Whitehead1962} classified 
those completely; see also Brown--Gray \cite{Brown1967}. %
To our best knowledge, the splitting of the exterior algebra into
irreducible $\Gtwo$--representations is due to Fern\'andez--Gray
\cite{Fernandez1982}*{Section 3}, who also emphasize the relation
between $\Gtwo$ and the cross product in dimension seven. %
This as well as the analogous result for $\Spin(7)$ can also be found
in Bryant \cite{Bryant1987}*{Section 2}.

Among many others, the more recent articles by Bryant~\cite{Bryant2006}, 
Karigiannis~\cites{Karigiannis2008,Karigiannis2009a,Karigiannis2010}
and Mu\~noz~\cite{Munoz2014}*{Section 2} contain useful summaries 
of the linear algebra related to $\Gtwo$ and $\Spin(7)$.



\section{Cross products} \label{sec:cross}  

We assume throughout that $V$ is a finite dimensional real Hilbert
space. %
\begin{definition}
  \label{def:cross}
  A skew-symmetric bilinear map
  \begin{equation}
    \label{eq:cross}
    V\times V\to V\co (u,v)\mapsto u\times v
  \end{equation}
  is called a \defined{cross product} if it satisfies
  \begin{align}
    \label{eq:orthogonal}
    \inner{u\times v}{u}
    &=
      \inner{u\times v}{v} = 0, \qand \\
    \label{eq:area}
    \abs{u\times v}^2
    &=
      \abs{u}^2\abs{v}^2 - \inner{u}{v}^2
  \end{align}
  for all $u,v\in V$.
\end{definition}
A bilinear map~\eqref{eq:cross} that satisfies~\eqref{eq:area} also
satisfies $u\times u=0$ for all $u\in V$ and, hence, is necessarily
skew-symmetric. %

\begin{theorem}
  \label{thm:cross}
  $V$ admits a cross product if and only if its dimension is either
  $0$, $1$, $3$, or $7$. %
  In dimensions $0$ and $1$ the cross product vanishes, in
  dimension~$3$ it is unique up to sign and determined by an
  orientation of $V$, and in dimension~$7$ it is unique up to
  orthogonal isomorphism.
\end{theorem}

\begin{proof}
  See page~\pageref{proof:cross}.
\end{proof}

The proof of \autoref{thm:cross} is based on the next five lemmas.

\begin{lemma}
  \label{le:phi}
  Let~\eqref{eq:cross} be a skew-symmetric bilinear map. %
  Then the following are equivalent:
  \begin{enumerate}[(i)]
  \item
    \label{It_Orthogonal}
    Equation~\eqref{eq:orthogonal} holds for all $u,v\in V$.
    
  \item
    \label{It_Frobenius}
    For all $u,v,w\in V$ we have
    \begin{equation}
      \label{eq:frobenius}
      \inner{u\times v}{w} = \inner{u}{v\times w}.
    \end{equation}  

  \item
    \label{It_3Form}
    The map ${\phi\co V^3\to\R}$, defined by 
    \begin{equation}
      \label{eq:phi}
      \phi(u,v,w) := \inner{u\times v}{w},
    \end{equation}
    is an alternating $3$--form (called the \defined{associative
      calibration} of $(V,\times)$).
  \end{enumerate}
\end{lemma}

\begin{proof}  
  Let~\eqref{eq:cross} be a skew-symmetric bilinear map. %
  Assume that it satisfies~\eqref{eq:orthogonal}. %
  Then, for all $u,v,w\in V$, we have
  \begin{align*}
    0
    &=
      \inner{v\times(u+w)}{u+w} \\
    &=
      \inner{v\times w}{u} + \inner{v\times u}{w} \\
    &=
      \inner{u}{v\times w}-\inner{u\times v}{w}.
  \end{align*}
  This proves~\eqref{eq:frobenius}. %
  
  Now assume~\eqref{eq:frobenius} and let $\phi$ be defined
  by~\eqref{eq:phi}. %
  Then, by skew-symmetry, we have $\phi(u,v,w)+\phi(v,u,w)=0$ for all
  $u,v,w$ and, by~\eqref{eq:frobenius}, we have
  $\phi(u,v,w)=\phi(v,w,u)$ for all $u,v,w$. %
  Hence, $\phi$ is an alternating $3$--form. %
  Thus we have proved that~\itref{It_Orthogonal}
  implies~\itref{It_Frobenius} implies~\itref{It_3Form}. %

  That~\itref{It_3Form} implies~\itref{It_Orthogonal} is obvious. %
  This proves \autoref{le:phi}.
\end{proof}

\begin{lemma}
  \label{le:area}
  Let~\eqref{eq:cross} be a skew-symmetric bilinear map that
  satisfies~\eqref{eq:orthogonal}. %
  Then the following are equivalent:
  \begin{enumerate}[(i)]
  \item
    \label{It_Area}
    The bilinear map~\eqref{eq:cross} satisfies~\eqref{eq:area}.

  \item
    \label{It_Orthonormal}
    If $u$ and $w$ are orthonormal, then $\abs{u\times w} = 1$.

  \item
    \label{It_Complex}
    If $\abs{u}=1$ and $w$ is orthogonal to $u$, then
    $u\times (u\times w) = - w$.

  \item
    \label{It_CrossSquare}
    For all $u,w\in V$ we have 
    \begin{equation}
      \label{eq:crosscomplex}
      u\times(u\times w) = \inner{u}{w}u - \abs{u}^2w.
    \end{equation}

  \item
    \label{It_CrossSquare2}
    For all $u,v,w\in V$ we have
    \begin{equation}\label{eq:AREA}
      u\times (v\times w) + v\times(u\times w) 
      = \inner{u}{w}v + \inner{v}{w}u - 2\inner{u}{v}w.
    \end{equation}
  \end{enumerate}
\end{lemma}

\begin{proof}
  That~\itref{It_Area} implies~\itref{It_Orthonormal} is obvious.  

  We prove that~\itref{It_Orthonormal} implies~\itref{It_Complex}. %
  Fix a vector $u\in V$ with $\abs{u}=1$ and define the linear map
  $A\co V\to V$ by $Aw := u\times w$. %
  Then, by skew-symmetry and~\eqref{eq:frobenius}, $A$ is skew-adjoint
  and, by~\eqref{eq:orthogonal}, it preserves the subspace
  $W:=u^\perp$. %
  Hence, the restriction of $A^2$ to $W$ is self-adjoint and,
  by~\itref{It_Orthonormal}, it satisfies
  $ \inner{w}{A^2w} = -\abs{u\times w}^2= -\abs{w}^2 $ for $w\in W$. %
  Hence, the restriction of $A^2$ to $W$ is equal to minus the identity. %
  This proves that~\itref{It_Orthonormal} implies~\itref{It_Complex}.

  We prove that~\itref{It_Complex} implies~\itref{It_CrossSquare}. %
  Fix a vector $u\in V$ and define $A\co V\to V$ by ${Aw:=u\times w}$ as
  above. %
  By~\itref{It_Complex} we have $A^2w=-\abs{u}^2w$ whenever $w$ is
  orthogonal to~$u$. %
  Since $A^2u=0$, this implies~\itref{It_CrossSquare}.

  Assertion~\itref{It_CrossSquare2} follows
  from~\itref{It_CrossSquare} by replacing $u$ with $u+v$. %
  To prove that~\itref{It_CrossSquare2} implies~\itref{It_Area}, set
  $w=v$ in~\eqref{eq:AREA} and take the inner product with $u$. %
  Then
  $ \abs{u\times v}^2 = \inner{u}{u\times(v\times v) 
  + v\times(u\times v)} = \abs{u}^2\abs{v}^2-\inner{u}{v}^2$.
  Here the first equality follows from~\eqref{eq:frobenius} and the
  second from~\eqref{eq:AREA} with $w=v$. %
  This proves \autoref{le:area}.
\end{proof}

\bigbreak

\begin{lemma}
  \label{le:cross3}
  Assume $\dim V=3$.  

  \begin{enumerate}[(i)]
  \item
    \label{It_3Uniqueness}
    A cross product on $V$ determines a unique orientation such that
    $u,v,u\times v$ form a positive basis for every pair of linearly
    independent vectors $u,v\in V$.

  \item
    \label{It_3Volume}
    If~\eqref{eq:cross} is a cross product on $V$, then the $3$--form
    $\phi$ given by~\eqref{eq:phi} is the volume form associated to
    the inner product and the orientation in~\itref{It_3Uniqueness}.

  \item
    \label{It_BACCAB}
    If~\eqref{eq:cross} is a cross product on $V$, then
    \begin{equation}
      \label{eq:assoc3}
      (u\times v)\times w = \inner{u}{w}v-\inner{v}{w}u
    \end{equation}
    for all $u,v,w\in V$.

  \item
    \label{It_3Volume2}
    Fix an orientation on $V$ and denote by $\phi\in\Lambda^3V^*$ the
    associated volume form. %
    Then~\eqref{eq:phi} determines a cross product on~$V$.
  \end{enumerate}
\end{lemma}

\begin{proof}
  Assertion~\itref{It_3Uniqueness} follows from the fact that the
  space of pairs of linearly independent vectors in $V$ is connected
  (whenever $\dim V\ne2$). %
  Assertion~\itref{It_3Volume} follows from the fact that, if $u,v$
  are orthonormal, then $u,v,u\times v$ form a positive orthonormal
  basis and
  \begin{equation*}
    \phi(u,v,u\times v)=\abs{u\times v}^2=1.
  \end{equation*}

  We prove~\itref{It_BACCAB}. %
  If $u$ and $v$ are linearly dependent, then both sides
  of~\eqref{eq:assoc3} vanish. %
  Hence we may assume that $u$ and $v$ are linearly independent or,
  equivalently, that $u\times v\ne 0$. %
  Since $(u\times v)\times w$ is orthogonal to $u\times v$, by
  equation~\eqref{eq:frobenius}, and $V$ has dimension~$3$, it follows
  that $(u\times v)\times w$ must be a linear combination of $u$ and
  $v$. %
  The formula~\eqref{eq:assoc3} follows by taking the inner products
  with $u$ and $v$, and using
  \autoref{le:area}~\itref{It_CrossSquare2}.

  We prove~\itref{It_3Volume2}. %
  Assume that the bilinear map~\eqref{eq:cross} is defined
  by~\eqref{eq:phi}, where $\phi$ is the volume form associated to an
  orientation of $V$. %
  Then skew-symmetry and~\eqref{eq:orthogonal} follow from the fact
  that $\phi$ is a $3$--form (see \autoref{le:phi}). %
  If $u,v$ are linearly independent, then by~\eqref{eq:phi} we have
  \begin{equation*}
    u\times v\ne 0
  \end{equation*}
  and 
  \begin{equation*}
    {\phi(u,v,u\times v)=\abs{u\times v}^2>0}.
  \end{equation*}
  If $u,v$ are orthonormal, it follows that ${u,v,u\times v}$ 
  is a positive orthogonal basis and so 
  \begin{equation*}
    \phi(u,v,u\times v)=\abs{u\times v}.
  \end{equation*}
  Combining these two identities we obtain 
  $\abs{u\times v}=1$ when $u,v$ are orthonormal. %
  Hence,~\eqref{eq:area} follows from \autoref{le:area}. %
  This proves \autoref{le:cross3}.
\end{proof}

\begin{example}
  \label{ex:cross3}
  On $\R^3$ the cross product associated to the standard inner product
  and the standard orientation is given by the familiar formula
  \begin{equation*}
    u\times v
    =
    \begin{pmatrix}
      u_2v_3-u_3v_2 \\
      u_3v_1-u_1v_3 \\
      u_1v_2-u_2v_1
    \end{pmatrix}.
  \end{equation*}
\end{example}

\begin{example}\label{ex:cross7}
  The standard structure on $\R^7$ can be obtained from 
  a basis of the form $\i,\j,\k,\e,\e\i,\e\j,\e\k$,
  where $\i,\j,\k,\e$ are anti-commuting
  generators with square minus one and $\i\j=\k$. %
  Then the cross product is given by
  \begin{equation}\label{eq:cross7}
    u\times v
    :=
    \begin{pmatrix*}[r]
      u_2v_3 - u_3v_2 - u_4v_5 + u_5v_4 - u_6v_7 + u_7v_6 \\
      u_3v_1 - u_1v_3 - u_4v_6 + u_6v_4 - u_7v_5  + u_5v_7 \\
      u_1v_2 - u_2v_1 - u_4v_7 + u_7v_4 - u_5v_6 + u_6v_5 \\
      u_1v_5 - u_5v_1 + u_2v_6 - u_6v_2 + u_3v_7 - u_7v_3 \\
      -u_1v_4 + u_4v_1 - u_2v_7 + u_7v_2 + u_3v_6 -u_6v_3  \\
      u_1v_7-u_7v_1 - u_2v_4 + u_4v_2 - u_3v_5 + u_5v_3  \\
      -u_1v_6 + u_6v_1+ u_2v_5 - u_5v_2 - u_3v_4 + u_4v_3  
    \end{pmatrix*}.
  \end{equation}
  With 
  \begin{equation*}
    e^{ijk}:=dx_i\wedge dx_j\wedge dx_k
  \end{equation*}
  the associated $3$--form~\eqref{eq:phi} is given by
  \begin{equation}
    \label{Eq_phi0}
    \phi_0 = e^{123} - e^{145} - e^{167} - e^{246} - e^{275} - e^{347} - e^{356}.
  \end{equation}
  The product~\eqref{eq:cross7} is skew-symmetric
  and~\eqref{eq:frobenius} follows from the fact that the matrix
  $A(u)$ defined by
  \begin{equation*}
    A(u)v:=u\times v 
  \end{equation*}
  is skew symmetric for all $u$, namely, 
  \begin{equation*}
    A(u)
    := 
    \begin{pmatrix*}[r]
      0\;\,& -u_3 &  u_2 &  u_5 & -u_4 &  u_7 & -u_6 \\
      u_3 & 0\;\,& -u_1 &  u_6 & -u_7 & -u_4 &  u_5 \\
      -u_2 &  u_1 & 0\;\,&  u_7 &  u_6 & -u_5 & -u_4 \\
      -u_5 & -u_6 & -u_7 & 0\;\,&  u_1 &  u_2 &  u_3 \\
      u_4 &  u_7 & -u_6 & -u_1 & 0\;\,&  u_3 & -u_2 \\
      -u_7 &  u_4 &  u_5 & -u_2 & -u_3 & 0\;\,&  u_1 \\
      u_6 & -u_5 &  u_4 & -u_3 &  u_2 & -u_1 & 0\;\, 
    \end{pmatrix*}.
  \end{equation*}
  We leave it to the reader to verify~\eqref{eq:area} (or equivalently
  $\abs{u\times v}=1$ whenever $u$ and $v$ are orthonormal). %
 
  See also \autoref{rmk:cross7} below.
\end{example}

\begin{lemma}
  \label{le:phisymp}
  Let $V$ be a be a real Hilbert space and~\eqref{eq:cross} be
  a cross product on $V$. %
  Let $\phi\in\Lambda^3V^*$ be given by~\eqref{eq:phi}. %
  Then the following holds:
  \begin{enumerate}[(i)]
  \item
    \label{It_Symplectic}
    Let $u\in V$ be a unit vector and $W_u:=u^\perp$. %
    Define ${\om_u\co W_u\times W_u\to\R}$ and $J_u\co W_u\to W_u$ by
    \begin{equation*}
      \om_u(v,w)
      :=
        \inner{u}{v\times w},\qquad 
      J_uv
      :=
        u\times v
    \end{equation*}
    for $v,w\in W_u$. %
    Then $\om_u$ is a symplectic form on $W_u$, $J_u$ is a complex
    structure compatible with $\om_u$, and the associated inner
    product is the one inherited from $V$. %
    In particular, the dimension of $V$ is odd.

  \item
    \label{It_Volume}
    Suppose $\dim V=2n+1\ge3$. %
    Then there is a unique orientation of $V$ such that the associated
    volume form $\dvol\in\Lambda^{2n+1}V^*$ satisfies
    \begin{equation}
      \label{eq:phin}
      \left(\iota(u)\phi\right)^{n-1}\wedge\phi = n!\abs{u}^{n-1}\dvol
    \end{equation}
    for every $u\in V$. %
    In particular, $n$ is odd.
  \end{enumerate}
\end{lemma}

\begin{proof}
  We prove~\itref{It_Symplectic}. %
  By \autoref{le:phi} the bilinear form $\om_u$ is skew symmetric and,
  by \autoref{le:area}, we have $J_u\circ J_u=-\one$. %
  Moreover,
  \begin{equation*}
    \om_u(v,J_uw)
    = \inner{u\times v}{u\times w}
    = -\inner{v}{u\times(u\times w)}
    = \inner{v}{w}
  \end{equation*}
  for all $v,w\in V$. %
  Here the first equation follows from the definition of $\om_u$ and
  $J_u$, the second follows from~\eqref{eq:frobenius}, and the last
  from \autoref{le:area}. %
  Thus the dimension of $W_u$ is even and so the dimension of $V$ is
  odd. %
 
  We prove~\itref{It_Volume}. %
  The set of all bases $(u,v_1,\dots,v_{2n})\in V^{2n+1}$, where $u$
  has norm one and $v_1,\dots,v_{2n}$ is a symplectic basis of $W_u$,
  is connected. %
  Hence, there is a unique orientation of $V$ with respect to which
  every such basis is positive. %
  Let $\dvol\in\Lambda^{2n+1}V^*$ be the associated volume form. %
  To prove equation~\eqref{eq:phin} assume first that $\abs{u}=1$ and
  choose an orthonormal symplectic basis $v_1,\dots,v_{2n}$ of
  $W_u$. %
  (For example pick an orthonormal basis $v_1,v_3,\dots,v_{2n-1}$ of a
  Lagrangian subspace of $W_u$ and define $v_{2k}:=J_uv_{2k-1}$ for
  $k=1,\dots,n$.) %
  Now evaluate both sides of the equation on the tuple
  $(u,v_1,\dots,v_{2n})$. %
  Then we obtain $n!$ on both sides. %
  This proves~\eqref{eq:phin} whenever $u$ has norm one. %
  The general case follows by scaling. %
  It follows from~\eqref{eq:phin} that $n$ is odd since otherwise the
  left hand side changes sign when we replace $u$ by $-u$. %
  This proves \autoref{le:phisymp}.
\end{proof}

\begin{lemma}
  \label{le:Vphi}
  Let $n>1$ be an odd integer and $V$ be an oriented real Hilbert
  space of dimension $2n+1$ with volume form
  $\dvol\in\Lambda^{2n+1}V^*$. %
  Let $\phi\in\Lambda^3V^*$ be a $3$--form and denote its isotropy
  group by
  \begin{equation*}
    \rG:=\left\{g\in\Aut(V) : g^*\phi=\phi\right\}.
  \end{equation*}
  If $\phi$ satisfies~\eqref{eq:phin}, then $\rG\subset\SO(V)$.
\end{lemma}

\begin{proof}
  Let $g\in\rG$ and $u\in V$. %
  Then it follows from~\eqref{eq:phin} that
  \begin{align*}
    \abs{gu}^{n-1}g^*\dvol
    &=
      \frac{1}{n!}g^*\left(\left(\iota(gu)\phi\right)^{n-1}\wedge\phi\right) \\
    &= 
      \frac{1}{n!}\left(\left(g^*\iota(gu)\phi\right)^{n-1}\wedge g^*\phi\right) \\
    &=
      \frac{1}{n!} \left(\iota(u)g^*\phi\right)^{n-1}\wedge g^*\phi  \\
    &= 
      \frac{1}{n!} \left(\iota(u)\phi\right)^{n-1}\wedge\phi  \\
    &= 
      \abs{u}^{n-1}\dvol.
  \end{align*}
  Hence, there is a constant $c>0$ such that 
  \begin{equation*}
    g^*\dvol = c^{-1}\dvol,\qquad \abs{gu}^{n-1} = c\abs{u}^{n-1}
  \end{equation*}
  for every $u\in V$. %
  Since $n>1$, this gives 
  $
  \abs{gu} = c^{\frac{1}{n-1}}\abs{u}
  $
  for $u\in V$ and hence
  \begin{equation*}
    g^*\dvol
    = c^{\frac{2n+1}{n-1}}\dvol 
    = c^{\frac{3n}{n-1}}g^*\dvol.
  \end{equation*}
  Thus $c=1$ and this proves \autoref{le:Vphi}. 
\end{proof}

\begin{proof}[Proof of \autoref{thm:cross}]\label{proof:cross}
  Assume $\dim V>1$, let~\eqref{eq:cross} be a cross product on $V$,
  and define $\phi\co V\times V\times V\to\R$ by~\eqref{eq:phi}. %
  By \autoref{le:phi}, we have $\phi\in\Lambda^3V^*$. %
  By \autoref{le:phisymp}~\itref{It_Symplectic}, the dimension of $V$ is odd. %
  By \autoref{le:Vphi}, we have $\dim V=4n+3$ for some integer $n\ge0$. %
  In particular $\dim V\ne 5$. 

  We prove that $\dim V \le 7$. %
  Define $A\co V\to\End(V)$ 
  by 
  $
  A(u)v:=u\times v.
  $ 
  Then it follows from \autoref{le:area} that
  \begin{equation*}
    A(u)u=0,\qquad A(u)^2=uu^*-\abs{u}^2\one.
  \end{equation*}
  Define $\gamma\co V\to\End(\R\oplus V)$ by 
  \begin{equation}
    \label{eq:gamma}
    \gamma(u)
    :=
    \begin{pmatrix}
      0 & -u^* \\
      u & A(u)
    \end{pmatrix},
  \end{equation}
  where $u^*\co V\to\R$ denotes the linear functional $v\mapsto\inner{u}{v}$.
  Then 
  \begin{equation}
    \label{eq:gammau}
    \gamma(u)^*+\gamma(u)
      = 0,\qquad
    \gamma(u)^*\gamma(u)
      = \abs{u}^2\one
  \end{equation}
  for every $u\in V$. %
  Here the first equation follows from the fact that $A(u)$ is
  skew-adjoint for every $u$ and the last equation follows by direct
  calculation. %
  This implies that $\gamma$ extends to a linear map from the Clifford
  algebra $\Cl(V)$ to $\End(\R\oplus V)$. %
  The restriction of this extension to the Clifford algebra of any
  even dimensional subspace of $V$ is injective (see,
  e.g.~\cite{Salamon1999a}*{Proposition~4.13}). %
  Hence, $2^{2n}\le (2n+2)^2$. %
  This implies $n\le 3$ and so $\dim V=2n+1\le 7$.  Thus we have
  proved that the dimension of $V$ is either $0$, $1$, $3$, or $7$. %
  That the cross product vanishes in dimension $0$ and $1$ is
  obvious. %
  That it is uniquely determined by the orientation of $V$ in
  dimension $3$ follows from \autoref{le:cross3}. %
  The last assertion of \autoref{thm:cross} is restated and proved in
  \autoref{thm:imO} below.
\end{proof}

\begin{remark}
  \label{rmk:5}
  Let $V$ be a nonzero real Hilbert space that admits a $3$--form
  $\phi$ whose isotropy subgroup $\rG$ is contained in $\SO(V)$. %
  Then
  \begin{equation*}
    \dim\Aut(V)-\dim\Lambda^3V^*\le\dim\rG\le\dim\SO(V).
  \end{equation*}
  Hence, $\dim V\ge7$ as otherwise
  $\dim\SO(V)<\dim\Aut(V)-\dim\Lambda^3V^*$. %
  This gives another proof for the nonexistence of cross products in
  dimension $5$.
\end{remark}

    
\section{Associative calibrations}
\label{sec:imO}  

\begin{definition}\label{def:nondegenerate}
  Let $V$ be a real vector space. %
  A $3$--form $\phi\in\Lambda^3V^*$ is called \defined{nondegenerate} if,
  for every pair of linearly independent vectors $u,v\in V$, there is
  a vector $w\in V$ such that $\phi(u,v,w)\ne 0$.  %
  An inner product on $V$ is called \defined{compatible} with $\phi$ if
  the map~\eqref{eq:cross} defined by~\eqref{eq:phi} is a cross
  product.
\end{definition}

\begin{theorem}\label{thm:imO}
  Let $V$ be a $7$--dimensional real vector space and
  $\phi,\phi'\in\Lambda^3V^*$. %
  Then the following holds:
  \begin{enumerate}[(i)]
  \item
    \label{It_Compatible}
    $\phi$ is nondegenerate if and only if it admits a compatible inner product.

  \item
    \label{It_UniqueG}
    The inner product in~\itref{It_Compatible}, if it exists, is uniquely determined by $\phi$. 

  \item
    \label{It_Transitive}
    If $\phi$ and $\phi'$ are nondegenerate, the vectors $u,v,w$ are orthonormal 
    for~$\phi$ and satisfy $\phi(u,v,w)=0$, and the vectors $u',v',w'$ are orthonormal 
    for~$\phi'$ and satisfy ${\phi'(u',v',w')=0}$, then there exists a $g\in\Aut(V)$ 
    such that ${g(u)=u'}$, ${g(v)=v'}$, ${g(w)=w'}$, and ${g^*\phi'=\phi}$.
  \end{enumerate}
\end{theorem}

\begin{proof}
  See pages~\pageref{proof:imO1} and~\pageref{proof:imO2}.
\end{proof}

\begin{remark}\label{rmk:nondeg}
  If $\dim V=3$, then $\phi\in\Lambda^3V^*$ is nondegenerate if and
  only if it is nonzero. %
  If $\phi\ne 0$, then, by \autoref{le:cross3}, an inner product on $V$
  is compatible with $\phi$ if and only if $\phi$ is the associated
  volume form with respect to some orientation, i.e.,
  $\phi(u,v,w)=\pm1$ for every orthonormal basis $u,v,w$ of $V$. %
  Thus assertion~\itref{It_Compatible} of \autoref{thm:imO} continues
  to hold in dimension three. %

  However, assertion~\itref{It_UniqueG} is specific to dimension seven.
\end{remark}

\begin{lemma}\label{le:imO}
  Let $V$ be a $7$--dimensional real Hilbert space and
  $\phi\in\Lambda^3V^*$. %
  Then the following are equivalent:
  \begin{enumerate}[(i)]
  \item
    \label{It_Compatible2}
    $\phi$ is compatible with the inner product.

  \item
    \label{It_Volume2}
    There is an orientation on $V$ such that the associated 
    volume form $\dvol\in\Lambda^7V^*$ satisfies
    \begin{equation}\label{eq:uvphi}
      \iota(u)\phi\wedge\iota(v)\phi\wedge\phi = 6\inner{u}{v}\dvol
    \end{equation}
    for all $u,v\in V$.
  \end{enumerate}
  Each of these conditions implies that $\phi$ is nondegenerate.
  Moreover, the orientation in~\itref{It_Volume2}, if it exists, is
  uniquely determined by $\phi$.
\end{lemma}

\begin{remark}
  \label{rmk:cross7}
  It is convenient to use equation~\eqref{eq:uvphi} to verify that the
  bilinear map in Example~\ref{ex:cross7} satisfies~\eqref{eq:area}. %
  In fact, it suffices to check~\eqref{eq:uvphi} for every pair of
  standard basis vectors. %
  Care must be taken. %
  There are examples of $3$--forms $\phi$ on $V=\R^7$ for which the
  quadratic form
  \begin{equation*}
    V\times V\to\Lambda^7V^*\co
    (u,v)\mapsto\iota(u)\phi\wedge\iota(v)\phi\wedge\phi
  \end{equation*}
  has signature $(3,4)$. %
  One such example can be obtained from the $3$--form $\phi_0$ in
  \autoref{ex:cross7} by changing the minus signs to plus.
\end{remark}

\begin{proof}[Proof of \autoref{le:imO}]
  If~\itref{It_Compatible2} holds, then, by
  \autoref{le:phisymp}~\itref{It_Volume}, there is a unique
  orientation on $V$ such that the associated volume form satisfies
  \begin{equation*}
    \iota(u)\phi\wedge\iota(u)\phi\wedge\phi = 6\abs{u}^2\dvol
  \end{equation*}
  for every $u\in V$. %
  Applying this identity to $u+v$ and $u-v$ and taking the difference
  we obtain~\eqref{eq:uvphi}. %
  Moreover, if $u,v\in V$ are linearly independent, then
  $\phi(u,v,u\times v) = \abs{u\times v}^2 =
  \abs{u}^2\abs{v}^2-\inner{u}{v}^2\ne 0$. %
  Hence, $\phi$ is nondegenerate. %
  This shows that~\itref{It_Compatible2} implies~\itref{It_Volume2}
  and nondegeneracy.

  Conversely, assume~\itref{It_Volume2}. %
  We prove that $\phi$ is nondegenerate. %
  Let $u,v\in V$ be linearly independent. %
  Then $u\ne 0$ and, hence, by~\eqref{eq:uvphi}, the $7$--form
  \begin{equation*}
    \sigma
    :=
      \iota(u)\phi\wedge\iota(u)\phi\wedge\phi
    =
      6\abs{u}^2\dvol\in\Lambda^7V^*
  \end{equation*}
  is nonzero. %
  Choose a basis $v_1,\dots,v_7$ of $V$ with $v_1=u$ and $v_2=v$. %
  Evaluating $\sigma$ on this basis we obtain that one of the terms
  $\phi(u,v,v_j)$ with $j\ge 3$ must be nonzero. %
  Hence, $\phi$ is nondegenerate as claimed. %

  Now define the bilinear map
  $V\times V\to V\co (u,v)\mapsto u\times v$ by~\eqref{eq:phi}. %
  This map is skew-symmetric and, by \autoref{le:phi}, it
  satisfies~\eqref{eq:orthogonal}. %
  We must prove that it also satisfies~\eqref{eq:area}. %
  By \autoref{le:area}, it suffices to show
  \begin{equation}\label{eq:uxv}
    \abs{u}=1,\;\;
    \inner{u}{v}=0
      \qquad\implies\qquad
    \abs{u\times v}=\abs{v}.
  \end{equation}
  We prove this in five steps. %
  Throughout we fix a unit vector $u\in V$.

  \setcounter{step}{0}
  \begin{step}
    \label{Pf_ImO1}
    Define the linear map $A\co V\to V$ by $Av:=u\times v$. %
    Then $A$ is skew-adjoint and its kernel is spanned by $u$.
  \end{step}

  That $A$ is skew-adjoint follows from the identity
  $\inner{Av}{w}=\phi(u,v,w)$. %
  That its kernel is spanned by $u$ follows from the fact that $\phi$
  is nondegenerate.

  \begin{step}
    \label{Pf_ImO2}
    Let $A$ be as in \autoref{Pf_ImO1}. %
    Then there are positive constants $\lambda_1,\lambda_2,\lambda_3$
    and an orthonormal basis $v_1,w_1,v_2,w_2,v_3,w_3$ of $u^\perp$
    such that $Av_j=\lambda_jw_j$ and $Aw_j=-\lambda_jv_j$ for
    $j=1,2,3$.
  \end{step}

  By \autoref{Pf_ImO1}, there is a constant $\lambda>0$ and a vector 
  $v\in u^\perp$ such that 
  \begin{equation*}
    A^2v=-\lambda^2v,\qquad \abs{v}=1.
  \end{equation*}
  Denote $w:=\lambda^{-1}Av$. %
  Then $Av=\lambda w$, $Aw=-\lambda v$, $w$ is orthogonal to $v$, and
  \begin{equation*}
    \abs{w}^2
    = \lambda^{-2}\inner{Av}{Av}
    =-\lambda^{-2}\inner{v}{A^2v}
    = \abs{v}^2=1.
  \end{equation*}
  Moreover, the orthogonal complement of $u,v,w$ is invariant under
  $A$. %
  Hence, \autoref{Pf_ImO2} follows by induction.

  \begin{step}
    \label{Pf_ImO3}
    Let $\lambda_i$ be as in \autoref{Pf_ImO2}. %
    Then $\lambda_1\lambda_2\lambda_3=1$.
  \end{step}

  Let $A$ be as in \autoref{Pf_ImO1}, denote $W:=u^\perp$, and define
  $\omega\co W\times W\to\R$ by
  \begin{equation*}
    \omega(v,w):=\inner{Av}{w} = \phi(u,v,w)
  \end{equation*}
  for $v,w\in W$. %
  Then, by \autoref{Pf_ImO1}, $\omega\in\Lambda^2W^*$ is a symplectic form. %
  Moreover, $\omega(v_i,w_i)=\inner{Av_i}{w_i}=\lambda_i$ for
  $i=1,2,3$ while $\omega(v_i,w_j)=0$ for $i\ne j$ and
  $\omega(v_i,v_j)=\omega(w_i,w_j)=0$ for all $i$ and $j$. %
  Hence,
  \begin{align*}
    \lambda_1\lambda_2\lambda_3
    &=
      \frac{1}{6}\om^3(v_1,w_1,v_2,w_2,v_3,w_3) \\
    &=
      \dvol(u,v_1,w_1,v_2,w_2,v_3,w_3).
  \end{align*}
  Here the first equation follows from \autoref{Pf_ImO2} and the definition of
  $\omega$ and the second equation follows from~\eqref{eq:uvphi} with
  $u=v$ and $\abs{u}=1$. %
  Since the vectors $u,v_1,w_1,v_2,w_2,v_3,w_3$ form an orthonormal
  basis of $V$, the last expression must be plus or minus one. %
  Since it is positive, \autoref{Pf_ImO3} follows.

  \begin{step}
    \label{Pf_ImO4}
    Define
    \begin{equation}
      \label{eq:GH}
      \rG := \left\{g\in\Aut(V) : g^*\phi=\phi\right\},
      \qquad
      \rH := \left\{g\in\rG : gu=u\right\}.
    \end{equation}
    Then $\dim\rG\ge14$ and $\dim\rH\ge 8$
  \end{step}

  Since $\dim\Aut(V)=49$ and $\dim\Lambda^3V^*=35$, the isotropy
  subgroup $\rG$ of $\phi$ has dimension at least $14$. %
  Moreover, by \autoref{le:Vphi}, $\rG$ acts on the sphere
  $S:=\left\{v\in V : \abs{v}=1\right\}$ which has dimension~$6$. %
  Thus the isotropy subgroup~$\rH$ of $u$ under this action has
  dimension ${\dim\rH \ge \dim\rG - \dim S \ge 14-6 = 8}$. %
  This proves \autoref{Pf_ImO4}.

  \begin{step}
    \label{Pf_ImO5}
    Let $\lambda_i$ be as in \autoref{Pf_ImO2}. %
    Then $\lambda_1=\lambda_2=\lambda_3=1$.
  \end{step}

  By definition of $A$ in \autoref{Pf_ImO1} and $\rH$ in \autoref{Pf_ImO4}, we have
  $\inner{Agv}{gw}=\inner{Av}{w}$ for all $g\in\rH$ and all
  $v,w\in V$. %
  Moreover, $\rH\subset\SO(V)$, by \autoref{le:Vphi}. %
  Hence,
  \begin{equation}
    \label{eq:gH}
    g\in\rH\qquad\implies\qquad gA=Ag.
  \end{equation}
  Now suppose that the eigenvalues $\lambda_1,\lambda_2,\lambda_3$ are
  not all equal. %
  Without loss of generality, we may assume
  $\lambda_1\notin\{\lambda_2,\lambda_3\}$. %
  Then, by~\eqref{eq:gH}, the subspaces $W_1:=\SPAN\{v_1,w_1\}$ and
  $W_{23}:=\SPAN\{v_2,w_2,v_3,w_3\}$ are preserved by each element
  $g\in\rH$. %
  Thus $\rH\subset\rO(W_1)\times\rO(W_{23})$. %
  Since $\dim\rO(W_1)=1$ and $\dim\rO(W_{23})=6$, this implies
  $\dim\rH\le 7$ in contradiction to \autoref{Pf_ImO4}. %
  Thus we have proved that $\lambda_1=\lambda_2=\lambda_3$ and, by
  \autoref{Pf_ImO3}, this implies $\lambda_j=1$ for every $j$. %
  This proves \autoref{Pf_ImO5}.

  By \autoref{Pf_ImO2} and \autoref{Pf_ImO5} we have $A^2v=-v$ for every $v\in u^\perp$. %
  Hence, by \autoref{Pf_ImO1}, $\abs{Av}^2=-\inner{v}{A^2v}=\abs{v}^2$ for every
  $v\in u^\perp$. %
  By definition of $A$, this proves~\eqref{eq:uxv} and
  \autoref{le:imO}.
\end{proof}

\begin{proof}[Proof of \autoref{thm:imO}~\itref{It_Compatible} and~\itref{It_UniqueG}]\label{proof:imO1}
  The ``if'' part of~\itref{It_Compatible} is the last assertion made in
  \autoref{le:imO}. %
  To prove~\itref{It_UniqueG} and the ``only if'' part
  of~\itref{It_Compatible} we assume that $\phi$ is nondegenerate. %
  Then, for every nonzero vector $u\in V$, the restriction of the
  $2$--form $\iota(u)\phi\in\Lambda^2V^*$ to $u^\perp$ is a symplectic
  form. %
  Namely, if $v\in u^\perp$ is nonzero, then $u,v$ are linearly
  independent and hence there is a vector $w\in V$ such that
  $\phi(u,v,w)\ne 0$; the vector $w$ can be chosen orthogonal to $u$. %

  This implies that the restriction of the $6$--form
  $\left(\iota(u)\phi\right)^3\in\Lambda^6V^*$ to $u^\perp$ is nonzero
  for every nonzero vector $u\in V$. %
  Hence, the $7$--form
  $\iota(u)\phi\wedge\iota(u)\phi\wedge\phi\in\Lambda^7V^*$ is nonzero
  for every nonzero vector $u\in V$. %
  Since $V\setminus\{0\}$ is connected, there is a unique orientation
  of $V$ such that $\iota(u)\phi\wedge\iota(u)\phi\wedge\phi$ is a
  positive volume form on $V$ for every $u\in V\setminus\{0\}$. %
  Fix a volume form $\sigma\in\Lambda^7V^*$ compatible with this
  orientation. %
  Then the bilinear form
  \begin{equation*}
    V\times V\to\R\co (u,v)\mapsto 
    \frac{\iota(u)\phi\wedge\iota(v)\phi\wedge\phi}{\sigma}=:g(u,v)
  \end{equation*}
  is an inner product. %
  Define $\mu>0$ by $\sigma=\mu\dvol_g$. %
  Replacing $\sigma$ by $\tsi:=\lambda^2\sigma$ we get
  \begin{equation*}
  \tg=\lambda^{-2}g,\qquad
  \dvol_\tg=\lambda^{-7}\dvol_g. 
  \end{equation*}
  Thus 
  \begin{equation*}
    \tsi = \lambda^2\sigma = \lambda^2\mu\dvol_g=\lambda^9\mu\dvol_\tg.
  \end{equation*}
  With $\lambda:=(6/\mu)^{1/9}$ we get $\tsi=6\dvol_\tg$.

  Thus we have proved that there is a unique orientation and inner
  product on $V$ such that $\phi$ satisfies~\eqref{eq:uvphi}.  Hence
  the assertion follows from \autoref{le:imO}.  
  This proves parts~\itref{It_Compatible} and~\itref{It_UniqueG} of~\autoref{thm:imO}.
\end{proof}

\begin{remark}
  \label{Rmk_Hitchin}
  Let $V,W$ be $n$-dimensional real vector spaces.
  Then the determinant of a linear map $A\co V\to W$ 
  is an element $\det A\in \Lambda^nV^*\otimes\Lambda^nW$.
  In particular, if $V$ is equipped with an orientation and an inner product $g \in S^2V^*$,
  and $i_g\co V \to V^*$ denotes the isomorphism defined by $i_g v := g(v,\cdot)$,
  then $\det i_g \in (\Lambda^n V^*)^2$ and the volume form $\vol_g$ 
  associated to $g$ is
  \begin{equation*}
    \vol_g = \sqrt{\det i_g}.
  \end{equation*}
  Here the orientation is needed to determine the sign of the square root.
 
  If $V$ is $7$--dimensional and $\phi \in \Lambda^3 V^*$ is nondegenerate, 
  then the formula
  \begin{equation*}
    G(u,v) := \frac16 i(u)\phi \wedge i(v)\phi \wedge \phi
    \qquad\mbox{for }u,v\in V
  \end{equation*}
  defines a symmetric bilinear form $G\co V\times V\to \Lambda^7 V^*$ 
  and $i_G \co V \to V^*\otimes \Lambda^7 V^*$ is an isomorphism 
  (see second paragraph in the proof of \autoref{le:imO}).
  The determinant of $i_G$ is an element of $(\Lambda^7 V^*)^9$ 
  and $(\det i_G)^{1/9}$ can be defined without an orientation on $V$.
  If an inner product $g$ and an orientation on $V$ are such that \eqref{eq:uvphi} holds, then
  \begin{equation*}
    \vol_g = (\det(i_G))^{1/9}
    \quad\text{and}\quad
    g = \frac{G}{\vol_g}.
  \end{equation*}
  Conversely, with this choice of inner product and orientation, \eqref{eq:uvphi} holds.
  This observation is due to Hitchin \cite{Hitchin2001}*{Section 8.3}.
\end{remark}

\begin{lemma}
  \label{le:uvw}
  Let $V$ be a $7$--dimensional real Hilbert space equipped with a
  cross product $ V\times V\to V\co (u,v)\to u\times v$. %
  If $u$ and $v$ are orthonormal and $w:=u\times v$, then $v\times w=u$
  and $w\times u=v$.
\end{lemma}

\begin{proof}
  This follows immediately from equation~\eqref{eq:AREA}
  in \autoref{le:area}.
\end{proof}

\begin{proof}[Proof of \autoref{thm:imO}~\itref{It_Transitive}]
  \label{proof:imO2}
  Let $\phi_0\co \R^7\times\R^7\times\R^7\to\R$ be the $3$--form in
  Example~\ref{ex:cross7} and let $\phi\in\Lambda^3V^*$ be a
  nondegenerate $3$--form. %
  Let $V$ be equipped with the compatible inner product of
  \autoref{thm:imO} and denote by
  $V\times V\to V\co (u,v)\mapsto u\times v$ the associated cross
  product. %
  Let $e_1,e_2\in V$ be orthonormal and define
  \begin{equation*}
    e_3 := e_1\times e_2.
  \end{equation*}
  Let $e_4\in V$ be any unit vector orthogonal to $e_1,e_2,e_3$ and define 
  \begin{equation*}
    e_5:= - e_1\times e_4.
  \end{equation*}
  Then $e_5$ has norm one and is orthogonal to $e_1,e_2,e_3,e_4$.  For
  $e_1$ and $e_4$ this follows from the definition
  and~\eqref{eq:frobenius}. %
  For $e_3$ we observe
  \begin{equation*}
    {\inner{e_3}{e_5}
    =
      -\inner{e_1\times e_2}{e_1\times e_4}
    =
      \inner{e_2}{e_1\times(e_1\times e_4)}
    =
      -\inner{e_2}{e_4}=0}.
  \end{equation*}
  Here the last but one equation follows from \autoref{le:area}. %
  For $e_2$ the argument is similar; since $e_2=e_3\times e_1$, by
  \autoref{le:uvw}, and $\inner{e_3}{e_4}=0$, we obtain
  $\inner{e_2}{e_5}=0$. %
  Now let $e_6$ be a unit vector orthogonal to $e_1,\dots,e_5$ and
  define
  \begin{equation*}
    e_7:= - e_1\times e_6.
  \end{equation*}
  As before we have that $e_7$ has norm one and is orthogonal to
  $e_1,\dots,e_6$. %
  Thus the vectors $e_1,\dots,e_7$ form an orthonormal basis of $V$
  and it follows from \autoref{le:uvw} that they satisfy the same
  relations as the standard basis of $\R^7$ in
  Example~\ref{ex:cross7}. %
  Hence, the map
  \begin{equation*}
    \R^7\stackrel{g}{\longrightarrow} V\co
    x=(x_1,\dots,x_7)\mapsto \sum_{i=1}^7x_ie_i
  \end{equation*}
  is a Hilbert space isometry and it satisfies $g^*\phi=\phi_0$.  This
  proves \autoref{thm:imO} (and the last assertion of
  \autoref{thm:cross}).
\end{proof}



\section{The associator and coassociator brackets}
\label{sec:assoc}

We assume throughout that $V$ is a $7$--dimensional real Hilbert
space, that ${\phi\in\Lambda^3V^*}$ is a nondegenerate $3$--form
compatible with the inner product, and~\eqref{eq:cross} is the cross
product given by~\eqref{eq:phi}. %
It follows from~\eqref{eq:AREA} that the expression
$(u\times v)\times w$ is alternating on any triple of pairwise
orthogonal vectors $u,v,w\in V$. %
Hence, it extends uniquely to an alternating $3$--form
$
V^3\to V\co (u,v,w)\mapsto[u,v,w]
$
called the \defined{associator bracket}. %
An explicit formula for this $3$--form is
\begin{equation}
  \label{eq:associator}
  [u,v,w] := (u\times v)\times w + \inner{v}{w}u-\inner{u}{w}v.
\end{equation}
The associator bracket can also be expressed in the form
\begin{equation}
  \label{eq:associator1}
  [u,v,w]
  =
  \frac13\Bigl(
    (u\times v)\times w + (v\times w)\times u + (w\times u)\times v
  \Bigr).
\end{equation}

\begin{remark}
  \label{rmk:associator}
  If $V$ is any Hilbert space with a skew-symmetric bilinear
  form~\eqref{eq:cross}, then the associator
  bracket~\eqref{eq:associator} is alternating iff~\eqref{eq:AREA}
  holds. %
  Indeed, skew-symmetry of the associator bracket in the first two
  arguments is obvious, and the identity
  \begin{align*}
    [u,v,w]+[u,w,v]
    &=
      w\times(v\times u) + v\times(w\times u) \\ &\quad
      - \inner{u}{w}v - \inner{u}{v}w + 2\inner{v}{w}u
  \end{align*}
  shows that skew-symmetry in the last two arguments is equivalent
  to~\eqref{eq:AREA}. %
  By \autoref{le:cross3}, the associator bracket vanishes in dimension
  three.
\end{remark}

The square of the volume of the $3$--dimensional parallelepiped
spanned by $u,v,w\in V$ will be denoted by
\begin{equation*}
  \abs{u\wedge v\wedge w}^2
  :=
  \det
  \begin{pmatrix}
    \abs{u}^2 & \inner{u}{v} & \inner{u}{w} \\
    \inner{v}{u} & \abs{v}^2 & \inner{v}{w} \\
    \inner{w}{u} & \inner{w}{v} & \abs{w}^2
  \end{pmatrix}.
\end{equation*}

\begin{lemma}
  \label{le:assocphi}
  For all $u,v,w\in V$ we have
  \begin{equation}
    \label{eq:assocphi}
    \phi(u,v,w)^2 + \abs{[u,v,w]}^2 = \abs{u\wedge v\wedge w}^2.
  \end{equation}
\end{lemma}

\begin{proof}
  If $w$ is orthogonal to $u$ and $v$, then we have 
  \begin{align*}
    \abs{[u,v,w]}^2
    &=
      \abs{(u\times v)\times w}^2 \\
    &=
      \abs{u\times v}^2\abs{w}^2 - \inner{u}{v\times w}^2 \\
    &= 
      \abs{u\wedge v\wedge w}^2 - \phi(u,v,w)^2.
  \end{align*}
  Here the first equation follows from the definition of the
  associator bracket and orthogonality, the second equation follows
  from~\eqref{eq:area}, and the last equation follows
  from~\eqref{eq:area} and orthogonality, as well as~\eqref{eq:phi}. %
  The general case can be reduced to the orthogonal case by
  Gram--Schmidt.
\end{proof}

\begin{definition}
  \label{def:assoc}
  A $3$--dimensional subspace $\Lambda\subset V$ is called
  \defined{associative} the associator bracket vanishes on $\Lambda$,
  i.e., 
  \begin{equation*}
    [u,v,w]=0\qquad
    \mbox{for all }u,v,w\in\Lambda.
  \end{equation*} 
\end{definition}

\begin{lemma}
  \label{le:assoc}
  Let $\Lambda\subset V$ be a $3$--dimensional linear subspace. %
  Then the following are equivalent:
  \begin{enumerate}[(i)]
  \item
    \label{It_Assoc1}
    $\Lambda$ is associative.

  \item
    \label{It_Assoc2}
    If $u,v,w$ is an orthonormal basis of $\Lambda$, then
    $\phi(u,v,w)=\pm1$.

  \item
    \label{It_Assoc3}
    If $u, v\in\Lambda$, then $u\times v\in\Lambda$.

  \item
    \label{It_Assoc4}
    If $u\in\Lambda^\perp$ and $v\in\Lambda$,
    then $u\times v\in\Lambda^\perp$.

  \item
    \label{It_Assoc5}
    If $u,v\in\Lambda^\perp$, then $u\times v\in\Lambda$.
  \end{enumerate}
  Moreover, if $u,v\in V$ are linearly independent,
  then the subspace spanned by the vectors $u,v,u\times v$ is associative.
\end{lemma}

\begin{proof}
  That~\itref{It_Assoc1} is equivalent to~\itref{It_Assoc2} follows from \autoref{le:assocphi}.

  We prove that~\itref{It_Assoc1} is equivalent
  to~\itref{It_Assoc3}. %
  That the associator bracket vanishes on a $3$--dimensional subspace
  that is invariant under the cross product follows from
  \autoref{le:cross3}~\itref{It_BACCAB}. %
  Conversely suppose that the associator bracket vanishes on
  $\Lambda$. %
  Let $u,v\in\Lambda$ be linearly independent and let $w\in\Lambda$ be
  a nonzero vector orthogonal to $u$ and $v$. %
  Then, by \autoref{le:assocphi}, we have
  \begin{equation*}
    \inner{u\times v}{w}^2
    =\phi(u,v,w)^2
    =\abs{u\wedge v\wedge w}^2
    =\abs{u\times v}^2\abs{w}^2
  \end{equation*}
  and hence $u\times v$ is a real multipe of $w$. %
  Thus $u\times v\in\Lambda$. 

  We prove that~\itref{It_Assoc3} is equivalent
  to~\itref{It_Assoc4}. %
  First assume~\itref{It_Assoc3} and let $u\in\Lambda$,
  $v\in\Lambda^\perp$. %
  Then, by~\itref{It_Assoc3}, we have $w\times u\in\Lambda$ for every
  $w\in\Lambda$. %
  Hence, $\inner{w}{u\times v}=\inner{w\times u}{v}=0$ for every
  $w\in\Lambda$ and so $u\times v\in\Lambda^\perp$. %
  Conversely assume~\itref{It_Assoc4} and let $u,v\in\Lambda$. %
  Then, by~\itref{It_Assoc3}, we have $w\times u\in\Lambda^\perp$ for
  every $w\in\Lambda^\perp$. %
  Hence, $\inner{w}{u\times v}=\inner{w\times u}{v}=0$ for every
  $w\in\Lambda^\perp$. %
  This implies $u\times v\in\Lambda$. %
  Thus we have proved that~\itref{It_Assoc3} is equivalent
  to~\itref{It_Assoc4}.

  We prove that~\itref{It_Assoc4} is equivalent
  to~\itref{It_Assoc5}. %
  Fix a unit vector $u\in\Lambda^\perp$ and define the endomorphism
  $J\co u^\perp\to u^\perp$ by $Jv:=u\times v$. %
  By \autoref{le:area} this is an isomorphism with inverse $-J$.
  Condition~\itref{It_Assoc4} asserts that $J$ maps $\Lambda$ to
  $\Lambda^\perp\cap u^\perp$ while condition~\itref{It_Assoc5}
  asserts that $J$ maps $\Lambda^\perp\cap u^\perp$ to $\Lambda$. %
  Since both are $3$--dimensional subspaces of $u^\perp$, these two
  assertions are equivalent. %
  This proves that~\itref{It_Assoc4} is equivalent
  to~\itref{It_Assoc5}.

  If $u$ and $v$ are linearly independent, then $u\times v\ne 0$,
  by~\eqref{eq:area}, and ${u\times v}$ is orthogonal to $u$ and $v$,
  by~\eqref{eq:orthogonal}. %
  Hence, the subspace~$\Lambda$ spanned by ${u,v,u\times v}$ is
  $3$--dimensional. %
  That it is invariant under the cross product follows from
  assertion~\itref{It_Assoc4} in \autoref{le:area}. %
  Hence, $\Lambda$ is associative, and this proves \autoref{le:assoc}.
\end{proof}

\begin{lemma}
  \label{le:psi}
  The map $\psi\co V^4\to\R$ defined by
  \begin{equation}
    \label{eq:psi}
    \begin{split}
      \psi(u,v,w,x) 
      &:=
        \inner{[u,v,w]}{x} \\
      &\;=
        \frac13 \Bigl(
          \phi(u\times v,w,x)+\phi(v\times w,u,x)+\phi(w\times u,v,x)
        \Bigr)
    \end{split}
  \end{equation}
  is an alternating $4$--form (the \defined{coassociative calibration}
  of $(V,\phi)$). %
  Moreover, it is given by ${\psi=*\phi}$, where
  $*\co \Lambda^kV^*\to\Lambda^{7-k}V^*$ denotes the Hodge
  $*$--operator associated to the inner product and the orientation in
  \autoref{le:imO}.
\end{lemma}

\begin{proof}
  See page~\pageref{proof:psi}.
\end{proof}

\begin{remark}
  \label{rmk:assbrack}
  By \autoref{le:assoc} and \autoref{le:psi} the associator bracket $[u,v,w]$ is 
  orthogonal to the vectors $u,v,w,v\times w,w\times u,u\times v$. %
  Second, these six vectors are linearly independent if only if $[u,v,w]\ne0$. %
  (Make them pairwise orthogonal by adding to $v$ a real multiple 
  of $u$ and to $w$ a linear combination of ${u,v,u\times v}$. %
  Then their span and $[u,v,w]$ remain unchanged.) %
  Third, if $[u,v,w]\ne0$ then the vectors 
  $u,v,w,v\times w,w\times u,u\times v,[u,v,w]$ form a positive basis of $V$.
\end{remark}

\begin{remark}
  \label{rmk:psi}
  The standard associative calibration on $\R^7$ is
  \begin{equation}
    \label{eq:phi0}
    \phi_0 
    = e^{123}- e^{145} - e^{167} - e^{246} + e^{257} - e^{347} - e^{356}
  \end{equation}
  (see \autoref{ex:cross7}). %
  The corresponding coassociative calibration is
  \begin{equation}
    \label{eq:psi0}
    \psi_0
    = - e^{1247} - e^{1256} + e^{1346} - e^{1357} 
      - e^{2345} - e^{2367} + e^{4567}.
    \end{equation}
\end{remark}

\begin{remark}
  \label{rmk:star}
  Let $V\to V^*\co u\mapsto u^*:=\inner{u}{\cdot}$ be the 
  isomorphism induced by the inner product. %
  Then, for 
  $\alpha\in\Lambda^kV^*$ and $u\in V$, we have 
  \begin{equation}
    \label{eq:star}
    *\iota(u)\alpha = (-1)^{k-1}u^*\wedge *\alpha.
  \end{equation}
  This holds on any finite dimensional oriented Hilbert space.
\end{remark}

\begin{remark}
  \label{rmk:uphipsi}
  Throughout we use the notation 
  \begin{equation}
    \label{eq:cLA}
    (\cL_A\alpha)(v_1,\dots,v_k)
    := \alpha(Av_1,v_2,\dots,v_k)
    + \cdots + \alpha(v_1,\dots,v_{k-1},Av_k)
  \end{equation}
  for the infinitesimal action of $A\in\End(V)$ on a $k$--form
  ${\alpha\in\Lambda^kV^*}$. %
  For ${u\in V}$ denote by ${A_u\in\so(V)}$ the skew-adjoint
  endomorphism ${A_uv:=u\times v}$. %
  Then equation~\eqref{eq:psi} can be expressed in the form
  \begin{equation}
    \label{eq:uphi}
    \cL_{A_u}\phi = 3\iota(u)\psi.
  \end{equation}
  Since $\psi=*\phi$, we have $\cL_A\psi=*\cL_A\phi$ for all
  $A\in\so(V)$. %
  Hence, it follows from equation~\eqref{eq:star} that
  \begin{equation}
    \label{eq:upsi}
    \cL_{A_u}\psi 
    = *(3\iota(u)\psi)
    = -3u^*\wedge\phi.
  \end{equation}
\end{remark}

\begin{proof}[Proof of \autoref{le:psi}]
  \label{proof:psi}
  It follows from \autoref{rmk:associator} that $\psi$ is alternating
  in the first three arguments. %
  To prove that $\psi\in\Lambda^4V^*$ we compute
  \begin{equation}
    \label{eq:psicross}
    \begin{split}
      \psi(u,v,w,x)
      &=
        \inner{(u\times v)\times w + \inner{v}{w}u - \inner{u}{w}v}{x} \\
      &=
        \inner{u\times v}{w\times x}
        + \inner{v}{w}\inner{u}{x}
        - \inner{u}{w}\inner{v}{x}. 
    \end{split}
  \end{equation}
  Here the first equation follows from the definition of $\psi$
  in~\eqref{eq:psi} and the definition of the associator bracket
  in~\eqref{eq:associator}. %
  Swapping $x$ and $w$ as well as $u$ and~$v$ in~\eqref{eq:psicross}
  gives the same expression. %
  Thus
  \begin{equation*}
    \psi(u,v,w,x) = \psi(v,u,x,w) = -\psi(u,v,x,w).
  \end{equation*}
  This shows that $\psi\in\Lambda^4V^*$ as claimed. %
  To prove the second assertion we observe the following.

  \begin{claim}
    If $u,v,w,x$ are orthonormal and $u\times v = w\times x$, then
    $\psi(u,v,w,x)=1$.
  \end{claim}

  This 
  follows directly from the definition of $\psi$ and of the
  associator bracket in~\eqref{eq:associator} and~\eqref{eq:psi}. %
  Now, by \autoref{thm:imO}, we can restrict attention to the standard
  structures on $\R^7$.  %
  Thus $\phi=\phi_0$ is given by~\eqref{eq:phi0} and this $3$--form is
  compatible with the standard inner product on $\R^7$. %
  We have the product rule $e_i\times e_j=e_k$ whenever the term
  $e^{ijk}$ or one of its cyclic permutations shows up in this sum,
  and the claim shows that we have a summand $\eps e^{ijk\ell}$ in
  $\psi=\psi_0$ whenever $e_i\times e_j=\eps e_k\times e_\ell$ with
  $\eps\in\{\pm1\}$. %
  Hence, $\psi_0$ is given by~\eqref{eq:psi0}. %
  Term by term inspection shows that $\psi_0=*\phi_0$. %
  This proves \autoref{le:psi}.
\end{proof}

\begin{lemma}
  \label{le:coass}
  For all $u,v,w,x\in V$ we have
  \begin{equation}
    \label{eq:coass}
    \begin{split}
      &[u,v,w,x] 
      := \phi(u,v,w)x-\phi(x,u,v)w+\phi(w,x,u)v-\phi(v,w,x)u \\
      &= \frac13\bigl(
      - [u,v,w]\times x + [x,u,v]\times w 
      - [w,x,u]\times v + [v,w,x]\times u\Bigr).
    \end{split}
  \end{equation}
  The resulting multi-linear map 
  \begin{equation*}
    V^4\to V\co (u,v,w,x)\mapsto[u,v,w,x]
  \end{equation*}
  is alternating and is called the \defined{coassociator bracket} on $V$.
\end{lemma}

\begin{proof}
  Define the alternating multi-linear map $\tau\co V^4\to V$ by 
  \begin{align*}
    \tau(u,v,w,x)
    &:=
      3\Bigl(\phi(u,v,w)x-\phi(x,u,v)w+\phi(w,x,u)v-\phi(v,w,x)u\Bigr) \\ &\quad
      + [u,v,w]\times x - [x,u,v]\times w + [w,x,u]\times v - [v,w,x]\times u.
  \end{align*}
  We must prove that $\tau$ vanishes. %
  The proof has three steps.

  \setcounter{step}{0}
  \begin{step}
    \label{Pf_Coass1}
    $\tau(u,v,w,x)$ is orthogonal to $u,v,w,x$ for all $u,v,w,x\in V$.
  \end{step}

  It suffices to assume that $u,v,w,x$ are pairwise orthogonal. %
  Then we have $[u,v,w]=(u\times v)\times w$ and similarly for
  $[x,v,w]$ etc. %
  Hence,
  \begin{align*}
    \inner{\tau(u,v,w,x)}{x}
    &=
      3\abs{x}^2\phi(u,v,w) - \inner{[u,v,x]}{w\times x} \\ &\quad
      -\, \inner{[w,u,x]}{v\times x} - \inner{[v,w,x]}{u\times x} \\
    &=
      3\abs{x}^2\phi(u,v,w) - \inner{(u\times v)\times x}{w\times x} \\ &\quad
      -\, \inner{(w\times u)\times x}{v\times x} 
      - \inner{(v\times w)\times x}{u\times x} \\
    &=
      0.
  \end{align*} 
  Here the last step uses the identity~\eqref{eq:frobenius} and the
  fact that $x\times (u\times x)=\abs{x}^2u$ whenever $u$ is
  orthogonal to $x$. %
  Thus $\tau(u,v,w,x)$ is orthogonal to $x$. %
  Since $\tau$ is alternating, this proves \autoref{Pf_Coass1}.

  \begin{step}
    \label{Pf_Coass2}
    $\tau(u,v,u\times v,x)=0$ for all $u,v,x\in V$.
  \end{step}

  It suffices to assume that $u,v$ are orthonormal and that $x$ is
  orthogonal to $u$, $v$, and $w:=u\times v$. %
  Then $v\times w=u$, $w\times u=v$, $\phi(u,v,w)=1$, and
  $\phi(x,v,w)=\phi(x,w,u)=\phi(x,u,v)=0$. %
  Moreover, $[u,v,w]=0$ and
  \begin{equation*}
    [x,v,w]
    = [v,w,x]
    = (v\times w)\times x 
    = u\times x,
      \qquad
    [x,v,w]\times u
    = x,
  \end{equation*}
  and similarly $[x,w,u]\times v=[x,u,v]\times w=x$. %
  This implies that $\tau(u,v,w,x)=0$.

  \begin{step}
    \label{Pf_Coass3}
    $\tau(u,v,w,x)=0$ for all $u,v,w,x\in V$.
  \end{step}

  By the alternating property we may assume that $u$ and $v$ are
  orthonormal. %
  Using the alternating property again and \autoref{Pf_Coass2} we may
  assume that $w$ is a unit vector orthogonal to $u,v,u\times v$ and
  that $x$ is a unit vector orthogonal to $u,v,w$ and
  $v\times w,w\times u,u\times v$. %
  This implies that
  \begin{equation*}
    \phi(u,v,w)
    = \phi(x,v,w)
    = \phi(x,w,u)
    = \phi(x,u,v)
    = 0.
  \end{equation*}
  Hence, the vectors $x\times u,x\times v,x\times w$ form a basis of
  the orthogonal complement of the space spanned by $u,v,w,x$. %
  Each of these vectors is orthogonal to $\tau(u,v,w,x)$ and hence
  $\tau(u,v,w,x)=0$ by \autoref{Pf_Coass1}. %
  This proves \autoref{le:coass}. %
\end{proof}

The square of the volume of the 
$4$--dimensional parallelepiped spanned by $u,v,w,x\in V$ 
will be denoted by 
\begin{equation*}
  \abs{u\wedge v\wedge w\wedge x}^2
   :=
   \det
   \begin{pmatrix}
     \abs{u}^2 & \inner{u}{v} & \inner{u}{w} & \inner{u}{x} \\
     \inner{v}{u} & \abs{v}^2 & \inner{v}{w} & \inner{v}{x} \\
     \inner{w}{u} & \inner{w}{v} & \abs{w}^2 & \inner{w}{x} \\
     \inner{x}{u} & \inner{x}{v} & \inner{x}{w} & \abs{x}^2
   \end{pmatrix}.
\end{equation*}

\begin{lemma}
  \label{le:coassocpsi}
  For all $u,v,w,x\in V$ we have 
  \begin{equation}
    \label{eq:coassocpsi}
    \psi(u,v,w,x)^2 + \abs{[u,v,w,x]}^2 = \abs{u\wedge v\wedge w\wedge x}^2.
  \end{equation}
\end{lemma}

\begin{proof}
  The proof has four steps.
  
  \setcounter{step}{0}
  \begin{step}
    \label{Pf_CoassocPsi1}
    If $u,v,w,x$ are orthogonal, then
    \begin{align*}
      \psi(u,v,w,x)^2
      &=
        \inner{u\times v}{w\times x}^2, \\
      \abs{[u,v,w,x]}^2
      &=
        \inner{u\times v}{w}^2\abs{x}^2 
        + \inner{u\times v}{x}^2\abs{w}^2 \\ &\quad
        + \inner{u}{w\times x}^2\abs{v}^2 
        + \inner{v}{w\times x}^2\abs{u}^2, \\
      \abs{u\wedge v\wedge w\wedge x}^2
      &=
        \abs{u}^2\abs{v}^2\abs{w}^2\abs{x}^2.
    \end{align*}
  \end{step}

  The first equation follows from~\eqref{eq:associator} 
  and~\eqref{eq:psi}, using~\eqref{eq:frobenius}. %
  The other two equations follow immediately 
  from the definitions.

  \begin{step}
    \label{Pf_CoassocPsi2}
    Equation~\eqref{eq:coassocpsi} holds when $u,v,w,x$ are orthogonal
    and, in addition, $w$ and $x$ are orthogonal to $u\times v$.
  \end{step}

  Since $[u,v,w] \neq 0$, it follows from the assumptions and
  \autoref{le:assoc} that $w\times x$ is a linear combination of the
  vectors $u$, $v$, $u\times v$. %
  Hence, the assertion follows from \autoref{Pf_CoassocPsi1}.

  \begin{step}
    \label{Pf_CoassocPsi3}
    Equation~\eqref{eq:coassocpsi} holds when $u,v,w,x$ are orthogonal
  \end{step}

  Suppose, in addition, that $w$ and $x$ are orthogonal to $u\times v$
  and replace $x$ by $x_\lambda:=x+\lambda u\times v$ for
  $\lambda\in\R$. %
  Then $\psi(u,v,w,x_\lambda)$ is independent of $\lambda$ and
  \begin{equation*}
    \abs{[u,v,w,x_\lambda]}^2
    =\abs{[u,v,w,x]}^2 +\lambda^2\abs{u}^2\abs{v}^2\abs{w}^2\abs{u\times v}^2.
  \end{equation*}
  Hence, it follows from \autoref{Pf_CoassocPsi2}
  that~\eqref{eq:coassocpsi} holds when $u,v,w,x$ are orthogonal and,
  in addition, $w$ is orthogonal to $u\times v$. %
  This condition can be achieved by rotating the pair $(w,x)$. %
  This proves \autoref{Pf_CoassocPsi3}.

  \begin{step}
    \label{Pf_CoassocPsi4}
    Equation~\eqref{eq:coassocpsi} holds always.
  \end{step}

  The general case follows from the orthogonal case via Gram--Schmidt,
  because both sides of equation~\eqref{eq:coassocpsi} remain
  unchanged if we add to any of the four vectors a multiple of any of
  the other three. %
  This proves the lemma.
\end{proof}

\begin{definition}
  \label{def:coassoc}
  A $4$--dimensional subspace $H\subset V$ is called 
  \defined{coassociative} if 
  \begin{equation*}
    [u,v,w,x]=0\qquad\mbox{for all }u,v,w,x\in H.
  \end{equation*}
\end{definition}

\begin{lemma}
  \label{le:coassoc}
  Let $H\subset V$ be a $4$--dimensional linear subspace. %
  Then the following are equivalent:
  \begin{enumerate}[(i)]
  \item
    \label{It_Coassoc1}
    $H$ is coassociative.

  \item
    \label{It_Coassoc2}
    If $u,v,w,x$ is an orthonormal basis of $H$, then $\psi(u,v,w,x)=\pm1$.

  \item
    \label{It_Coassoc3}
    For all $u,v,w\in H$ we have $\phi(u,v,w)=0$.  

  \item
    \label{It_Coassoc4}
    If $u,v\in H$, then $u\times v\in H^\perp$.

  \item
    \label{It_Coassoc5}
    If $u\in H$ and $v\in H^\perp$, then $u\times v\in H$.

  \item
    \label{It_Coassoc6}
    If $u,v\in H^\perp$, then $u\times v\in H^\perp$.

  \item
    \label{It_Coassoc7}
    The orthogonal complement $H^\perp$ is associative. 
  \end{enumerate}
\end{lemma}

\begin{proof}
  That~\itref{It_Coassoc1} is equivalent to~\itref{It_Coassoc2}
  follows from \autoref{le:coassocpsi}.

  We prove that~\itref{It_Coassoc1} is equivalent
  to~\itref{It_Coassoc3}. %
  That~\itref{It_Coassoc3} implies~\itref{It_Coassoc1} is obvious by
  definition of the coassociator bracket in~\eqref{eq:coass}. %
  Conversely, assume~\itref{It_Coassoc1} and choose a basis $u,v,w,x$
  of $H$. %
  Then $[u,v,w,x]=0$ and hence, by~\eqref{eq:coass}, we have
  $\phi(u,v,w)=\phi(x,v,w)=\phi(x,w,u)=\phi(x,u,v)=0$. %
  This implies~\itref{It_Coassoc3}.

  We prove that~\itref{It_Coassoc3} is equivalent
  to~\itref{It_Coassoc4}. %
  If~\itref{It_Coassoc3} holds and $u,v\in H$, then
  $\inner{u\times v}{w}=\phi(u,v,w)=0$ for every $w\in H$ and hence
  $u\times v\in H^\perp$. %
  Conversely, if~\itref{It_Coassoc4} holds and $u,v\in H$, then
  $u\times v\in H^\perp$ and hence
  $\phi(u,v,w)=\inner{u\times v}{w}=0$ for all $w\in H$.

  Thus we have proved that~\itref{It_Coassoc1}, \itref{It_Coassoc2},
  \itref{It_Coassoc3}, \itref{It_Coassoc4} are equivalent. %
  That assertions~\itref{It_Coassoc4}, \itref{It_Coassoc5},
  \itref{It_Coassoc6}, \itref{It_Coassoc7} are equivalent was proved
  in \autoref{le:assoc}.
\end{proof}

\begin{remark}\label{rmk:LaH}
  Let $V$ be a $7$--dimensional real Hilbert space equipped 
  with a cross product and denote the associative and 
  coassociative calibrations by $\phi$ and $\psi$. %
  Let ${\Lambda\subset V}$ be an  associative subspace 
  and define~${H:=\Lambda^\perp}$. %
  Orient $\Lambda$ and $H$ by the volume forms 
  $\dvol_\Lambda:=\phi|_\Lambda$ and $\dvol_H:=\psi|_H$. %
  A \defined{standard basis} of the space~$\Lambda^+H^*$ of self-dual 
  $2$--forms on $H$ is a triple $\om_1,\om_2,\om_3\in\Lambda^+H^*$
  that satisfies the condition $\om_i\wedge\om_j = 2\delta_{ij}\dvol_H$ 
  for all $i$ and $j$. %
  In this situation the map
  \begin{equation}\label{eq:LaH1}
    \Lambda\to\Lambda^+H^*\co u\mapsto -\iota(u)\phi|_H
  \end{equation}
  is an orientation preserving isomorphism that sends every orthonormal 
  basis of~$\Lambda$ to a standard basis of $\Lambda^+H^*$. %
  (To see this, choose a standard basis of $V$ as in 
  \autoref{rmk:psi} with $\Lambda=\mathrm{span}\{e_1,e_2,e_3\}$.) %
  Let ${\pi_\Lambda:V\to\Lambda}$ and ${\pi_H:V\to H}$ be
  the orthogonal projections. Let $u_1,u_2,u_3$ be any
  orthonormal basis of $\Lambda$ and define 
  $\alpha_i:=u_i^*|_\Lambda$ and $\om_i:=-\iota(u_i)\phi|_H$
  for $i=1,2,3$. %
  Then the associative calibration $\phi$ 
  can be expressed in the form
  \begin{equation}\label{eq:LaH2}
    \phi = \pi_\Lambda^*\dvol_\Lambda
    -  \pi_\Lambda^*\alpha_1\wedge\pi_H^*\om_1
    -  \pi_\Lambda^*\alpha_2\wedge\pi_H^*\om_2
    -  \pi_\Lambda^*\alpha_3\wedge\pi_H^*\om_3.
  \end{equation}
\end{remark}

The next theorem characterizes a nondegenerate $3$--form $\phi$ 
in terms of its coassociative calibration $\psi$ in \autoref{le:psi}.

\begin{theorem}
  \label{thm:phipsi}
  Let $V$ be a $7$--dimensional vector space over the reals, 
  let $\phi,\phi'\in\Lambda^3V^*$ be nondegenerate $3$--forms,
  and let $\psi,\psi'\in\Lambda^4V^*$ be their coassociative
  calibrations. %
  Then the following are equivalent:
  \begin{enumerate}[(i)]
  \item
    \label{It_Phipsi1}
    $\phi'=\phi$ or $\phi'=-\phi$.

  \item
    \label{It_Phipsi2}
    $\psi'=\psi$.
  \end{enumerate}
\end{theorem}

\begin{proof}
  That~\itref{It_Phipsi1} implies~\itref{It_Phipsi2} follows from the
  definition of $\psi$ in \autoref{le:psi} and the fact that reversing
  the sign of $\phi$ also reverses the sign of the cross product and
  thus leaves $\psi$ unchanged (see equation~\eqref{eq:psi}). %
  To prove the converse assume that $\psi'=\psi$ and denote by
  $\inner{\cdot}{\cdot}'$ the inner product determined by~$\phi'$,
  by~$\times'$ the cross product determined by~$\phi'$, and
  by~$[\cdot,\cdot,\cdot]'$ the associator bracket determined
  by~$\phi'$. %
  We prove in four steps that~${\phi'=\pm\phi}$.

  \setcounter{step}{0}
  \begin{step}
    \label{Pf_PhiPsi1}
    A $3$--dimensional subspace $\Lambda\subset V$
    is associative for $\phi$ if and only if it is associative for $\phi'$.
  \end{step}

  Let $\Lambda\subset V$ be a three-dimensional linear subspace. %
  By \autoref{def:assoc} it is associative for $\phi$ if and only if
  $[u,v,w]=0$ for all $u,v,w\in\Lambda$. %
  By \autoref{le:psi} this is equivalent to the condition that the
  linear functional $\psi(u,v,w,\cdot)$ on~$V$ vanishes for all
  $u,v,w\in\Lambda$. %
  Since $\psi=\psi'$, this proves \autoref{Pf_PhiPsi1}.

  \begin{step}
    \label{Pf_PhiPsi2}
    There is a linear functional $\alpha\co V\to\R$ and a  
    $c\in\R\setminus\{0\}$ such that
    \begin{equation*}
      u\times'v = \alpha(u)v-\alpha(v)u+cu\times v
    \end{equation*}
    for all $u,v\in V$.
  \end{step}

  Fix two linearly independent vectors $u,v\in V$. %
  Then the vectors $u,v,u\times v$ span a $\phi$--associative subspace
  $\Lambda\subset V$ by \autoref{le:assoc}. %
  The subspace $\Lambda$ is also $\phi'$--associative by
  \autoref{Pf_PhiPsi1}. %
  Hence, $u\times'v\in\Lambda$ by \autoref{le:assoc} and so there
  exist real numbers $\alpha(u,v),\beta(u,v),\gamma(u,v)$ such that
  \begin{equation}
    \label{eq:albega}
    u\times'v = \alpha(u,v)v+\beta(u,v)u+\gamma(u,v)u\times v.
  \end{equation}
  Since $u,v,u\times'v$ are linearly independent, it follows that
  $\gamma(u,v)\ne 0$ and the coefficients $\alpha,\beta,\gamma$ depend
  smoothly on $u$ and $v$. %
  Differentiate equation~\eqref{eq:albega} with respect to $v$ to
  obtain that $\alpha$ and $\gamma$ are locally independent of $v$. %
  Differentiate it with respect to $u$ to obtain that $\beta$ and
  $\gamma$ are locally independent of $u$.  %
  Since the set of pairs of linearly independent vectors in $V$ is
  connected, it follows that there exist functions
  $\alpha,\beta\co V\to\R$ and a constant $c\in\R\setminus\{0\}$ such
  that
  \begin{equation*}
    u\times'v = \alpha(u)v+\beta(v)u+cu\times v
  \end{equation*}
  for all pairs of linearly independent vectors $u,v\in V$. %
  Interchange $u$ and $v$ to obtain $\beta(v)=-\alpha(v)$ for all $v\in V$. %
  Since the function $V\to V\co u\mapsto u\times'v$ is linear 
  for all $v\in V$ it follows that $\alpha\co V\to\R$ is linear. %
  This proves \autoref{Pf_PhiPsi2}.

  \begin{step}
    \label{Pf_PhiPsi3}
    Let $\alpha$ and $c$ be as in \autoref{Pf_PhiPsi2}. %
    Then $\alpha=0$ and $\inner{u}{v}' = c^2\inner{u}{v}$
    for all $u,v\in V$.
  \end{step}

  Fix a vector $u\in V\setminus\{0\}$ and choose a vector $v\in V$ such that
  $u$ and $v$ are linearly independent. %
  Then $u\times(u\times v) = \inner{u}{v}u-\abs{u}^2v$ by \autoref{le:area}. %
  Hence, it follows from \autoref{Pf_PhiPsi2} that
  \begin{align*}
    \inner{u}{v}'u-{\abs{u}'}^2v
    &=
      u\times'(u\times' v) \\
    &=
      u\times'\left(\alpha(u)v-\alpha(v)u+cu\times v\right) \\
    &=
      \alpha(u)u\times'v + cu\times'(u\times v) \\
    &=
      \alpha(u)\bigl(\alpha(u)v-\alpha(v)u+cu\times v\bigr) \\ &\quad
      +\, c\bigl(\alpha(u)u\times v-\alpha(u\times v)u+cu\times(u\times v)\bigr) \\
    &=
      \alpha(u)\bigl(\alpha(u)v-\alpha(v)u+cu\times v\bigr) \\ &\quad
      +\, c\bigl(\alpha(u)u\times v-\alpha(u\times v)u
      + c\inner{u}{v}u-c\abs{u}^2v\bigr) \\
    &=
      \bigl(c^2\inner{u}{v}-c\alpha(u\times v)-\alpha(u)\alpha(v)\bigr)u \\ &\quad
      +\,\bigl(\alpha(u)^2-c^2\abs{u}^2\bigr)v +  2c\alpha(u)u\times v.
  \end{align*}
  Since $u$, $v$, and $u\times v$ are linearly independent,
  it follows that 
  \begin{equation*}
    \alpha(u)=0,\qquad  {\abs{u}'}^2 = c^2\abs{u}^2 - \alpha(u)^2.
  \end{equation*}
  Since $u\in V\setminus\{0\}$ was chosen arbitrarily, 
  it follows that $\alpha(u)=0$ and 
  \begin{equation*}
    \inner{u}{v}'=c^2\inner{u}{v},\qquad u\times'v=cu\times v
  \end{equation*}
  for all $u,v\in V$. %
  This proves \autoref{Pf_PhiPsi3}.

  \begin{step}
    \label{Pf_PhiPsi4}
    $\phi'=\pm\phi$.
  \end{step}

  It follows from \autoref{Pf_PhiPsi2} and \autoref{Pf_PhiPsi3} that
  \begin{equation*}
    \phi'(u,v,w) = \inner{u\times'v}{w}' = c^3\inner{u\times v}{w} = c^3\phi(u,v,w)
  \end{equation*}
  for all $u,v,w\in V$, and so $\psi=\psi'=c^4\psi$
  by equation~\eqref{eq:psi}. %
  Hence $c=\pm1$ and this proves \autoref{thm:phipsi}.
\end{proof}

The next theorem follows a suggestion by Donaldson for characterizing
coassociative calibrations in terms of their dual $3$--forms.

\begin{theorem}
  \label{thm:psi}
  Let $V$ be a $7$--dimensional vector space over the reals and let
  $\psi\in\Lambda^4V^*$. %
  Then the following are equivalent:
  \begin{enumerate}[(i)]
  \item
    \label{It_Psi1}
    There exists a nondegenerate $3$--form $\phi\in\Lambda^3V^*$ and a
    number $\eps=\pm1$ such that $\eps\psi$ is the coassociative
    calibration of $(V,\phi)$.

  \item
    \label{It_Psi2}
    If $\alpha,\beta\in V^*$ are linearly independent, then there
    exists a $1$--form $\gamma\in V^*$ such that
    $\alpha\wedge\beta\wedge\gamma\wedge\psi\ne0$.
  \end{enumerate}
\end{theorem}

\begin{proof}
  That~\itref{It_Psi1} implies~\itref{It_Psi2} follows 
  from equation~\eqref{eq:phipsi1} in \autoref{le:phipsi} below. %
  To prove the converse, assume~\itref{It_Psi2} 
  and fix any volume form $\sigma\in\Lambda^7V^*$. %
  Define the $3$--form $\Phi$ on the dual space $V^*$ by
  \begin{equation}
    \label{eq:Phipsi}
    \Phi(\alpha,\beta,\gamma)
    :=
    \frac{\alpha\wedge\beta\wedge\gamma\wedge\psi}{\sigma}
      \qquad\mbox{for }
    \alpha,\beta,\gamma\in V^*. 
  \end{equation}
  This $3$--form is nondegenerate by~\itref{It_Psi2}. %
  Denote the corresponding coassociative calibration by
  ${\Psi\co V^*\times V^*\times V^*\times V^*\to\R}$ and let
  $\inner{\cdot}{\cdot}_{V^*}$ be the inner product on~$V^*$
  determined by $\Phi$. %
  Let $\kappa\co V\to V^*$ be the isomorphism induced by this inner
  product, so $\alpha(u)=\inner{\alpha}{\kappa(u)}_{V^*}$ for
  $\alpha\in V^*$ and $u\in V$. %
  Let $\inner{\cdot}{\cdot}_V$ be the pullback under $\kappa$ of the
  inner product on $V^*$. %
  Then $ \phi := \kappa^*\Phi\in\Lambda^3V^* $ is~a nondegenerate
  $3$--form compatible with the inner product and the volume form
  \begin{equation*}
    \dvol := \tfrac{1}{7}\kappa^*\Phi\wedge\kappa^*\Psi.
  \end{equation*}
  By equation~\eqref{eq:Phipsi},
  \begin{equation}
    \label{eq:kappaphipsi}
    \phi(u,v,w)\sigma
    = \kappa(u)\wedge\kappa(v)\wedge\kappa(w)\wedge\psi.
  \end{equation}
  for all $u,v,w\in V$. %
  Choose $\lambda>0$ and $\eps=\pm1$ such that
  \begin{equation}
    \label{eq:dvolasi}
    \dvol
    = \eps\lambda^{-4/3}\sigma.
  \end{equation}
  Replace $\sigma$ by $\sigma_\lambda:=\lambda\sigma$
  in~\eqref{eq:Phipsi} to obtain $\Phi_\lambda=\lambda^{-1}\Phi$. %
  Its coassociative calibration is $\Psi_\lambda=\lambda^{-4/3}\Psi$,
  the inner product on $V^*$ induced by $\Phi_\lambda$ is
  $\inner{\cdot}{\cdot}_{V^*,\lambda}=\lambda^{-2/3}\inner{\cdot}{\cdot}_{V^*}$,
  and the isomorphism $\kappa_\lambda\co V\to V^*$ is
  $\kappa_\lambda=\lambda^{2/3}\kappa$. %
  Hence ,
  \begin{equation*}
    \phi_\lambda := \kappa_\lambda^*\Phi_\lambda = \lambda\phi,\qquad
    \psi_\lambda := \kappa_\lambda^*\Psi_\lambda = \lambda^{4/3}\kappa^*\Psi. 
  \end{equation*}
  By~\eqref{eq:dvolasi} this implies
  \begin{equation*}
    \dvol_\lambda
    :=
      \tfrac{1}{7}\phi_\lambda\wedge\psi_\lambda
    =
      \lambda^{7/3}\dvol
    =
      \eps\lambda\sigma=\eps\sigma_\lambda.
  \end{equation*}
  Multiply both sides in equation~\eqref{eq:kappaphipsi} by 
  $\eps\lambda^2$ to obtain
  \begin{equation*}
    \phi_\lambda(u,v,w)\eps\sigma_\lambda
    = \kappa_\lambda(u)\wedge\kappa_\lambda(v)\wedge\kappa_\lambda(w)\wedge\eps\psi
  \end{equation*}
  Since $\eps\sigma_\lambda=\dvol_\lambda$, 
  it follows from~\eqref{eq:phipsi1} below that the same equation 
  holds with~$\eps\psi$ replaced by~$\psi_\lambda$. %
  Thus $\eps\psi=\psi_\lambda$ is the associative calibration of $\phi_\lambda$. %
  (Here $\eps$ is independent of the choice of $\sigma$.) %
  This proves \autoref{thm:psi}. %
\end{proof}

\begin{remark}
  We can interpret \autoref{thm:psi} in the spirit of \autoref{Rmk_Hitchin}. %
  In the notation of \autoref{Rmk_Hitchin},
  if $V$ is an oriented $n$--dimensional vector space with an inner product $g$, 
  then the Hodge $*$--operator $*\co \Lambda^k V^* \to \Lambda^{n-k} V^*$ can be defined as
  \begin{equation*}
    *\alpha
    = (i_g^{-1})^*\alpha \otimes \vol_g \in \Lambda^k V \otimes \Lambda^n V^* = \Lambda^{n-k} V^*.
  \end{equation*}

  If $V$ is a $7$--dimensional vector space and $\psi \in \Lambda^4 V^*$,
  then we can equivalently think of it as a $3$--form $\phi^*$ on
  $V^*$ with values in $\Lambda^7 V^*$ since
  $\Lambda^4 V^* = \Lambda^3 V \otimes \Lambda^7 V^*$. %
  Define a symmetric bilinear form $G^* \co V^*\times V^*\to (\Lambda^7 V^*)^2$ by
  \begin{equation*}
    G^*(\alpha,\beta)
    :=
    \frac16 i(\alpha)\phi^* \wedge i(\beta)\phi^* \wedge \phi^*
    \qquad\mbox{for }\alpha,\beta\in V^*.
  \end{equation*}
  Condition~\itref{It_Psi2} in \autoref{thm:psi} is equivalent to
  $i_{G^*}\co V^* \to V \otimes (\Lambda^7 V^*)^2$ being an
  isomorphism. %
  Note that $\det i_{G^*} \in (\Lambda^7 V^*)^{12}$. %
  After picking an orientation we define a positive root
  $(\det i_{G^*})^{1/12} \in \Lambda^7 V^*$. %
  Define a volume form on $V$ and an inner product on $V^*$ by
  \begin{equation*}
    \vol_g := (\det(i_{G^*}))^{1/12}
    \qandq
    g^* := \frac{G^*}{\vol_g^2}.
  \end{equation*}
  A moment's thought shows that $\vol_g$ is the volume form associated
  with the dual inner product $g$ and the chosen orientation on $V$. %
  Further, the $3$--form 
  \begin{equation*}
  \phi := \frac{(i_g)^*\phi^*}{\vol_g} \in \Lambda^3 V^*
  \end{equation*}
  satisfies
  \begin{equation*}
    \frac16 i(u)\phi \wedge i(v)\phi \wedge \phi
    = g(u,v)\, \vol_g.
  \end{equation*}
  and $*\phi = \psi$.
\end{remark}

The next lemma summarizes some useful identities that will be needed
throughout. %
The first of these has already been used in the proof of
\autoref{thm:psi}. %
Assume that~$V$ is a $7$--dimensional oriented real Hilbert space
equipped with a compatible cross product, $\phi\in\Lambda^3V^*$ is the
associative calibration, and~${\psi:=*\phi\in\Lambda^4V^*}$ is the
coassociative calibration of~$(V,\phi)$.

\begin{lemma}
  \label{le:phipsi}
  The following hold for all $u,v,w,x\in V$ and all
  $\om\in\Lambda^2V^*$:
  \begin{align}
    \psi\wedge u^*\wedge v^*\wedge w^*
    &=
      \phi(u,v,w)\dvol,
      \label{eq:phipsi1}\\
    \phi\wedge u^*\wedge v^*\wedge w^*\wedge x^*
    &=
      \psi(u,v,w,x)\dvol,
      \label{eq:phipsi2}\\
    \iota(u)\psi\wedge v^*\wedge \iota(v)\psi
    &=
      0,
      \label{eq:phipsi3}\\
    *(\psi\wedge u^*)
    &=
      \iota(u)\phi,
      \label{eq:phipsi4}\\
    *(\phi\wedge u^*)
    &= 
      \iota(u)\psi,
      \label{eq:phipsi5}\\
    \abs{\iota(u)\phi}^2 
    &= 
      3\abs{u}^2,
      \label{eq:phipsi6}\\
    \abs{\iota(u)\psi}^2 
    &= 
      4\abs{u}^2,
      \label{eq:phipsi7}\\
    \phi\wedge\iota(u)\phi 
    &= 
      2\psi\wedge u^*,
      \label{eq:phipsi8}\\
    \phi\wedge\iota(u)\psi 
    &= 
      -4\iota(u)\dvol,
      \label{eq:phipsi9}\\
    \psi\wedge\iota(u)\phi 
    &= 
      3\iota(u)\dvol,
      \label{eq:phipsi10}\\
    \psi\wedge\iota(u)\psi
    &=
      0,
      \label{eq:phipsi11}\\
    *(\phi\wedge\iota(u)\phi)
    &=
      2\iota(u)\phi,
      \label{eq:phipsi12}\\
    *(\phi\wedge\iota(u)\psi)
    &=
      -4u^*,
      \label{eq:phipsi13}\\
    *(\psi\wedge\iota(u)\phi)
    &=
      3u^*,
      \label{eq:phipsi14}\\
    *(\psi\wedge*(\psi\wedge\iota(u)\phi))
    &= 
      3\iota(u)\phi,
      \label{eq:phipsi15}\\
    \iota(u)\phi\wedge*\iota(v)\phi 
    &= 
      3\inner{u}{v}\dvol,
      \label{eq:phipsi16}\\
    u^*\wedge v^* 
    &= 
      \iota(u\times v)\phi-\iota(v)\iota(u)\psi,
      \label{eq:phipsi17}\\
    u^*\wedge v^*\wedge*\iota(u\times v)\phi 
    &= 
      \abs{u\times v}^2\dvol,
      \label{eq:phipsi18}\\
    u^*\wedge v^*\wedge\iota(u)\phi\wedge\iota(v)\psi 
    &= 
      2\abs{u\times v}^2\dvol,
      \label{eq:phipsi19}\\
    \psi\wedge u^*\wedge v^*
    &= 
      \iota(u\times v)\dvol,
      \label{eq:phipsi20}\\
    \phi\wedge u^*\wedge v^*\wedge w^*
    &= 
      \iota([u,v,w])\dvol,
      \label{eq:phipsi21}\\
    \phi\wedge u^*\wedge v^*
    &= 
      *\iota(v)\iota(u)\psi,
      \label{eq:phipsi22}\\
    *(\psi\wedge *(\psi\wedge\om))
    &= 
      \om+*(\phi\wedge\om),
      \label{eq:phipsi23}\\
    *(\phi\wedge *(\phi\wedge\om))
    &= 
      2\om+*(\phi\wedge\om).
      \label{eq:phipsi24}
  \end{align}
\end{lemma}

\begin{proof}
  It is a general fact about alternating $k$--forms on a finite-dimensional
  Hilbert space $V$ that $\inner{u_1^*\wedge\cdots\wedge u_k^*}{\alpha}
  =\alpha(u_1,\dots,u_k)$
  for all $u_i\in V$ and all $\alpha\in\Lambda^kV^*$. %
  This proves~\eqref{eq:phipsi1} and~\eqref{eq:phipsi2}. %
  Equations~\eqref{eq:phipsi4} and~\eqref{eq:phipsi5} 
  follow from~\eqref{eq:star} in \autoref{rmk:star}. 

  To prove equations~\eqref{eq:phipsi3} and~\eqref{eq:phipsi6}--\eqref{eq:phipsi10}
  assume without loss of generality that $u,v$ are orthonormal.  %
  By \autoref{thm:imO} assume that $V=\R^7$
  with $u=e_1$ and $v=e_2$, and that $\phi$ and $\psi$ 
  are as in~\eqref{eq:phi0} and~\eqref{eq:psi0}, i.e.,
  \begin{equation}
    \label{eq:phipsi0}
    \begin{split}
      \phi
      &
      =\phi_0
      = e^{123} - e^{145} - e^{167} 
      - e^{246} + e^{257} - e^{347} - e^{356}, \\
      \psi
      &
      = \psi_0
      = - e^{1247} - e^{1256} + e^{1346} - e^{1357} 
      - e^{2345} - e^{2367} + e^{4567}.
    \end{split}
  \end{equation}
  Then 
  \begin{equation}
    \label{eq:uphipsi}
    \begin{split}
      \iota(u)\phi &=e^{23}-e^{45}-e^{67},\\
      \iota(u)\psi &= -e^{247}-e^{256}+e^{346}-e^{357}, \\
      v^*\wedge\iota(v)\psi &= -e^{1247}-e^{1256}+e^{2345}-e^{2367}.
    \end{split}
  \end{equation}
  Equation~\eqref{eq:phipsi3} follows by multipying the last two sums,
  and~\eqref{eq:phipsi6} and~\eqref{eq:phipsi7} follow by examining 
  the first two sums. %
  Moreover, by~\eqref{eq:phipsi0} and~\eqref{eq:uphipsi},
  \begin{equation*}
    \phi\wedge\iota(u)\phi 
    = -2e^{12345}-2e^{12367}+2e^{14567}
    = 2\,{*\iota(u)\phi} = 2u^*\wedge\psi.
  \end{equation*}
  This proves~\eqref{eq:phipsi8}. %
  By~\eqref{eq:phipsi0} and~\eqref{eq:uphipsi} we also have
  $\psi\wedge\iota(u)\phi = 3e^{234567}$
  and $\phi\wedge\iota(u)\psi = -4e^{234567}$. %
  This proves~\eqref{eq:phipsi9} and~\eqref{eq:phipsi10}.

  Equation~\eqref{eq:phipsi11} follows by contracting $u$ with the
  $8$--form $\psi\wedge\psi=0$. %
  Equations~\eqref{eq:phipsi12}--\eqref{eq:phipsi14} follow
  from~\eqref{eq:phipsi8}--\eqref{eq:phipsi10} and the fact that
  $*u^*=\iota(u)\dvol$ and $*(u^*\wedge\psi)=\iota(u)\phi$
  by~\eqref{eq:phipsi4}. %
  To prove equation~\eqref{eq:phipsi15} take the exterior product of
  equation~\eqref{eq:phipsi14} with $\psi$ and then
  use~\eqref{eq:phipsi4} to obtain
  \begin{equation*}
    \psi\wedge*(\psi\wedge\iota(u)\phi)=\psi\wedge3u^*=3\,{*\iota(u)\phi}.
  \end{equation*}
  Equation~\eqref{eq:phipsi16} follows from~\eqref{eq:phipsi6} 
  and the fact that the left hand side in~\eqref{eq:phipsi16} 
  is symmetric in $u$ and $v$. %
  Equation~\eqref{eq:phipsi17} is equivalent to~\eqref{eq:psicross}
  in the proof of \autoref{le:psi}. %
  To prove equation~\eqref{eq:phipsi18} choose $w:=u\times v$ 
  in~\eqref{eq:phipsi1} to obtain 
  \begin{equation*}
    \abs{u\times v}^2\dvol
    = u^*\wedge v^*\wedge(u\times v)^*\wedge\psi
    = u^*\wedge v^*\wedge *\iota(u\times v)\phi.
  \end{equation*}
  Here the last equation follows from~\eqref{eq:phipsi4}. %
  To prove~\eqref{eq:phipsi19} we compute
  \begin{equation*}
    \begin{split}
      &
      u^*\wedge v^*\wedge\iota(u)\phi\wedge\iota(v)\psi   \\
      &= 
      -\iota(v)\bigl(u^*\wedge v^*\wedge\iota(u)\phi\bigr)\wedge\psi  \\
      &= 
      - \inner{u}{v}v^*\wedge\iota(u)\phi\wedge\psi  
      + \abs{v}^2u^*\wedge\iota(u)\phi\wedge\psi 
      - u^*\wedge v^*\wedge (u\times v)^*\wedge\psi \\
      &= 
      -\inner{u}{v}\iota(u)\phi\wedge *\iota(v)\phi  
      + \abs{v}^2\iota(u)\phi\wedge *\iota(u)\phi 
      - u^*\wedge v^*\wedge *\iota(u\times v)\phi \\
      &=
      2\abs{u\times v}^2\dvol.
    \end{split}
  \end{equation*}
  Here the second step uses the identity 
  $\iota(v)\iota(u)\phi=\phi(u,v,\cdot)=(u\times v)^*$,
  the third step follows from~\eqref{eq:phipsi4}, and the 
  last step follows from~\eqref{eq:phipsi8} and~\eqref{eq:phipsi18}.

  To prove equation~\eqref{eq:phipsi20} take the exterior product
  with a $1$--form $w^*$ and use equation~\eqref{eq:phipsi1} to obtain
  \begin{align*}
    \bigl(\psi\wedge u^*\wedge v^*\bigr)\wedge w^*
    &=
      \phi(u,v,w)\dvol 
      =
      \inner{u\times v}{w}\dvol \\
    &=
      (*(u\times v)^*)\wedge w^*
      =
      (\iota(u\times v)\dvol)\wedge w^*.
  \end{align*}

  To prove equation~\eqref{eq:phipsi21} take the exterior product
  with a $1$--form $x^*$ and use equation~\eqref{eq:phipsi2} to obtain
  \begin{align*}
    \bigl(\phi\wedge u^*\wedge v^*\wedge w^*\bigr)\wedge x^*
    &=
      \psi(u,v,w,x)\dvol 
      =
      \inner{[u,v,w]}{x}\dvol \\
    &=
      (*[u,v,w]^*)\wedge x^*
      =
      (\iota([u,v,w])\dvol)\wedge x^*.
  \end{align*}

  To prove equation~\eqref{eq:phipsi22} take the exterior 
  product with $w^*\wedge x^*$ for $w,x\in V$ and use 
  equation~\eqref{eq:phipsi20} to obtain
  \begin{align*}
    \bigl(\phi\wedge u^*\wedge v^*\bigr)\wedge (w^*\wedge x^*)
    &=
      \psi(u,v,w,x)\dvol \\
    &=
      \inner{\iota(v)\iota(u)\psi}{w^*\wedge x^*}\dvol \\
    &=
      (*\iota(v)\iota(u)\psi)\wedge(w^*\wedge x^*).
  \end{align*}
  Since $\Lambda^2V^*$ has a basis of $2$--forms of the form
  $w^*\wedge x^*$, this proves~\eqref{eq:phipsi22}.

  To prove equations~\eqref{eq:phipsi23}
  and~\eqref{eq:phipsi24} it suffices to assume 
  \begin{equation*}
    \om=u^*\wedge v^*
  \end{equation*}
  for $u,v\in V$. %
  Then it follows from~\eqref{eq:phipsi17} and~\eqref{eq:phipsi22}
  that
  \begin{equation}
    \label{eq:utvphi}
    \begin{split}
      \iota(u\times v)\phi
      &= 
      u^*\wedge v^* + \iota(v)\iota(u)\psi \\
      &= 
      u^*\wedge v^* + *(u^*\wedge v^*\wedge\phi) \\
      &= 
      \om + *(\phi\wedge\om).
    \end{split}
  \end{equation}
  Moroever, $*(\psi\wedge\om)=(u\times v)^*$ by~\eqref{eq:phipsi20}.
  Hence, by~\eqref{eq:phipsi4} and~\eqref{eq:utvphi},
  \begin{align*}
    *\bigl(\psi\wedge *\bigl(\psi\wedge\om\bigr)\bigr)
    &=
      *\bigl(\psi\wedge(u\times v)^*\bigr) \\
    &=
      \iota(u\times v)\phi \\
    &=
      \om+*(\phi\wedge\om).
  \end{align*}
  This proves equation~\eqref{eq:phipsi23}.
  Moreover, by~\eqref{eq:phipsi12} and~\eqref{eq:utvphi}, 
  \begin{align*}
    *\bigl(\phi\wedge *\bigl(\phi\wedge\om\bigr)\bigr)
    &=
      *\bigl(\phi\wedge\bigl(\iota(u\times v)\phi-\om\bigr)\bigr) \\
    &=
      *\bigl(\phi\wedge\iota(u\times v)\phi\bigr) - *\bigl(\phi\wedge\om\bigr) \\
    &=
      2\iota(u\times v)\phi - *\bigl(\phi\wedge\om\bigr) \\
    &=
      2\om+*(\phi\wedge\om).
  \end{align*}
  This proves equation~\eqref{eq:phipsi24} and \autoref{le:phipsi}.
\end{proof}


    
\section{Normed algebras} \label{sec:NA}  

\begin{definition}
  \label{def:O}
  A \defined{normed algebra} consists of a finite dimensional 
  real Hilbert space $W$, a bilinear map
  \begin{equation*}
    W\times W\to W \co (u,v)\mapsto uv,
  \end{equation*}
  (called the \defined{product}), and a unit vector $1\in W$ (called
  the \defined{unit}), satisfying
  \begin{equation*}
    1u=u1=u
  \end{equation*}
  and
  \begin{equation}
    \label{eq:NORM}
    \abs{uv}=\abs{u}\abs{v}
  \end{equation}
  for all $u,v\in W$. 
\end{definition}

When $W$ is a normed algebra it is convenient to identify the real
numbers with a subspace of $W$ via multiplication with the unit $1$.
Thus, for $u\in W$ and $\lambda\in\R$, we write $u+\lambda$ instead of
$u+\lambda 1$.  Define an involution $W\to W\co u\mapsto\ubar$ (called
\defined{conjugation}) by $\bar 1:=1$ and $\ubar:=-u$ for
$u\in 1^\perp$.  Thus
\begin{equation}\label{eq:ubar}
  \ubar := 2\inner{u}{1} - u.
\end{equation}
We think of $\R\subset W$ as the real part of $W$ and of its
orthogonal complement as the imaginary part. 
The real and imaginary parts of $u\in W$ will be denoted by
$\RE\,u:=\inner{u}{1}$ and $\IM\,u:=u-\inner{u}{1}$. 

\begin{theorem}
  \label{thm:NDG}
  Normed algebras and vector spaces with cross products are related as
  follows.

\begin{enumerate}[(i)]
  \item
    \label{It_NDG1} 
    If $W$ is a normed algebra, then $V:=1^\perp$ is equipped with a
    cross product $V\times V\to V\co (u,v)\mapsto u\times v$ defined by
    \begin{equation}\label{eq:CROSS}
      u\times v := uv+\inner{u}{v}
    \end{equation}
    for $u,v\in 1^\perp$.

  \item
    \label{It_NDG2}
    If $V$ is a finite dimensional Hilbert space equipped with a cross
    product,
    then $W:=\R\oplus V$ is a normed algebra with
    \begin{equation}\label{eq:uv}
      uv := u_0v_0-\inner{u_1}{v_1}+u_0v_1+v_0u_1+u_1\times v_1
    \end{equation}
    for $u=u_0+u_1,v=v_0+v_1\in\R\oplus V$. %
    Here we identify a real number $\lambda$ with the pair
    $(\lambda,0)\in\R\oplus V$ and a vector $v\in V$ with the pair
    $(0,v)\in\R\oplus V$.
  \end{enumerate}

  These constructions are inverses of each other. %
  In particular, a normed algebra has dimension $1$, $2$, $4$, or $8$
  and is isomorphic to $\R$, $\C$, $\H$, or $\Oc$.
\end{theorem}

\begin{proof}
  See page~\pageref{proof:NDG}.
\end{proof}

\bigbreak

\begin{lemma}
  \label{le:one}
  Let $W$ be a normed algebra.
  Then the following hold:

  \begin{enumerate}[(i)]
  \item
    \label{It_One1}
    For all $u,v,w\in W$ we have 
    \begin{equation}
      \label{eq:W1}
      \inner{uv}{w}=\inner{v}{\ubar w},\qquad
      \inner{uv}{w}=\inner{u}{w\vbar}.
    \end{equation}

  \item
    \label{It_One2}
    For all $u,v\in W$ we have 
    \begin{equation}\label{eq:W2}
      u\bar u = \abs{u}^2,\qquad
      u\vbar+v\ubar = 2\inner{u}{v}.
    \end{equation} 

  \item
    \label{It_One3}
    For all $u,v\in W$ we have
    \begin{equation}\label{eq:W3}
      \inner{u}{v}=\inner{\ubar}{\vbar},\qquad
      \overline{uv}=\vbar\ubar.
    \end{equation}

  \item
    \label{It_One4}
    For all $u,v,w\in W$ we have
    \begin{equation}\label{eq:W4}
      u(\vbar w)+v(\ubar w) = 2\inner{u}{v}w,\qquad
      (u\vbar)w+(u\wbar)v = 2 \inner{v}{w}u
    \end{equation} 
  \end{enumerate}
\end{lemma}

\begin{proof}
  We prove~\itref{It_One1}. %
  The first equation in~\eqref{eq:W1} is obvious when $u$ is a real
  multiple of $1$. %
  Hence, it suffices to assume that $u$ is orthogonal to $1$. %
  Expanding the identities $\abs{uv+uw}^2=\abs{u}^2\abs{v+w}^2$ and
  $\abs{uv+wv}^2=\abs{u+w}^2\abs{v}^2$ we obtain the equations
  \begin{equation}
    \label{eq:uvuw}
    \inner{uv}{uw} = \abs{u}^2\inner{v}{w},\qquad
    \inner{uv}{wv} = \inner{u}{w}\abs{v}^2.
  \end{equation}
  If $u$ is orthogonal to $1$, the first
  equation in~\eqref{eq:uvuw} gives
  \begin{equation*}
    \inner{uv}{w}+\inner{v}{uw}
    = \inner{(1+u)v}{(1+u)w} - (1+\abs{u}^2)\inner{v}{w}=0.
  \end{equation*}
  Since $\bar u=-u$ for $u\in 1^\perp$, this proves the first equation
  in~\eqref{eq:W1}. %
  The proof of the second equation is similar.

  We prove~\itref{It_One2}.
  Using the second equation in~\eqref{eq:W1} with 
  $v=\bar u$ we obtain
  $
    \inner{u\ubar}{w}=\inner{u}{wu}=\inner{1}{w}\abs{u}^2.
  $
  Here we have used the second equation in~\eqref{eq:uvuw}.
  This implies $u\ubar=\abs{u}^2$ for every $u\in W$.  
  Replacing $u$ by $u+v$ gives $u\vbar+v\ubar = 2\inner{u}{v}$.
  This proves~\eqref{eq:W2}.

  We prove~\itref{It_One3}.
  That conjugation is an isometry follows immediately from the
  definition.
  Using~\eqref{eq:W2} with $v$ replaced by $\vbar$ we obtain
  \begin{equation*}
    \vbar\ubar 
    = 2\inner{u}{\vbar} - uv
    = 2\inner{uv}{1} - uv
    = \overline{uv}.
  \end{equation*}
  Here the second equation follows from~\eqref{eq:W1}.
  This proves~\eqref{eq:W3}. 

  We prove~\itref{It_One4}.
  For all $u,w\in W$ we have
  \begin{equation}
    \label{eq:uww}
    \inner{u(\ubar w)}{w}
    = \abs{\ubar w}^2
    = \abs{\ubar}^2\abs{w}^2
    = \abs{u}^2\abs{w}^2
  \end{equation}
  Since the operator $w\mapsto u(\ubar w)$ is self-adjoint,
  by~\eqref{eq:W1}, this shows that $u(\ubar w) = \abs{u}^2w$ for all
  $u,w\in W$. %
  Replacing $u$ by $u+v$ we obtain the first equation
  in~\eqref{eq:W4}.  The proof of the second equation is similar.
\end{proof}

\begin{proof}[Proof of \autoref{thm:NDG}]\label{proof:NDG}
  Let $W$ be a normed algebra. %
  It follows from~\eqref{eq:W1} that $\inner{u}{v}=-\inner{uv}{1}$
  and, hence, $u\times v:=uv+\inner{u}{v}\in 1^\perp$ for all
  $u,v\in 1^\perp$. %
  We write an element of $W$ as $u=u_0+u_1$ with
  $u_0:=\inner{u}{1}\in\R$ and $u_1:=u-\inner{u}{1}\in V=1^\perp$. %
  For $u,v\in W$ we compute
  \begin{equation*}
    \begin{split}
      \abs{u}^2\abs{v}^2-\abs{uv}^2
      =&
      \left(u_0^2+\abs{u_1}^2\right)
      \left(v_0^2+\abs{v_1}^2\right) 
      - \left(u_0v_0-\inner{u_1}{v_1}\right)^2 \\
      & 
      - \abs{u_0v_1+v_0u_1+u_1\times v_1}^2  \\
      =&
      \;u_0^2\abs{v_1}^2 + v_0^2\abs{u_1}^2 
      + 2u_0v_0\inner{u_1}{v_1} 
      + \abs{u_1}^2\abs{v_1}^2 
      - \inner{u_1}{v_1}^2 \\
      &
      - \abs{u_0v_1+v_0u_1}^2 - \abs{u_1\times v_1}^2 
      - 2\inner{u_0v_1+v_0u_1}{u_1\times v_1} \\
      =&
      \abs{u_1}^2\abs{v_1}^2 - \inner{u_1}{v_1}^2 
      - \abs{u_1\times v_1}^2 \\
      &
      -\; 2u_0\inner{v_1}{u_1\times v_1}
      - 2v_0\inner{u_1}{u_1\times v_1}.
    \end{split}
  \end{equation*}
  The right hand side vanishes for all $u$ and $v$ if and only if the
  product on~$V$ satisfies~\eqref{eq:orthogonal} and~\eqref{eq:area}. %
  Hence,~\eqref{eq:CROSS} defines a cross product on $V$ and the
  product can obviously be recovered from the cross product
  via~\eqref{eq:uv}. %
  Conversely, the same argument shows that, if $V$ is equipped with a
  cross product, the formula~\eqref{eq:uv} defines a normed algebra
  structure on $W:=\R\oplus V$. %
  Moreover, by \autoref{thm:cross}, $V$ has dimension $0$, $1$,
  $3$, or $7$. %
  This proves \autoref{thm:NDG}.
\end{proof}

\begin{remark}
  \label{rmk:W}
  If $W$ is a normed algebra and the cross product on $V:=1^\perp$ is
  defined by~\eqref{eq:CROSS}, then the commutator of two elements
  $u,v\in W$ is given by
  \begin{equation}
    \label{eq:uvcross}
    [u,v] := uv-vu = 2u_1\times v_1.
  \end{equation}
  In particular, the product on $W$ is commutative in dimensions $1$
  and $2$ and is not commutative in dimensions $4$ and $8$.
\end{remark}

\begin{remark}\label{rmk:orientW}
  Let $W$ be a normed algebra of dimension $4$ or $8$. %
  Then $V:=1^\perp$ has a natural orientation determined by
  \autoref{le:cross3} or \autoref{le:imO}, respectively, in
  dimensions $3$ and $7$. %
  We orient $W$ as $\R\oplus V$.
\end{remark}

\begin{remark}\label{rmk:assocW}
  If $W$ is a normed algebra and the cross product on $V:=1^\perp$ is
  defined by~\eqref{eq:CROSS}, then the associator bracket on $V$ is
  related to the product on $W$ by
  \begin{equation}
    \label{eq:assocW}
    (uv)w-u(vw) = 2[u_1,v_1,w_1]
  \end{equation}
  for all $u,v,w\in W$. %
  Thus $W$ is an associative algebra in dimensions $1,2,4$ and is not
  associative in dimension $8$. %
  The formula~\eqref{eq:assocW} is the reason for 
  the term {\it associator bracket}. %
  Many authors actually define the associator bracket as the left hand
  side of equation~\eqref{eq:assocW} (see for example~\cite{Harvey1982}). %

  To prove~\eqref{eq:assocW}, we observe that the associator bracket
  on $V$ can be written in the form
  \begin{equation}
    \label{eq:assocV}
    \begin{split}
      2[u,v,w]
      &= 2(u\times v)\times w + 2\inner{v}{w}u-2\inner{u}{w}v \\
      &= (u\times v)\times w - u\times(v\times w) 
       + \inner{v}{w}u - \inner{u}{v}w
    \end{split}
  \end{equation}
  for $u,v,w\in V$. %
  Here the first equation follows from~\eqref{eq:associator} and the
  second 
  equation follows
  from~\eqref{eq:AREA}. For $u,v,w\in V$ we compute
  \begin{align*}
    (uv)w-u(vw)
    &=
      (-\inner{u}{v} + u\times v)w
      - u(-\inner{v}{w} + v\times w) \\
    &=
      (u\times v)\times w - u\times(v\times w)
      - \inner{u}{v}w + \inner{v}{w}u \\
    &=
      2[u,v,w].
  \end{align*}
  Here the first equation follows from the definition of the cross
  product in~\eqref{eq:CROSS}, the second equation follows by
  applying~\eqref{eq:CROSS} again and using~\eqref{eq:frobenius}, and
  the last equation follows from~\eqref{eq:assocV}. %
  Now, if any of the factors $u,v,w$ is a real number, the term on the
  left vanishes. %
  Hence, real parts can be added to the vectors without changing the
  expression.
\end{remark}

\begin{theorem}\label{thm:tcW}
  Let $W$ be an $8$--dimensional normed algebra.

  \begin{enumerate}[(i)]
  \item
    \label{It_tcW1}
    The map $W^3\to W\co (u,v,w)\mapsto u\times v\times w$ defined by
    \begin{equation}\label{eq:tcW}
      u\times v\times w 
      := \tfrac{1}{2}\bigl((u\vbar)w - (w\vbar)u\bigr)
    \end{equation}
    (called the \defined{triple cross product} of $W$) 
    is alternating and satisfies
    \begin{equation}
      \label{eq:tcW1}
      \inner{x}{u\times v\times w} + \inner{u\times v\times x}{w} = 0,
    \end{equation}
    \begin{equation}
      \label{eq:tcW2}
      \abs{u\times v\times w} = \abs{u\wedge v\wedge w},
    \end{equation}
    for all $u,v,w,x\in W$ and
    \begin{equation}\label{eq:tcW3}
      \inner{e\times u\times v}{e\times w\times x} 
      = -\abs{e}^2\inner{u\times v\times w}{x}
    \end{equation}
    whenever $e,u,v,w,x\in W$ are orthonormal.

  \item
    \label{It_tcW2}
    The map $\Phi\co W^4\to\R$ defined by
    \begin{equation*}
      \Phi(x,u,v,w) := \inner{x}{u\times v\times w}
    \end{equation*}
    (called the \defined{Cayley calibration} of $W$)
    is an alternating $4$--form. Moreover, 
    $\Phi$ is self-dual, i.e., 
    \begin{equation}\label{eq:ASD}
      \Phi=*\Phi,
    \end{equation}
    where $*\co \Lambda^kW^*\to\Lambda^{8-k}W^*$ denotes the
    Hodge $*$--operator associated to the inner product and 
    the orientation of \autoref{rmk:orientW}.

  \item
    \label{It_tcW3}
    Let $V:=1^\perp$ with the cross product defined
    by~\eqref{eq:CROSS} and the associator bracket
    $[\cdot,\cdot,\cdot]$ defined by~\eqref{eq:associator}. %
    Let $\phi\in\Lambda^3V^*$ and ${\psi\in\Lambda^4V^*}$ be the
    associative and coassociative calibrations of $V$ defined
    by~\eqref{eq:phi} and~\eqref{eq:psi}, respectively. %
    Then the triple cross product~\eqref{eq:tcW} of $u,v,w\in W$ can
    be expressed as
    \begin{equation}
      \label{eq:tcW4}
      \begin{split}
        u\times v\times w 
        = &\; \phi(u_1,v_1,w_1) - [u_1,v_1,w_1]  \\
        &- u_0(v_1\times w_1) - v_0(w_1\times u_1) - w_0(u_1\times v_1)
      \end{split}
    \end{equation}
    and the Cayley calibration is given by
    \begin{equation}
      \label{eq:Phiphi}
      \Phi = 1^*\wedge\phi+\psi.
    \end{equation}

  \item
    \label{It_tcW4}
    For all $u,v\in W$ we have
    \begin{equation}
      \label{eq:tcW5}
      uv = u\times 1\times v+\inner{u}{1}v+\inner{v}{1}u-\inner{u}{v}.
    \end{equation}
  \end{enumerate}
\end{theorem}

\begin{remark}
  There is a choice involved in the definition of the triple cross
  product in~\eqref{eq:tcW}. %
  An alternative formula is
  \begin{equation*}
    (u,v,w)\mapsto\tfrac12\bigl(u(\vbar w)-w(\vbar u)\bigr).
  \end{equation*}
  This map also satisfies~\eqref{eq:tcW1} and~\eqref{eq:tcW2}. %
  However, it satisfies~\eqref{eq:tcW3} with the minus sign changed to
  plus and the resulting Cayley calibration is given by
  $\Phi=1^*\wedge\phi-\psi$ and is anti-self-dual. %
  Equation~\eqref{eq:tcW5} remains unchanged.
\end{remark}

\begin{proof}[Proof of \autoref{thm:tcW}]
  Let $W\times W\times W\to W:(u,v,w)\mapsto u\times v\times w$ 
  be the trilinear map defined by~\eqref{eq:tcW}. %
  We prove that this map satisfies~\eqref{eq:tcW4}.
  To see this, fix three vectors $u,v,w\in W$. 
  Then, by \eqref{eq:uvcross}, we have
  \begin{equation*}
    \vbar w-w\vbar
      = -2v_1\times w_1,\;\;
    uw-wu
      = -2w_1\times u_1,\;\;
    u\bar v-\vbar u
      = -2u_1\times v_1.
  \end{equation*}
  Multiplying these expressions by $u_0$, $v_0$, $w_0$, respectively,
  we obtain (twice) the last three expressions on the right in~\eqref{eq:tcW4}. 
  Thus it suffices to assume $u,v,w\in V$. %
  Then we obtain
  \begin{align*}
    2u\times v\times w
    &=
      (u\vbar)w-(w\vbar)u \\
    &=
      -(uv)w+(wv)u \\
    &=
      -(-\inner{u}{v}+u\times v)w + (-\inner{w}{v}+w\times v)u \\
    &=
      \inner{u\times v}{w} + \inner{u}{v}w - (u\times v)\times w \\ &\quad
       -\inner{w\times v}{u} - \inner{w}{v}u + (w\times v)\times u \\
    &= 
        2\phi(u,v,w) - 2[u,v,w].
  \end{align*}
  Here the third and fourth equations follow from~\eqref{eq:CROSS}, 
  and the last 
  equation
  follows from~\eqref{eq:phi} and~\eqref{eq:assocV}. %
  This proves that the formulas~\eqref{eq:tcW} and~\eqref{eq:tcW4}
  agree. %

  \smallbreak

  We prove~\itref{It_tcW1}. %
  By~\eqref{eq:tcW4} we have
  \begin{equation}\label{eq:xuvw}
    \begin{split}
      \inner{x}{u\times v\times w}
      &=
        x_0\phi(u_1,v_1,w_1) + \psi(x_1,u_1,v_1,w_1) \\ &\quad
        - u_0\phi(x_1,v_1,w_1) - v_0\phi(x_1,w_1,u_1) \\ &\quad
        - w_0\phi(x_1,u_1,v_1)
    \end{split}
  \end{equation}
  for $x,u,v,w\in W$. %
  Here we have used ${\phi(u_1,v_1,w_1)=\inner{u_1}{v_1\times w_1}}$
  and
  \begin{equation*}
    -\inner{x_1}{[u_1,v_1,w_1]} 
    =
      -\psi(u_1,v_1,w_1,x_1) = \psi(x_1,u_1,v_1,w_1).
  \end{equation*}
  It follows from the alternating properties of $\phi$ and $\psi$ that
  the right hand side of~\eqref{eq:xuvw} is an alternating $4$--form. %
  Hence, the map~\eqref{eq:tcW} is alternating and
  satisfies~\eqref{eq:tcW1}. %
  For $u,v,w\in V=1^\perp$ equation~\eqref{eq:tcW2} follows from
  \autoref{le:assocphi}. %
  In general, if $u,v,w\in W$ are pairwise orthogonal, it follows
  from~\eqref{eq:W2} and~\eqref{eq:W4} that
  \begin{equation*}
    (u\vbar)w = -(u\wbar)v = (w\ubar)v = -(w\vbar)u.
  \end{equation*}
  This shows that 
  \begin{equation}
    \label{eq:orthogonalW}
    \inner{u}{v}
    = \inner{v}{w}
    = \inner{w}{u} = 0
      \qquad\implies\qquad
    u\times v\times w
    = u(\vbar w)
  \end{equation}
  and, hence, by~\eqref{eq:NORM}, we have
  $\abs{u\times v\times w}=\abs{u\wedge v\wedge w}$ 
  in the orthogonal case. %
  This equation continues to hold in general by Gram--Schmidt. %
  This proves that the triple cross product satisfies~\eqref{eq:tcW2}.

  We prove~\eqref{eq:tcW3}.  The second equation in~\eqref{eq:W4}
  asserts that
  $ (yz)\zbar = \abs{z}^2y $
  for all $y,z\in W$. %
  Hence, by~\eqref{eq:orthogonalW}, we have
  \begin{align*}
    \inner{e\times u\times v}{e\times w\times x}
    &= 
      \inner{u\times v\times e}{w\times x\times e} \\
    &= 
      \inner{(u\vbar)e}{(w\xbar)e} \\
    &= 
      \inner{u\vbar}{((w\xbar)e)\ebar} \\
    &= 
      \abs{e}^2\inner{u\vbar}{w\xbar} \\
    &= 
      \abs{e}^2\inner{(u\vbar)x}{w} \\
    &= 
      \abs{e}^2\inner{u\times v\times x}{w} \\
    &= 
      -\abs{e}^2\inner{x}{u\times v\times w}
  \end{align*}
  whenever $e,u,v,w,x\in W$ are pairwise orthogonal. %
  Thus the triple cross product~\eqref{eq:tcW}
  satisfies~\eqref{eq:tcW3}. %
  This proves~\itref{It_tcW1}.

  We prove~\itref{It_tcW2} and~\itref{It_tcW3}. %
  That $\Phi$ is a $4$--form follows from~\itref{It_tcW1}. %
  That it satisfies equation~\eqref{eq:Phiphi} follows directly from
  the definition of $\Phi$ and equation~\eqref{eq:xuvw}.  %
  That $\Phi$ is self-dual with respect to the orientation of
  \autoref{rmk:orientW} follows from~\eqref{eq:Phiphi} and
  \autoref{le:psi}. %
  Equation~\eqref{eq:tcW4} was proved above.

  We prove~\itref{It_tcW4}.  %
  By~\eqref{eq:uvcross} and~\eqref{eq:tcW}, we have
  \begin{equation*}
    u_1\times v_1 = \tfrac12\bigl(uv-vu\bigr) = u\times 1\times v.
  \end{equation*}
  Hence, it follows from~\eqref{eq:uv} that
  \begin{align*}
    uv 
    &=
      u_0v_0 - \inner{u_1}{v_1} + u_0v_1+v_0u_1+u_1\times v_1 \\
    &=
      -u_0v_0 - \inner{u_1}{v_1} + u_0v + v_0u + u_1\times v_1 \\
    &=
      -\inner{u}{v} + \inner{u}{1}v + \inner{v}{1}u + u\times 1\times v.
  \end{align*}
  This proves~\eqref{eq:tcW5} and \autoref{thm:tcW}.
\end{proof}

\begin{example}\label{ex:O}
  If $W=\R^8=\R\oplus\R^7$ with coordinates $x_0,x_1,\dots,x_7$ and
  the cross product of \autoref{ex:cross7} on $\R^7$, then the
  associated Cayley calibration is given by
  \begin{align*}
    \Phi_0
    &= 
      e^{0123} - e^{0145} - e^{0167} - e^{0246} + e^{0257} - e^{0347} - e^{0356} \\ &\quad                                                     
      + e^{4567} - e^{2367} - e^{2345}
      - e^{1357} + e^{1346} - e^{1256} - e^{1247}.
  \end{align*}
  Thus 
  \begin{equation*}
    \Phi_0\wedge\Phi_0 = 14\dvol. 
  \end{equation*}
  (See the proof of \autoref{le:psi}.)
\end{example}

\begin{definition}\label{def:four}
  Let $W$ be an $8$--dimensional normed algebra. %
  The 
  \defined{fourfold cross product} on $W$ is the alternating multi-linear map
  $
  W^4\to W\co (u,v,w,x)\mapsto u\times v\times w\times x
  $
  defined by
  \begin{equation}\label{eq:four}
    4x\times u\times v\times w
    :=(u\times v\times w)\xbar
    - (v\times w\times x)\ubar 
    + (w\times x\times u)\vbar 
    - (x\times u\times v)\wbar.
  \end{equation}
\end{definition}

\begin{theorem}\label{thm:four}
  Let $W$ be an $8$--dimensional normed algebra with 
  triple cross product~\eqref{eq:tcW}, Cayley calibration 
  $\Phi\in\Lambda^4W^*$, and fourfold cross product~\eqref{eq:four}. %
  Then, for all $x,u,v,w\in W$, we have
  \begin{equation}
    \label{eq:four1}
    \abs{x\times u\times v\times w}
    = \abs{x\wedge u\wedge v\wedge w}
  \end{equation}
  and
  \begin{equation}\label{eq:four2}
    \begin{split}
      \Re (x\times u\times v\times w)
      &= 
        \Phi(x,u,v,w), \\
      \Im (x\times u\times v\times w)
      &=
        [x_1,u_1,v_1,w_1] - x_0[u_1,v_1,w_1] \\ &\quad
        + u_0[v_1,w_1,x_1] - v_0[w_1,x_1,u_1] \\ &\quad
        +w_0[x_1,u_1,v_1],
    \end{split}
  \end{equation}
  where the last five terms use the associator and coassociator brackets 
  on ${V:=1^\perp}$ defined by~\eqref{eq:associator}
  and~\eqref{eq:coass}. %
  In particular,
  \begin{equation}\label{eq:four3}
    \Phi(x,u,v,w)^2 + \abs{\Im(x\times u\times v\times w)}^2
    = \abs{x\wedge u\wedge v\wedge w}^2.
  \end{equation}
\end{theorem}

\begin{proof}
  That the fourfold cross product is alternating is obvious from the
  definition and the alternating property of the triple cross
  product.  %
  We prove that it satisfies~\eqref{eq:four1}. %
  For this it suffices to assume that $u,v,w,x$ are pairwise
  orthogonal. %
  Then $u\times v\times w=(u\vbar)w$ and hence
  \begin{equation*}
    (u\times v\times w)\xbar 
    = ((u\vbar)w)\xbar
    = -((u\vbar)x)\wbar
    = -(u\times v\times x)\wbar.
  \end{equation*}
  Here we have used~\eqref{eq:W2} and~\eqref{eq:W4}. %
  Using the alternating property of the triple cross product we obtain
  that the four summands in~\eqref{eq:four} agree in the orthogonal
  case.  Hence, $x\times u\times v\times w=((u\vbar)w)\xbar$ and so
  equation~\eqref{eq:four1} follows from~\eqref{eq:NORM}. %

  We prove~\eqref{eq:four2}. %
  Since $u\times 1\times v=u_1\times v_1$, we have
  \begin{align*}
    1\times u\times v\times w 
    &= 
      \tfrac14\Bigl(
        u\times v\times w + (v_1\times w_1)\ubar
        + (w_1\times u_1)\vbar + (u_1\times v_1)\wbar
      \Bigr) \\
    &=
      \tfrac14\Bigl(
        u\times v\times w + u_0(v_1\times w_1)
      + v_0(w_1\times u_1) + w_0(u_1\times v_1)
      \Bigr) \\ &\quad
      +\tfrac14\Bigl(
        \inner{v_1\times w_1}{u_1}
        + \inner{w_1\times u_1}{v_1}
        + \inner{u_1\times v_1}{w_1}
      \Bigr) \\
    &\quad
      -\tfrac14\Bigl(
        (v_1\times w_1)\times u_1
        + (w_1\times u_1)\times v_1
        + (u_1\times v_1)\times w_1
      \Bigr) \\
    &=
      \phi(u_1,v_1,w_1) - [u_1,v_1,w_1].
  \end{align*}
  The last equation follows from~\eqref{eq:tcW4} and the definition of
  the associator bracket in~\eqref{eq:associator}. %
  This proves~\eqref{eq:four2} in the case $x_1=0$. %
  Using the alternating property we may now assume that
  $x,u,v,w\in V:=1^\perp$. %
  If $x,u,v,w$ are orthogonal to $1$ it follows from~\eqref{eq:tcW4}
  that $ u\times v\times w = \phi(u,v,w)-[u,v,w] $ and
  $ \Phi(x,u,v,w) = - \inner{x}{[u,v,w]} = \psi(x,u,v,w)$. %
  Moreover, $\xbar =-x$ and similarly for $u,v,w$. %
  Hence,
  \begin{align*}
    &
    4x\times u\times v\times w \\
    &= 
      -(u\times v\times w)x + (v\times w\times x)u 
      - (w\times x\times u)v + (x\times u\times v)w  \\
    &=
      [u,v,w]x - [v,w,x]u + [w,x,u]v - [x,u,v]w \\ &\quad
      - \phi(u,v,w)x + \phi(v,w,x)u - \phi(w,x,u)v + \phi(x,u,v)w \\
    &=
      -\inner{[u,v,w]}{x} + \inner{[v,w,x]}{u} 
      - \inner{[w,x,u]}{v} + \inner{[x,u,v]}{w} \\ &\quad
      + [u,v,w]\times x - [v,w,x]\times u
      + [w,x,u]\times v - [x,u,v]\times w \\ &\quad
      - \phi(u,v,w)x + \phi(v,w,x)u - \phi(w,x,u)v + \phi(x,u,v)w \\
    &=
      -4\psi(u,v,w,x) - 4[u,v,w,x] \\
    &=
      4\Phi(x,u,v,w) + 4[x,u,v,w].
  \end{align*}
  Here the last but one equation follows from \autoref{le:coass}. %
  Thus we have proved \eqref{eq:four2} and \autoref{thm:four}.
\end{proof}



\section{Triple cross products} \label{sec:TCP}  

In this section we show how to recover the normed algebra structure on
$W$ from the triple cross product. %
In fact we shall see that every unit vector in $W$ can be used as a
unit for the algebra structure. %
We assume throughout that $W$ is a finite dimensional real Hilbert
space.

\begin{definition}\label{def:tc}
  An alternating multi-linear map
  \begin{equation}
    \label{eq:tc}
    W\times W\times W\to W\co (u,v,w)\mapsto u\times v\times w
  \end{equation}
  is called a \defined{triple cross product} if it satisfies
  \begin{gather}
    \label{eq:tc1}
    \inner{u\times v\times w}{u}
    = \inner{u\times v\times w}{v}
    = \inner{u\times v\times w}{w}
    = 0, \\
    \label{eq:tc2}
    \abs{u\times v\times w}
    = \abs{u\wedge v\wedge w}
  \end{gather}
  for all $u,v,w\in W$.
\end{definition}

A multi-linear map~\eqref{eq:tc} that satisfies~\eqref{eq:tc2}
also satisfies $u\times v\times w=0$ whenever $u,v,w \in W$ are
linearly dependent, and hence is necessarily alternating.

\begin{lemma}
  \label{le:tc1}
  Let~\eqref{eq:tc} be an alternating multi-linear map. %
  Then~\eqref{eq:tc1} holds if and only if, 
  for all ${x,u,v,w\in W}$, we have
  \begin{equation}\label{eq:tc3}
    \inner{x}{u\times v\times w} + \inner{u\times v\times x}{w} = 0.
  \end{equation}
\end{lemma}

\begin{proof}
  If~\eqref{eq:tc3} holds, then~\eqref{eq:tc1} follows 
  directly
  from the alternating property of the map~\eqref{eq:tc}. %
  To prove the converse, expand the expression
  $\inner{u\times v\times(w+x)}{w+x}$ and use~\eqref{eq:tc1} to
  obtain~\eqref{eq:tc3}.
\end{proof}

\begin{lemma}\label{le:tc2}
  Let~\eqref{eq:tc} be an alternating multi-linear map
  satisfying~\eqref{eq:tc1}. %
  Then equation~\eqref{eq:tc2} holds if and only if, for all ${u,v,w\in W}$, we
  have
  \begin{equation}
    \label{eq:tc4}
    \begin{split}
      &u\times v\times (u\times v\times w) + \abs{u\wedge v}^2w \\
      &= 
      \left(\abs{v}^2\inner{u}{w}-\inner{u}{v}\inner{v}{w}\right)u
      + \left(\abs{u}^2\inner{v}{w}-\inner{v}{u}\inner{u}{w}\right)v.
    \end{split}
  \end{equation} 
\end{lemma}

\begin{proof}
  If~\eqref{eq:tc4} holds and $w$ is orthogonal to $u$ and $v$, then
  \begin{equation*}
    u\times v\times (u\times v\times w) = - \abs{u\wedge v}^2w.
  \end{equation*}
  Taking the inner product with $w$ and using~\eqref{eq:tc3} we
  obtain~\eqref{eq:tc2} under the assumption
  $\inner{u}{w}=\inner{v}{w}=0$. %
  Since both sides of equation~\eqref{eq:tc2} remain unchanged if we
  add to $w$ a linear combination of $u$ and $v$, this proves
  that~\eqref{eq:tc4} implies~\eqref{eq:tc2}.
  
  To prove the converse we assume~\eqref{eq:tc2}. %
  If $w$ is orthogonal to $u$ and $v$, we have
  $
  \abs{u\times v\times w}^2 = \abs{u\wedge v}^2\abs{w}^2.
  $
  Replacing $w$ by $w+x$ we obtain
  \begin{equation}
    \label{eq:tc5}
    w,x\in u^\perp\cap v^\perp
      \quad\implies\quad
    \inner{u\times v\times w}{u\times v\times x}
    = \abs{u\wedge v}^2\inner{w}{x}.
  \end{equation}
  Using~\eqref{eq:tc3} we obtain~\eqref{eq:tc4} for every vector
  $w\in u^\perp\cap v^\perp$. %
  Replacing a general vector $w$ by its projection onto the orthogonal
  complement of the subspace spanned by $u$ and $v$ we deduce
  that~\eqref{eq:tc4} holds in general. %
  This proves \autoref{le:tc2}.
\end{proof}

Let~\eqref{eq:tc} be a triple cross product. %
If $e\in W$ is a unit vector, then the subspace
$
V_e:=e^\perp
$ 
carries a cross product $(u,v)\mapsto u\times_ev$ defined by 
$
u\times_ev:=u\times e\times v.
$
Hence, by \autoref{thm:cross}, the dimension of $V_e$ is $0$, $1$,
$3$, or $7$. %

It follows that the dimension of $W$ is $0$, $1$, $2$, $4$, or $8$. %

\begin{lemma}\label{le:tc3}
  Assume $\dim W=8$ and let~\eqref{eq:tc} be a triple cross product. %
  Then there is a number $\eps\in\{\pm1\}$ such that 
  \begin{equation}
    \label{eq:tc6}
    e\times u\times(e\times v\times w)
    = \eps\abs{e}^2u\times v\times w
  \end{equation}
  whenever $e,u,v\in W$ are pairwise orthogonal and $w\in W$ is
  orthogonal to $e,u,v,$ and $e\times u\times v$. 
\end{lemma}

\begin{proof}  
  It suffices to assume that the vectors $e,u,v\in W$ are orthonormal.
  Then the subspace 
  \begin{equation*}
    H := \mathrm{span}(e,u,v,e\times u \times v)^\perp
  \end{equation*}
  has dimension four. %
  It follows from~\eqref{eq:tc3} and~\eqref{eq:tc5} 
  that the formulas
  \begin{equation*}
      Iw := e\times u\times w,\qquad
      Jw := e\times v\times w,\qquad
      Kw := u\times v\times w,
  \end{equation*}
  define endomorphisms $I,J,K$ of $H$. %
  Moreover, by~\eqref{eq:tc3}, these 
  operators are skew adjoint and, by~\eqref{eq:tc5}, they are complex 
  structures on $H$. %
  It follows also from~\eqref{eq:tc5} that
  $
  e\times x\times(e\times x\times w) = -\abs{x}^2w
  $
  whenever $e,x,w$ are pairwise orthogonal and $\abs{e}=1$. %
  Assuming $w\in H$ and using this identity with $x=u+v$ we obtain
  $
  IJ+JI=0.
  $
  This implies that the automorphisms of $H$ of the form
  $aI+bJ+cIJ$ with $a^2+b^2+c^2 = 1$ belong to the space
  $\cJ$ of orthogonal complex structures on $H$. %
  They form one of the two components of $\cJ$ and $K$ belongs to this
  component because it anticommutes with $I$ and~$J$. %
  Hence, $K = \eps IJ$ with $\eps = \pm 1$. %
  Since the space of orthonormal triples in $W$ is connected, 
  and the constant $\eps$ depends continuously on the 
  triple $e,u,v$, we have proved~\eqref{eq:tc6} under the assumption 
  that $e,u,v$ are orthonormal and $w$ is orthogonal to the
  vectors ${e,u,v,e\times u\times v}$. %
  Hence, the assertion 
  follows by scaling. %
  This proves \autoref{le:tc3}.
\end{proof}

\begin{definition}
  \label{def:tc+}
  Assume $\dim W=8$. %
  A triple cross product~\eqref{eq:tc} is called \defined{positive} if
  it satisfies~\eqref{eq:tc6} with $\eps=1$ and is called
  \defined{negative} if it satisfies~\eqref{eq:tc6} with $\eps=-1$.
\end{definition}

\begin{definition}
  \label{def:Phi}
  Assume $\dim W=8$ and let~\eqref{eq:tc} be a triple cross product.
  Then, by \autoref{le:tc1}, the map
  $
  \Phi\co W\times W\times W\times W\to\R
  $ 
  defined by 
  \begin{equation}
    \label{eq:Phi}
    \Phi(x,u,v,w)
    := \inner{x}{u\times v\times w}
  \end{equation}
  is an alternating $4$--form. %
  It is called the \defined{Cayley calibration} of $W$.
\end{definition}

\begin{theorem}
  \label{thm:cayley}
  Assume $\dim W=8$ and let~\eqref{eq:tc} be a triple cross product
  with Cayley calibration ${\Phi\in\Lambda^4W^*}$ given
  by~\eqref{eq:Phi}. %
  Let $e\in W$ be a unit vector.
  \begin{enumerate}[(i)]
  \item
    \label{It_Cayley1}
    Define the map $\psi_e\co W^4\to\R$ by
    \begin{equation}\label{eq:psie}
      \begin{split}
        \psi_e(u,v,w,x)
        &:=
          \inner{e\times u\times v}{e\times w\times x} \\ &\quad
          - \bigl(\inner{u}{w}-\inner{u}{e}\inner{e}{w}\bigr)
          \bigl(\inner{v}{x}-\inner{v}{e}\inner{e}{x}\bigr) \\ &\quad
          + \bigl(\inner{u}{x}-\inner{u}{e}\inner{e}{x}\bigr)
         \bigl(\inner{v}{w}-\inner{v}{e}\inner{e}{w}\bigr).
      \end{split}
    \end{equation}
    Then $\psi_e\in \Lambda^4W^*$ and 
    \begin{equation}\label{eq:cayley}
      \Phi
      =
        e^*\wedge\phi_e+\eps\psi_e,\qquad
      \phi_e
      :=
        \iota(e)\Phi\in\Lambda^3W^*,
    \end{equation}
    where $\eps\in\{\pm1\}$ is as in \autoref{le:tc3}.

  \item
    \label{It_Cayley2}
    The subspace $V_e:=e^\perp$ carries a cross product
    \begin{equation}
      \label{eq:crosse}
      V_e\times V_e\to V_e\co
      (u,v)\mapsto u\times_ev := u\times e\times v,
    \end{equation}
    the restriction of $\phi_e$ to $V_e$ is the associative
    calibration of~\eqref{eq:crosse}, and the restriction of $\psi_e$
    to $V_e$ is the coassociative calibration of~\eqref{eq:crosse}.

  \item
    \label{It_Cayley3}
    The space $W$ is a normed algebra with unit $e$ and multiplication
    and conjugation given by
    \begin{equation}\label{eq:NDA}
      uv
      :=
        u\times e\times v+\inner{u}{e}v+\inner{v}{e}u-\inner{u}{v}e,\qquad 
      \bar u
      :=
        2\inner{u}{e}e-u.
    \end{equation}
    If the triple cross product is positive, then
    $(u\vbar)w-(w\vbar)u=2u\times v\times w$.
  \end{enumerate}
\end{theorem}

\begin{proof}
  We prove~\itref{It_Cayley1}. %
  If the vectors $e,u,v,w,x$ are pairwise orthogonal, then
  \begin{equation}\label{eq:tc7}
    \inner{e\times u\times x}{e\times v\times w}
    = -\eps\abs{e}^2\inner{x}{u\times v\times w}.
  \end{equation}
  To see this, take the inner product of~\eqref{eq:tc6} with $x$. %
  Then it follows from~\eqref{eq:tc3} that~\eqref{eq:tc7} holds under
  the additional assumption that $w$ is perpendicular to
  $e\times u\times v$. %
  Since $x$ is orthogonal to $e$, this additional condition can be
  dropped, as both sides of the equation remain unchanged if we add to
  $w$ a multiple of $e\times u\times v$. %
  Thus we have proved~\eqref{eq:tc7}.

  Now fix a unit vector $e\in W$. %
  By definition, $\psi_e$ is alternating in the first two and last two
  arguments, and satisfies $ \psi_e(u,v,w,x)=\psi_e(w,x,u,v) $ for all
  $u,v,w,x\in W$. %
  By~\eqref{eq:tc2} we also have $\psi_e(u,v,u,v)=0$. %
  Expanding the identity $\psi_e(u,v+x,u,v+x)=0$ we obtain
  $ \psi_e(u,v,u,x)=0 $ for all $u,v,x\in W$. %
  Using this identity with $u$ replaced by $u+w$ gives
  \begin{equation*}
    \psi_e(u,v,w,x)+\psi_e(w,v,u,x)=0. 
  \end{equation*}
  Hence, $\psi_e$ is also skew-symmetric in the first and third
  argument and so is an alternating $4$--form. %
  To see that it satisfies~\eqref{eq:cayley} it suffices to show that
  $\eps\Phi$ and $\psi_e$ agree on $e^\perp$. %
  Since they are both $4$--forms, it suffices to show that they agree
  on every quadrupel of pairwise orthogonal vectors
  ${u,v,w,x\in e^\perp}$. %
  But in this case we have
  $ \psi_e(u,x,v,w)=-\eps\Phi(x,u,v,w)=\eps\Phi(u,x,v,w), $ by
  equation~\eqref{eq:tc7}. %
  This proves~\itref{It_Cayley1}.

  We prove~\itref{It_Cayley2}. %
  That~\eqref{eq:crosse} is a cross product on $V_e=e^\perp$ follows
  immediately from the definitions. %

  By~\eqref{eq:Phi} we have
  \begin{equation*}
    \inner{u\times_ev}{w} = \Phi(w,u,e,v)=\Phi(e,u,v,w) = \phi_e(u,v,w)
  \end{equation*}
  for $u,v,w\in V_e$, and hence the restriction of $\phi_e$
  to $V_e$ is the associative calibration. %
  Moreover, the
  associator bracket~\eqref{eq:associator} on $V_e$
  is given by 
  \begin{equation*}
    [u,v,w]_e
    = (u\times e\times v)\times e\times w
      + \inner{v}{w}u-\inner{u}{w}v.
  \end{equation*}
  Hence, for all $u,v,w,x\in V_e$, we have 
  \begin{align*}
    \inner{[u,v,w]_e}{x}
    &=
      \inner{e\times w\times(u\times e\times v)}{x} 
      + \inner{v}{w}\inner{u}{x} - \inner{u}{w}\inner{v}{x} \\
    &= 
      \inner{e\times u\times v}{e\times w\times x} 
      - \inner{u}{w}\inner{v}{x} + \inner{u}{x}\inner{v}{w} \\
    &=
      \psi_e(u,v,w,x),
  \end{align*}
  where the last equation follows from~\eqref{eq:psie}. %
  Hence, the restriction of $\psi_e$ to $V_e$ is the coassociative
  calibration and this proves~\itref{It_Cayley2}.

  We prove~\itref{It_Cayley3}. %
  That $e$ is a unit follows directly from the definitions.  To prove
  that the norm of the product is equal to the product of the norms we
  observe that $u\times e\times v$ is orthogonal to $e$, $u$, and $v$,
  by equation~\eqref{eq:tc3}. %
  Hence,
  \begin{align*}
    \abs{uv}^2 
    &=
      \abs{u\times e\times v+ \inner{u}{e}v+ \inner{v}{e}u - \inner{u}{v}e}^2 \\
    &=
      \abs{u\times e\times v}^2 -2\inner{v}{e}\inner{u}{v}\inner{v}{e} \\ &\quad
      +\;\inner{u}{e}^2\abs{v}^2 + \inner{v}{e}^2\abs{u}^2 + \inner{u}{v}^2 \\
    &=
      \abs{u}^2\abs{v}^2.
  \end{align*}
  Here the last equality uses the fact that 
  $\abs{u\times e\times v}^2=\abs{u\wedge e\wedge v}^2$. %
  Thus we have proved that $W$ is a normed algebra with unit $e$.

  If the triple cross product~\eqref{eq:tc} is positive, then $\eps=1$
  and hence equation~\eqref{eq:cayley} asserts that
  $\Phi=e^*\wedge\phi_e+\psi_e$. %
  Hence, it follows from~\eqref{eq:Phiphi} in
  \autoref{thm:tcW} that the Cayley calibration $\Phi_e$ associated to
  the above normed algebra structure is equal to $\Phi$. %
  This implies that the given triple cross product~\eqref{eq:tc}
  agrees with the triple cross product defined by~\eqref{eq:tcW}. %
  This proves~\itref{It_Cayley3} and \autoref{thm:cayley}.
\end{proof}

\begin{remark}\label{rmk:Worient}
  Assume $\dim W=8$ and let~\eqref{eq:tc} be a triple cross product
  with Cayley calibration ${\Phi\in\Lambda^4W^*}$ given
  by~\eqref{eq:Phi}. %
  Then, for every unit vector $e\in W$, the subspace $V_e=e^\perp$ is
  oriented by \autoref{le:imO} and \autoref{thm:cayley}. %
  We orient $W$ as the direct sum $ W = \R e\oplus V_e $. %
  This orientation is independent of the choice of the unit vector
  $e$. %
  With this orientation we have $e^*\wedge\phi_e=*\psi_e$, by
  \autoref{thm:cayley}~\itref{It_Cayley2} and \autoref{le:psi}. %
  Hence, it follows from equation~\eqref{eq:cayley} in
  \autoref{thm:cayley}~\itref{It_Cayley1} that $\Phi\wedge\Phi\ne
  0$. %
  In fact, the triple cross product is positive if and only if
  $\Phi\wedge\Phi>0$ with respect to our orientation and negative if
  and only if $\Phi\wedge\Phi<0$. %
  In the positive case $\Phi$ is self-dual and in the negative case
  $\Phi$ is anti-self-dual. %
\end{remark}

\begin{cor}\label{cor:tctc}
  Assume $\dim W=8$ and let~\eqref{eq:tc} be a triple cross product and
  let $\eps$ be as in \autoref{le:tc3}. %
  Then, for all $e,u,v,w\in W$, we have 
  \begin{equation}\label{eq:tctc}
    \begin{split}
      e\times u\times(e\times v\times w)
      &=
      \eps \abs{e}^2u\times v\times w - \eps \inner{e}{u\times v\times w}e \\
      &\quad
      - \eps\inner{e}{u}e\times v\times w \\
      &\quad
      - \eps\inner{e}{v}e\times w\times u \\
      &\quad
      - \eps\inner{e}{w}e\times u\times v \\
      &\quad
      - \bigl(\abs{e}^2\inner{u}{v}-\inner{e}{u}\inner{e}{v}\bigr)w \\
      &\quad
      + \bigl(\abs{e}^2\inner{u}{w}-\inner{e}{u}\inner{e}{w}\bigr)v \\
      &\quad
      + \bigl(\inner{u}{v}\inner{e}{w} - \inner{u}{w}\inner{e}{v}\bigr)e.
    \end{split}
  \end{equation}
\end{cor}

\begin{proof}
  Both sides of the equation remain unchanged if we add to $u$, $v$,
  or~$w$ a multiple of $e$. %
  Hence, it suffices to prove~\eqref{eq:tctc} under the assumption
  that $u,v,w$ are all orthogonal to $e$. %
  Moreover, both sides of the equation are always orthogonal to $e$. %
  Hence, it suffices to prove that the inner products of both sides
  of~\eqref{eq:tctc} with every vector $x\in e^\perp$ agree. %
  It also suffices to assume $\abs{e}=1$. %
  Thus we must prove that, if $e\in W$ is a unit vector and
  $u,v,w,x\in W$ are orthogonal to $e$, then we have
  \begin{equation*}
    \inner{e\times u\times(e\times v\times w)}{x}
    =
      \eps\inner{u\times v\times w}{x} 
      - \inner{u}{v}\inner{w}{x} + \inner{u}{w}\inner{v}{x}
  \end{equation*}
  or equivalently
  \begin{equation}
    \label{eq:tctcx}
    -\inner{e\times u\times x}{e\times v\times w}
    + \inner{u}{v}\inner{x}{w} - \inner{u}{w}\inner{x}{v}
    =
    \eps\inner{x}{u\times v\times w}.
  \end{equation}
  The right hand side of~\eqref{eq:tctcx} is $\eps\Phi(x,u,v,w)$ and,
  by~\eqref{eq:psie}, the left hand side of~\eqref{eq:tctcx} is
  $-\psi_e(u,x,v,w)$. %
  Hence, equation~\eqref{eq:tctcx} is equivalent to the assertion that
  the restriction of $\psi_e$ to $e^\perp$ agrees with $\Phi$. %
  But this follows from equation~\eqref{eq:cayley} in
  \autoref{thm:cayley}. %
  This proves \autoref{cor:tctc}.
\end{proof}

\begin{lemma}\label{le:cayley-sub}
  Assume $\dim W=8$ and let~\eqref{eq:tc} be a triple cross product
  with Cayley calibration ${\Phi\in\Lambda^4W^*}$ given
  by~\eqref{eq:Phi}. %
  Let $H\subset W$ be a $4$--dimensional linear subspace. %
  Then the following are equivalent:
  \begin{enumerate}[(i)]
  \item
    \label{It_CayleySub1}
    If $u,v,w\in H$, then $u\times v\times w\in H$.

  \item
    \label{It_CayleySub2}
    If $u,v\in H$ and $w\in H^\perp$, then
    $u\times v\times w\in H^\perp$.

  \item
    \label{It_CayleySub3}
    If $u\in H$ and $v,w\in H^\perp$, then
    $u\times v\times w\in H$.

  \item
    \label{It_CayleySub4}
    If $u,v,w\in H^\perp$, then
    $u\times v\times w\in H^\perp$.

  \item
    \label{It_CayleySub5}
    If $u,v,w\in H$ and $x\in H^\perp$, then
    $\Phi(x,u,v,w)=0$.

  \item
    \label{It_CayleySub6}
    If $x,u,v,w$ is an orthonormal basis of $H$, then
    $\Phi(x,u,v,w)=\pm1$.

  \item
    \label{It_CayleySub7}
    If $e\in H^\perp$ has norm one, then $H$
    is a coassociative subspace of $V_e:=e^\perp$.

  \item
    \label{It_CayleySub8}
    If $e\in H$ has norm one, then $H\cap V_e$ is an
    associative subspace of $V_e$.
  \end{enumerate}
  A $4$--dimensional subspace that satisfies these equivalent
  conditions is called a \defined{Cayley subspace} of $W$. %
  If the vectors $u,v,w\in W$ are linearly independent, then
  $H:=\mathrm{span}\{u,v,w,u\times v\times w\}$ is a Cayley
  subspace of $W$.
\end{lemma}

\begin{proof}
  We prove that~\itref{It_CayleySub1} is equivalent
  to~\itref{It_CayleySub5}. %
  If~\itref{It_CayleySub1} holds and $u,v,w\in H$,
  $x\in H^\perp$, then $u\times v\times w\in H$ and, hence,
  $\Phi(x,u,v,w)=\inner{x}{u\times v\times w}=0$.
  Conversely, if~\itref{It_CayleySub5} holds and $u,v,w\in H$,
  then $\inner{x}{u\times v\times w}=\Phi(x,u,v,w)=0$ for every
  $x\in H^\perp$ and hence $u\times v\times w\in H$.

  We prove that~\itref{It_CayleySub1} is equivalent to~\itref{It_CayleySub6}.
  If~\itref{It_CayleySub1} holds and $x,u,v,w$ is an orthonormal basis
  of $H$, then $u\times v\times w$ is orthogonal to $u,v,w$ and
  has norm one. %
  Since $u\times v\times w\in H$, 
  we must have $x=\pm u\times v\times w$.
  Hence $\Phi(x,u,v,w) = \pm\abs{x}^2 = \pm1$.
  Conversely, assume~\itref{It_CayleySub6}, let $u,v,w\in H$ be 
  orthonormal, and choose $x$ such that $x,u,v,w$ 
  form an orthonormal basis of~$H$. %
  Then
  \begin{equation*}
    \inner{x}{u\times v\times w}^2
    = \Phi(x,u,v,w)^2
    = 1 
    = \abs{x}^2\abs{u\times v\times w}^2.
  \end{equation*}
  Hence, $u\times v\times w$ is a real multiple of $x$ and so
  $u\times v\times w\in H$. %
  Since the triple cross product is alternating, the general case can
  be reduced to the orthonormal case by scaling and Gram--Schmidt.

  That~\itref{It_CayleySub6} is equivalent to~\itref{It_CayleySub7}
  follows from \autoref{le:coassoc} and the fact that~$\Phi|_{V_e}$~is
  the coassociative calibration on $V_e$. %
  Likewise, that~\itref{It_CayleySub6} is equivalent
  to~\itref{It_CayleySub8} follows from \autoref{le:assoc} and the
  fact that $\iota(e)\Phi|_{V_e}$ is the associative calibration 
  on~$V_e$.

  Thus we have proved that~\itref{It_CayleySub1},
  \itref{It_CayleySub5}, \itref{It_CayleySub6}, \itref{It_CayleySub7},
  \itref{It_CayleySub8} are equivalent. %
  The equivalence of~\itref{It_CayleySub1}, \itref{It_CayleySub2},
  \itref{It_CayleySub3} for a unit vector $u=e\in H$ follows from
  \autoref{le:coassoc} with $V:=V_e$ and $H$ replaced by $H^\perp$, 
  using the fact that $v\times_ew=-e\times v\times w$ 
  is the cross product on $V_e$. %
  
  The equivalence of~\itref{It_CayleySub3} and~\itref{It_CayleySub4}
  follows from the equivalence of~\itref{It_CayleySub1}
  and~\itref{It_CayleySub2} by interchanging the roles of $\Lambda$
  and $\Lambda^\perp$. %
  Thus we have proved the equivalence of
  conditions~\itref{It_CayleySub1}--\itref{It_CayleySub8}. %
  The last assertion of the lemma follows from~\itref{It_CayleySub1}
  and equation~\eqref{eq:tc4}. %
  This proves \autoref{le:cayley-sub}.
\end{proof}


    
\section{Cayley calibrations}
\label{sec:cayley}  

We assume throughout that $W$ is an $8$--dimensional real vector space.

\begin{definition}
  \label{def:nondeg}
  A $4$--form ${\Phi\in\Lambda^4W^*}$ is called \defined{nondegenerate}
  if for every triple $u,v,w$ of linearly independent vectors in $W$
  there is a vector $x\in W$ such that ${\Phi(u,v,w,x)\ne 0}$. %
  An inner product on $W$ is called \defined{compatible} with a $4$--form
  $\Phi$ if the map $W^3\to W\co (u,v,w)\mapsto u\times v\times w$
  defined by
  \begin{equation}
    \label{eq:crossPhi}
    \inner{x}{u\times v\times w} := \Phi(x,u,v,w)
  \end{equation}
  is a triple cross product. %
  A $4$--form ${\Phi\in\Lambda^4W^*}$ is called a \defined{Cayley-form} if
  it admits a compatible inner product.
\end{definition}

\begin{example}
  \label{ex:cayley}
  The standard Cayley-form on $\R^8$ in coordinates 
  $x_0,x_1,\ldots,x_7$ is given by
  \begin{align*}
    \Phi_0
    &= 
      e^{0123} - e^{0145} - e^{0167} 
      - e^{0246} + e^{0257} - e^{0347} - e^{0356} \\ &\quad
      + e^{4567} - e^{2367} - e^{2345}
      - e^{1357} + e^{1346} - e^{1256} - e^{1247}.
  \end{align*}
  It is compatible with the standard inner product and induces the
  standard triple cross product on $\R^8$ (see \autoref{ex:O}). %
  Note that $\Phi_0\wedge\Phi_0 = 14\,\dvol$.
\end{example}

As in \autoref{sec:imO} we shall see that a compatible inner
product, if it exists, is uniquely determined by $\Phi$. %
However, in contrast to \autoref{sec:imO}, nondegeneracy is, in
the present setting, not equivalent to the existence of a compatible
inner product, but is only a necessary condition. %
The goal in this section is to give an intrinsic characterization of
Cayley-forms. %
In particular, we shall see that every Cayley-form satisfies the
condition
$\Phi\wedge\Phi \ne 0$.
It seems to be an open question whether or not every nondegenerate
$4$--form on $W$ has this property; %
we could not find a counterexample but also did not see how to prove
it. %
We begin by characterizing compatible inner products.

\begin{lemma}
  \label{le:cayley}
  Fix an inner product on $W$ and a $4$--form $\Phi\in\Lambda^4W^*$. %
  Then the following are equivalent:
  \begin{enumerate}[(i)]
  \item
    \label{Lem_Cayley1}
    The inner product is compatible with $\Phi$. 

  \item
    \label{Lem_Cayley2}
    There is a unique orientation on~$W$, with volume form 
    $\dvol\in\Lambda^8W^*$, such that, for all $u,v,w\in W$, we have 
    \begin{equation}
      \label{eq:uvPhi}
      \iota(v)\iota(u)\Phi\wedge\iota(v)\iota(u)\Phi\wedge\Phi
      = 6\abs{u\wedge v}^2\dvol.
    \end{equation}

  \item
    \label{Lem_Cayley3}
    Choose the orientation on~$W$ and the volume
    form $\dvol\in\Lambda^8W^*$ as in~\itref{Lem_Cayley2}.
    Then, for all $u,v,w\in W$, we have 
    \begin{equation}
      \label{eq:uvwPhi}
      \iota(v)\iota(u)\Phi\wedge\iota(w)\iota(u)\Phi\wedge\Phi
      =
      6\left(\abs{u}^2\inner{v}{w}-\inner{v}{u}\inner{u}{w}
      \right)\dvol
    \end{equation}
  \end{enumerate}
  Each of these conditions implies that $\Phi$
  is nondegenerate and $\Phi\wedge\Phi\ne0$. 
\end{lemma}

\begin{proof}
  We prove that~\itref{Lem_Cayley1} implies~\itref{Lem_Cayley2}. %
  Assume the inner product is compatible with $\Phi$ and let
  $W^3\to W\co (u,v,w)\mapsto u\times v\times w$ be the triple cross
  product on $W$ defined by~\eqref{eq:crossPhi}. %
  Assume $u,v\in W$ are linearly independent. %
  Then the subspace
  \begin{equation*}
    W_{u,v} := \left\{w\in W : \inner{u}{w}=\inner{v}{w}=0\right\}
  \end{equation*}
  carries a symplectic form $\om_{u,v}\co W_{u,v}\times W_{u,v}\to\R$
  and a compatible complex structure $J_{u,v}\co W_{u,v}\to W_{u,v}$
  given by
  \begin{equation*}
    \om_{u,v}(x,w)
    :=
    \frac{\Phi(x,u,v,w)}{\abs{u\wedge v}},
    \qquad
    J_{u,v}w
    :=
    -\frac{u\times v\times w}{\abs{u\wedge v}}.
  \end{equation*}
  Equation~\eqref{eq:tc2} asserts that $J_{u,v}$ is an isometry on
  $W_{u,v}$ and equation~\eqref{eq:tc3} asserts that $J_{u,v}$ is skew
  adjoint. %
  Hence, $J_{u,v}$ is a complex structure on $W_{u,v}$ and
  equation~\eqref{eq:crossPhi} shows that, for all $x,w\in W_{u,v}$,
  we have
  \begin{equation*}
    \om_{u,v}(x,w)
    = \frac{\inner{x}{u\times v\times w}}{\abs{u\wedge v}}
    = -\inner{x}{J_{u,v}w}.
  \end{equation*}
  Thus the inner product $\om_{u,v}(\cdot,J_{u,v}\cdot)$ on $W_{u,v}$
  is the one inherited from $W$. %
  It follows that
  \begin{equation}
    \label{eq:sympuv}
    \om_{u,v}\wedge\om_{u,v}\wedge\om_{u,v}
    =
    6\dvol_{u,v},
  \end{equation}
  where $\dvol_{u,v}\in\Lambda^6W_{u,v}^*$ denotes the volume form on
  $W_{u,v}$ with the symplectic orientation. %
  Since the space of linearly independent pairs $u,v\in W$ is
  connected, there is a unique orientation on $W$ such that, for every
  pair $u,v$ of linearly independent vectors in $W$ and every
  symplectic basis $e_1,\dots,e_6$ of $W_{u,v}$, the basis
  $u,v,e_1,\dots,e_6$ of $W$ is positively oriented. %
  Let $\dvol\in\Lambda^8W^*$ be the volume form of $W^*$ for this
  orientation. %
  Then
  \begin{equation*}
    \dvol_{u,v} = \frac{1}{\abs{u\wedge v}}\iota(v)\iota(u)\dvol|_{W_{u,v}}
  \end{equation*}
  and, hence, equation~\eqref{eq:uvPhi} follows
  from~\eqref{eq:sympuv}. %
  This shows that~\itref{Lem_Cayley1} implies~\itref{Lem_Cayley2}. %
  That~\itref{Lem_Cayley2} implies~\itref{Lem_Cayley3} follows by
  using~\eqref{eq:uvPhi} with $v$ replaced by~${v+w}$.

  We prove that~\itref{Lem_Cayley3} implies~\itref{Lem_Cayley1}.
  Assume there is an orientation on $W$ such that~\eqref{eq:uvwPhi}
  holds, and define the map
  $W^3\to W\co (u,v,w)\mapsto u\times v\times w$
  by~\eqref{eq:crossPhi}. %
  That this map is alternating and satisfies~\eqref{eq:tc1} is
  obvious. %
  We prove that it satisfies~\eqref{eq:tc2}. %
  Fix a unit vector $e\in W$ and denote
  \begin{equation*}
    V_e
    :=
      \left\{v\in W : \inner{e}{v}=0\right\},\quad
    \phi_e
    :=
      \iota(e)\Phi|_{V_e},\quad
    \dvol_e
    :=
      \iota(e)\dvol|_{V_e}.
  \end{equation*}
  Then equation~\eqref{eq:uvwPhi} asserts that 
  \begin{equation*}
    \iota(u)\phi_e\wedge\iota(v)\phi_e\wedge\phi_e 
    =
    6\inner{u}{v}\dvol_e
  \end{equation*}
  for every $u\in V_e$. %
  Hence, $\phi_e$ satisfies condition~\itref{Lem_Cayley1} in
  \autoref{le:imO} and therefore is compatible with the inner
  product. %
  This means that the bilinear map
  $V_e\times V_e\to V_e\co (u,v)\mapsto u\times_ev$ defined by
  $ \inner{u\times_ev}{w} := \phi_e(u,v,w) $ is a cross product on
  $V_e$. %
  Since
  $\phi_e(u,v,w) = \Phi(w,u,e,v) = \inner{u\times e\times v}{w}$, we
  have $ u\times_ev=u\times e\times v$. %
  This implies $\abs{u\times e\times v}=\abs{u\wedge v}$ whenever $u$
  and $v$ are orthogonal to $e$ and $e$ has norm one. %
  Using Gram--Schmidt and scaling, we deduce that our map
  $(u,v,w)\mapsto u\times v\times w$ satisfies~\eqref{eq:tc2} and,
  hence, is a triple cross product. %
  Thus we have proved that~\itref{Lem_Cayley1}, \itref{Lem_Cayley2},
  and~\itref{Lem_Cayley3} are equivalent. %
  Moreover, condition~\itref{Lem_Cayley2} implies that $\Phi$ is
  nondegenerate and~\itref{Lem_Cayley1} implies that
  $\Phi\wedge\Phi\ne 0$, by \autoref{rmk:Worient}. %
  This proves \autoref{le:cayley}.
\end{proof}

We are now in a position to characterize Cayley-forms intrinsically. %
A $4$--form $\Phi$ is nondegenerate if and only if the $2$--form
$\iota(v)\iota(u)\Phi\in\Lambda^2W^*$ descends to a symplectic form on
the quotient $W/\mathrm{span}\{u,v\}$ or, equivalently, the $8$--form
$\iota(v)\iota(u)\Phi\wedge\iota(v)\iota(u)\Phi\wedge\Phi$ is nonzero
whenever $u,v$ are linearly independent. %
The question to be adressed is under which additional condition we can
find an inner product on $W$ that satisfies~\eqref{eq:uvPhi}.

\begin{theorem}
  \label{thm:CAYLEY}  
  A $4$--form $\Phi\in\Lambda^4W^*$ admits a compatible inner product
  if and only if it satisfies the following condition.
  \begin{equation}    
    \tag{C}
    \label{Eq_C}
    \left\{~
    \parbox{\dimexpr\linewidth-4em}{%
      ~\\
      $\Phi$ is nondegenerate and,
      if $u,v,w\in W$ are linearly independent and
      \begin{equation}
        \label{eq:cayley1}
        \iota(v)\iota(u)\Phi\wedge\iota(w)\iota(u)\Phi\wedge\Phi
        = \iota(u)\iota(v)\Phi\wedge\iota(w)\iota(v)\Phi\wedge\Phi
        =0,
      \end{equation}
      then, for all $x\in W$, we have
      \begin{equation}
        \label{eq:cayley2}
        \begin{split}
          &\iota(w)\iota(u)\Phi\wedge\iota(x)\iota(u)\Phi\wedge\Phi=0 \\
          &\qquad\qquad
          \;\iff\;
          \iota(w)\iota(v)\Phi\wedge\iota(x)\iota(v)\Phi\wedge\Phi=0.
        \end{split}
      \end{equation}}\right.
  \end{equation}
  If this holds, then the compatible inner product 
  is uniquely determined by $\Phi$.
\end{theorem}

\begin{proof}
  See page~\pageref{proof:CAYLEY}.
\end{proof}

To understand condition~\eqref{Eq_C} geometrically, assume $\Phi$
satisfies~\eqref{eq:uvwPhi} for some inner product on $W$. %
Then $\iota(v)\iota(u)\Phi\wedge\iota(w)\iota(u)\Phi\wedge\Phi=0$ if
and only if $\abs{u}^2\inner{v}{w}-\inner{v}{u}\inner{u}{w}=0$. %
Hence, if $u$, $v$, $w$ are linearly independent,
equation~\eqref{eq:cayley1} asserts that $w$ is orthogonal to $u$ and
$v$. %
Under this assumption both conditions in~\eqref{eq:cayley2} assert
that $w$ and $x$ are orthogonal.

Every Cayley-form $\Phi$ induces two orientations on~$W$. %
First, since the $8$--form
$\iota(v)\iota(u)\Phi\wedge\iota(v)\iota(u)\Phi\wedge\Phi$ is nonzero
for every linearly independent pair $u,v\in W$ and the space of
linearly independent pairs in $W$ is connected, there is a unique
orientation on $W$ such that
$\iota(v)\iota(u)\Phi\wedge\iota(v)\iota(u)\Phi\wedge\Phi>0$ whenever
$u,v\in W$ are linearly independent. %
The second orientation of $W$ is induced by the $8$--form
$\Phi\wedge\Phi$. %
This leads to the following definition.

\begin{definition}
  \label{def:cayleypos}
  A Cayley-form $\Phi\in\Lambda^4W^*$ is called \defined{positive} if the
  $8$--forms $\Phi\wedge\Phi$ and
  $\iota(v)\iota(u)\Phi\wedge\iota(v)\iota(u)\Phi\wedge\Phi$ induce
  the same orientation whenever $u,v\in W$ are linearly
  independent. %
  It is called \defined{negative} if it is not positive.
\end{definition}

Thus $\Phi$ is negative if and only if $-\Phi$ is positive. %
Moreover, it follows from \autoref{rmk:Worient} 
that a Cayley-form $\Phi\in\Lambda^4W^*$ is positive 
if and only if the associated triple cross product is positive.

\begin{theorem}
  \label{thm:CAYLEY1}
  If $\Phi,\Psi\in\Lambda^4W^*$ are two positive Cayley-forms, then
  there is an automorphism $g\in\Aut(W)$ such that $g^*\Phi=\Psi$.
\end{theorem}

\begin{proof}
  See page~\pageref{proof:CAYLEY1}.
\end{proof}

\begin{lemma}
  \label{le:inner}
  Let $W$ be a real vector space and 
  $g\co W^4\to\R$ be a multi-linear map satisfying
  \begin{equation}
    \label{eq:g1}
    g(u,v;w,x)
    = g(w,x;u,v)
    = - g(v,u;w,x)
  \end{equation}
  for all $u,v,w,x\in W$ and 
  \begin{equation}
    \label{eq:g2}
    g(u,v;u,v) > 0
  \end{equation}
  whenever $u,v\in W$ are linearly independent. %
  Then the matrices
  \begin{equation*}
    \Lambda_u(v,w)
    := 
    \begin{pmatrix}
      g(u,v;u,v) & g(u,v;u,w) \\
      g(u,w;u,v) & g(u,w;u,w)
    \end{pmatrix}
    \in\R^{2\times 2}, \qand
  \end{equation*}
  \begin{equation*}
    A(u,v,w)
    :=
    \begin{pmatrix}
      g(v,w;v,w) & g(v,w;w,u) & g(v,w;u,v) \\
      g(w,u;v,w) & g(w,u;w,u) & g(w,u;u,v) \\
      g(u,v;v,w) & g(u,v;w,u) & g(u,v;u,v)
    \end{pmatrix} \in\R^{3\times 3}
  \end{equation*}
  are positive definite whenever $u,v,w\in W$ are linearly
  independent. %
  Moreover, the following are equivalent:
  \begin{enumerate}[(i)]
  \item 
    \label{Lem_Inner1}
    If $u,v,w$ are linearly independent and $g(u,v;w,u)=g(v,w;u,v)=0$,
    then, for all $x\in W$, we have
    \begin{equation}
      \label{eq:vanish}
      g(u,w;u,x) = 0
      \qquad\iff\qquad
      g(v,w;v,x) = 0.
    \end{equation}

  \item 
    \label{Lem_Inner2}
    If $u,v,w$ and $u,v,w'$ are linearly independent, then
    \begin{equation}
      \label{eq:detLa}
      \frac{\det(\Lambda_u(v,w))}{\det(\Lambda_v(u,w))}
      = \frac{\det(\Lambda_u(v,w'))}{\det(\Lambda_v(u,w'))}.
    \end{equation}

  \item 
    \label{Lem_Inner3}
    If $u,v,w$ and $u,v',w'$ are linearly independent, then
    \begin{equation}
      \label{eq:detLA}
      \frac{\det(\Lambda_u(v,w))}{\sqrt{\det(A(u,v,w))}}
      = \frac{\det(\Lambda_u(v',w'))}{\sqrt{\det(A(u,v',w'))}}.
    \end{equation}

  \item 
    \label{Lem_Inner4}
    There is an inner product on $W$ such that
    \begin{equation}
      \label{eq:inner}
      g(u,v;u,v)
      = \abs{u}^2\abs{v}^2 - \inner{u}{v}^2
    \end{equation} 
    for all $u,v\in W$.
  \end{enumerate}
  If these equivalent conditions are satisfied, then the inner product
  in~\itref{Lem_Inner4} is uniquely determined by $g$ and it satisfies
  \begin{equation}
    \label{eq:detA}
    \begin{split}
      \det(\Lambda_u(v,w)) 
      &= 
        \abs{u}^2\abs{u\wedge v\wedge w}^2,\\
      \det(A(u,v,w)) 
      &= 
        \abs{u\wedge v\wedge w}^4.
    \end{split}
  \end{equation}
\end{lemma}

\bigbreak

\begin{proof}  
  Let $u,v,w\in W$ be linearly independent. %
  We prove that the matrices $\Lambda_u(v,w)$ and $A(u,v,w)$ are
  positive definite. %
  By~\eqref{eq:g2} they have positive diagonal entries.  %
  Since the determinant of $\Lambda_u(v,w)$ agrees with the
  determinant of the lower right $2\times 2$ block of $A(u,v,w)$, it
  suffices to prove that both matrices have positive determinants. %
  To see this, we observe that the determinants of $\Lambda_u(v,w)$
  and $A(u,v,w)$ remain unchanged if we add to $v$ a multiple of $u$
  and to $w$ a linear combination of $u$ and $v$. %
  With the appropriate choices both matrices become diagonal and thus
  have positive determinants. %
  Hence, $\Lambda_u(v,w)$ and $A(u,v,w)$ are positive definite, as
  claimed.

  We prove that~\itref{Lem_Inner4} implies~\eqref{eq:detA}. %
  The matrix $\Lambda_u(v,w)$ and
  $\abs{u\wedge v\wedge w}^2$ remain unchanged if we add to $v$ and
  $w$ multiples of $u$. %
  Hence, we may assume that $v$ and $w$ are orthogonal to $u$. %
  In this case
  \begin{equation*}
    \Lambda_u(v,w)
    =
    \abs{u}^2
    \begin{pmatrix}
      \abs{v}^2 & \inner{v}{w} \\
      \inner{w}{v} & \abs{w}^2
    \end{pmatrix}
  \end{equation*}
  and this implies the first equation in~\eqref{eq:detA}. %
  Since the determinant of the matrix $A(u,v,w)$ remains unchanged if
  we add to $v$ a multiple of $u$ and to $w$ a linear combination of
  $u$ and $v$, we may assume that $u,v,w$ are pairwise orthogonal. %
  In this case the second equation in~\eqref{eq:detA} is obvious. %
  Thus we have proved that~\itref{Lem_Inner4}
  implies~\eqref{eq:detA}. %
  By~\eqref{eq:detA} the inner product is uniquely
  determined by $g$.

  We prove that~\itref{Lem_Inner1} implies~\itref{Lem_Inner2}. %
  Fix two linearly independent vectors $u,v\in W$. %
  Then the subspace
  \begin{equation*}
    W_{u,v}
    := \left\{
      w\in W : g(u,v;w,u)=g(v,w;u,v)=0
    \right\}
  \end{equation*}
  has codimension two and $W=W_{u,v}\oplus\mathrm{span}\{u,v\}$. %
  Now fix an element $w\in W_{u,v}$.  %
  Then~\eqref{eq:vanish} asserts that the linear functionals
  $x\mapsto g(u,w;u,x)$ and $x\mapsto g(v,w;v,x)$ on $W$ have the same
  kernel. %
  Hence, there exists a constant $\lambda\in\R$ such that
  $g(v,w;v,x)=\lambda g(u,w;u,x)$ for all $x\in W$.  %
  With $x=w$ we obtain $\lambda=g(v,w;v,w)/g(u,w;u,w)$ and hence
  \begin{equation*}
    g(u,w;u,x)g(v,w;v,w) = g(u,w;u,w)g(v,w;v,x)
    \qquad\mbox{for all }x\in W.
  \end{equation*}
  This equation asserts that the differential of the map
  \begin{equation*}
    W_{u,v}\setminus\{0\}\to\R \co
    w\mapsto \frac{g(u,w;u,w)}{g(v,w;v,w)}
  \end{equation*}
  vanishes 
  and so the map is constant. %
  This proves~\eqref{eq:detLa} for all
  $w,w'\in W_{u,v}\setminus\{0\}$. %
  Since adding to $w$ a linear combination of $u$ and $v$ does not
  change the determinants of $\Lambda_u(v,w)$ and $\Lambda_v(u,w)$,
  equation~\eqref{eq:detLa} continues to hold for all $w,w'\in W$ that
  are linearly independent of~$u$ and~$v$. %
  Thus we have proved that~\itref{Lem_Inner1}
  implies~\itref{Lem_Inner2}.

  We prove that~\itref{Lem_Inner2} implies~\itref{Lem_Inner3}. %
  It follows from~\eqref{eq:detLa} that
  \begin{equation*}
    w,w'\in W_{u,v}\setminus\{0\}
      \qquad\implies\qquad
    \frac{g(u,w;u,w)}{g(v,w;v,w)}
    = \frac{g(u,w';u,w')}{g(v,w';v,w')}.
  \end{equation*}
  Using this identity with $u$ replaced by $u+v$ we obtain
  \begin{equation*}
    w,w'\in W_{u,v}\setminus\{0\}
      \qquad\implies\qquad
    \frac{g(u,w;v,w)}{g(v,w;v,w)}
    = \frac{g(u,w';v,w')}{g(v,w';v,w')}.
  \end{equation*}
  Now let $w,w'\in W_{u,v}$ and assume that $g(u,w;v,w)=0$. %
  Then we also have $g(u,w';v,w')=0$ and 
  so it follows from the definition of $W_{u,v}$ that
  all off-diagonal terms in the matrices 
  $\Lambda_u(v,w)$, $\Lambda_u(v,w')$, $A(u,v,w)$, and $A(u,v,w')$ vanish.  %
  Hence,
  \begin{align*}
    \frac{\det(\Lambda_u(v,w))^2}{\det(A(u,v,w))}
    &= 
      \frac{g(u,v;u,v)g(u,w;u,w)}{g(v,w;v,w)} \\
    &= 
      \frac{g(u,v;u,v)g(u,w';u,w')}{g(v,w';v,w')} 
      =
      \frac{\det(\Lambda_u(v,w'))^2}{\det(A(u,v,w'))}.
  \end{align*}
  Thus we have proved~\eqref{eq:detLA} under the assumption
  that
  $w,w'\in W_{u,v}\setminus\{0\}$ and $g(w,u;w,v)=0$. %
  Since the determinants of $\Lambda_u(v,w)$ and $A(u,v,w)$ remain
  unchanged if we add to $w$ a linear combination of $u$ and $v$ and
  if we add to $v$ a multiple of $u$, equation~\eqref{eq:detLA}
  continues to hold when $v=v'$. %
  If $u,v,w$ and $u,v,w'$ and $u,v',w'$ are all linearly independent
  triples we obtain
  \begin{equation*}
    \frac{\det(\Lambda_u(v,w))^2}{\det(A(u,v,w))}
    = \frac{\det(\Lambda_u(v,w'))^2}{\det(A(u,v,w'))}
    = \frac{\det(\Lambda_u(v',w'))^2}{\det(A(u,v',w'))}.
  \end{equation*}
  Here the last equation follows from the first 
  by symmetry in $v$ and $w$. %
  This proves equation~\eqref{eq:detLA} under the additional assumption 
  that $u,v,w'$ is a linearly independent triple. %
  This assumption can be dropped by continuity. %
  Thus we have proved that~\itref{Lem_Inner2}
  implies~\itref{Lem_Inner3}.

  We prove that~\itref{Lem_Inner3} implies~\itref{Lem_Inner4}. %
  Define a function $W\to[0,\infty)\co u\mapsto\abs{u}$ by
  $\abs{u}:=0$ for $u=0$ and by
  \begin{equation}
    \label{eq:Wnorm}
    \abs{u}^2
    := \frac{g(u,w;u,w)g(u,v;u,v)-g(u,v;u,w)^2}{\sqrt{\det(A(u,v,w))}}
  \end{equation}
  for $u\ne 0$, where $v,w\in W$ are chosen such that $u,v,w$ are
  linearly independent. %
  By~\eqref{eq:detLA} the right hand side of~\eqref{eq:Wnorm} is
  independent of $v$ and $w$. %
  It follows from~\eqref{eq:Wnorm} with $u$ replaced by $u+v$ that
  \begin{equation*}
    \abs{u+v}^2-\abs{u}^2-\abs{v}^2
    = 2\frac{g(u,w;v,w)g(u,v;u,v)-g(u,v;u,w)g(u,v;v,w)}{\sqrt{\det(A(u,v,w))}}.
  \end{equation*}
  Replacing $v$ by $-v$ gives
  ${\abs{u+v}^2+\abs{u-v}^2 = 2\abs{u}^2+2\abs{v}^2}$. %
  Thus the map $W\to[0,\infty)\co u\mapsto\abs{u}$ is continuous,
  satisfies the parallelogram identity, and vanishes only for $u=0$. %
  Hence, it is a norm on $W$ and the associated inner product of two
  linearly independent vectors $u,v\in W$ is given by
  \begin{equation}
    \label{eq:Winner}
    \inner{u}{v}
    := \frac{g(u,w;v,w)g(u,v;u,v)-g(u,v;u,w)g(u,v;v,w)}
    {\sqrt{\det(A(u,v,w))}}
  \end{equation}
  whenever $w\in W$ is chosen such that $u,v,w$ are linearly
  independent.  %
  That this inner product satisfies~\eqref{eq:inner} for every pair of
  linearly independent vectors follows from~\eqref{eq:Wnorm}
  and~\eqref{eq:Winner} with $w\in W_{u,v}$.  %
  This proves that~\itref{Lem_Inner3} implies~\itref{Lem_Inner4}.

  We prove that~\itref{Lem_Inner4} implies~\itref{Lem_Inner1}. %
  Replacing $v$ in equation~\eqref{eq:inner} by $v+w$ we obtain
  \begin{equation*}
    g(u,v;u,w) = \abs{u}^2\inner{v}{w} - \inner{u}{v}\inner{u}{w}.
  \end{equation*}
  for all $u,v,w\in W$. %
  Hence,
  \begin{equation*}
    g(u,v;w,u)=g(v,w;u,v)=0\qquad\iff\qquad \inner{u}{w}=\inner{v}{w}=0.
  \end{equation*}
  If $w\in W$ is orthogonal to $u$ and $v$, then we have
  $g(u,w;u,x)=\abs{u}^2\inner{w}{x}$ and
  $g(v,w;v,x)=\abs{v}^2\inner{w}{x}$. %
  This implies~\eqref{eq:vanish} and proves \autoref{le:inner}.
\end{proof}

\begin{proof}[Proof of \autoref{thm:CAYLEY}]
  \label{proof:CAYLEY}
  If $\Phi$ is nondegenerate and $u\in W$ is nonzero, then
  $\iota(u)\Phi$ descends to a nondegenerate $3$--form on the
  $7$--dimensional quotient space $W/\R u$. %
  By \autoref{le:imO} this implies that
  $\iota(v)\iota(u)\Phi\wedge\iota(v)\iota(u)\Phi\wedge\iota(u)\Phi$
  descends to a nonzero $7$--form on $W/\R u$ for every vector
  $v\in W\setminus\R u$. %
  Hence, the $8$--form
  $\iota(v)\iota(u)\Phi\wedge\iota(v)\iota(u)\Phi\wedge\Phi$ on $W$ is
  nonzero whenever $u,v$ are linearly independent. %
  The orientation on $W$ induced by this form is independent of the
  choice of the pair~${u,v}$. %
  Choose any volume form $\Omega\in\Lambda^8W^*$ compatible with this
  orientation and, for $\lambda>0$, define a multi-linear function
  $g_\lambda\co W^4\to\R$ by
  \begin{equation}
    \label{eq:gla}
    g_\lambda(u,v;w,x)
    := \frac{\iota(v)\iota(u)\Phi\wedge\iota(x)\iota(w)\Phi\wedge\Phi}
    {6\lambda^4\Omega}
  \end{equation}
  This function satisfies~\eqref{eq:g1} and~\eqref{eq:g2} and, if
  $\Phi$ satisfies~\eqref{Eq_C}, it also satisfies~\eqref{eq:vanish}. %
  Hence, it follows from \autoref{le:inner} that there is a unique
  inner product $\inner{\cdot}{\cdot}_\lambda$ on $W$ such that, for
  all $u,v\in W$, we have
  \begin{equation}
    \label{eq:glambda}
    g_\lambda(u,v;u,v) 
    = \abs{u}_\lambda^2\abs{v}_\lambda^2 
    - \inner{u}{v}_\lambda^2.
  \end{equation}
  Let $\dvol_\lambda$ be the volume form associated to the inner
  product and the orientation. %
  Then there is a constant $\mu(\lambda)>0$ such that
  $$
  \dvol_\lambda=\mu(\lambda)^2\Omega. 
  $$
  We have $g_\lambda=\lambda^{-4}g_1$, hence
  $\abs{u}_\lambda=\lambda^{-1}\abs{u}_1$ for every $u\in W$, and
  hence $\dvol_\lambda=\lambda^{-8}\dvol_1$. %
  Thus $\mu(\lambda)=\lambda^{-4}\mu(1)$. %
  With $\lambda:=\mu(1)^{1/6}$ we obtain
  $\mu(\lambda)=\lambda^{-4}\mu(1)=\mu(1)^{1/3}=\lambda^2$. %
  With this value of $\lambda$ we have
  $\lambda^4\Omega=\dvol_\lambda$. %
  Hence, it follows from~\eqref{eq:gla} and~\eqref{eq:glambda} that
  \begin{equation*}
    \iota(v)\iota(u)\Phi\wedge\iota(v)\iota(u)\Phi\wedge\Phi
    = 6\left(\abs{u}_\lambda^2\abs{v}_\lambda^2 
      - \inner{u}{v}_\lambda^2\right)\dvol_\lambda.
  \end{equation*}
  Hence, by \autoref{le:cayley}, $\Phi$ is compatible with the inner
  product $\inner{\cdot}{\cdot}_\lambda$. %
  This shows that every $4$--form $\Phi\in\Lambda^4W^*$ that
  satisfies~\eqref{Eq_C} is compatible with a unique inner product.

  Conversely, suppose that $\Phi$ is compatible with an inner
  product. %
  Then, by \autoref{le:cayley}, there is an orientation on $W$ such
  that the associated volume form $\dvol\in\Lambda^8W^*$
  satisfies~\eqref{eq:uvPhi}. %
  Define $g\co W^4\to\R$ by
  \begin{equation*}
    g(u,v;w,x)
    := \frac{\iota(v)\iota(u)\Phi\wedge\iota(x)\iota(w)\Phi\wedge\Phi}
    {6\dvol}.
  \end{equation*}
  By~\eqref{eq:uvPhi} this map satisfies condition~\itref{Lem_Inner4}
  in \autoref{le:inner} and it obviously satisfies~\eqref{eq:g1}
  and~\eqref{eq:g2}. %
  Hence, it satisfies condition~\itref{Lem_Inner1} in
  \autoref{le:inner} and this implies that $\Phi$
  satisfies~\eqref{Eq_C}.  This proves \autoref{thm:CAYLEY}.
\end{proof}

\begin{proof}[Proof of \autoref{thm:CAYLEY1}]
  \label{proof:CAYLEY1}
  Let $\Phi\in\Lambda^4W^*$ be a positive Cayley-form with
  the associated inner product, orientation, and triple cross product. %
  Let $\phi_0\in\Lambda^3(\R^7)^*$ and $\psi_0\in\Lambda^4(\R^7)^*$ be
  the standard associative and coassociative calibrations defined in
  \autoref{ex:cross7} and in the proof of \autoref{le:psi}. %
  Then $ \Phi_0:=1^*\wedge\phi_0+\psi_0 \in\Lambda^4(\R^8)^* $ is the
  standard Cayley-form on $\R^8$.

  Choose a unit vector $e\in W$ and denote 
  \begin{equation*}
    V_e:=e^\perp,\qquad 
    \phi_e := \iota(e)\Phi|_{V_e}\in\Lambda^3V_e^*,\qquad 
    \psi_e := \Phi|_{V_e}\in\Lambda^4V_e^*. 
  \end{equation*}
  Then $\phi_e$ is a nondegenerate $3$--form on $V_e$ and, hence, by
  \autoref{thm:imO}, there is an isomorphism $g\co \R^7\to V_e$ such
  that $g^*\phi_e = \phi_0$. %
  It follows also from \autoref{thm:imO} that $g$ identifies the
  standard inner product on $\R^7$ with the unique inner product on
  $V_e$ that is compatible with $\phi_e$, and the standard orientation
  on $\R^7$ with the orientation determined by $\phi_e$ via
  \autoref{le:imO}. %
  Hence, it follows from \autoref{le:psi} that $g$ also identifies
  the two coassociative calibrations, i.e., $g^*\psi_e = \psi_0$.  %
  Since $\Phi$ is a positive Cayley-form, we have
  $$
  \Phi=e^*\wedge\phi_e+\psi_e.  
  $$
  Hence, if we extend $g$ to an
  isomorphism $\R^8=\R\oplus\R^7\to W$, which is still denoted by $g$
  and sends $e_0=1\in\R\subset\R^8$ to $e$, we obtain
  $g^*\Phi = \Phi_0$ and this proves \autoref{thm:CAYLEY1}.
\end{proof}

\begin{remark}
  The space $S^2\Lambda^2W^*$ of symmetric bilinear forms on
  $\Lambda^2W$ can be identified with the space of multi-linear maps
  $g\co W^4\to\R$ that satisfy~\eqref{eq:g1}. %
  Denote by $S^2_0\Lambda^2W^*\subset S^2\Lambda^2W^*$ the subspace of
  all $g\in S^2\Lambda^2W^*$ that satisfy the algebraic Bianchi
  identity
  \begin{equation}
    \label{eq:bianchi}
    g(u,v;w,x)+g(v,w;u,x)+g(w,u;v,x)=0
  \end{equation}
  for all $u,v,w,x\in W$. %
  Then there is a direct sum decomposition
  \begin{equation*}
    S^2\Lambda^2W^* = \Lambda^4W^*\oplus S^2_0\Lambda^2W^*
  \end{equation*}
  and the projection 
  \begin{equation*}
    \Pi\co S^2\Lambda^2 W^*\to\Lambda^4W^*
  \end{equation*}
  is given by 
  \begin{equation*}
    (\Pi g)(u,v,w,x) 
    := \tfrac{1}{3}\bigl(g(u,v;w,x)+g(v,w;u,x)+g(w,u;v,x)\bigr).
  \end{equation*}
  Note that 
  \begin{gather}
    \dim\,\Lambda^2W = 28,\qquad \dim\,S^2\Lambda^2W = 406, \\
    \dim\,\Lambda^4W = 70,\qquad \dim\,S^2_0\Lambda^2W = 336.
  \end{gather}
  Moreover, there is a natural quadratic map
  $q^\Lambda\co S^2W^*\to S^2_0\Lambda^2W^*$ given by 
  \begin{equation*}
    \bigl(q^\Lambda(\gamma)\bigr)(u,v;x,y) 
    := \gamma(u,x)\gamma(v,y)-\gamma(u,y)\gamma(v,x)
  \end{equation*}
  for $\gamma\in S^2W^*$ and $u,v,x,y\in W$. %
  \autoref{le:inner} asserts, in particular, that the restriction of
  this map to the subset of inner products is injective and, for each
  element $g\in S^2\Lambda^2W^*$, it gives a necessary and sufficient
  condition for the existence of an inner product $\gamma$ on $W$ such
  that
  \begin{equation*}
    g-\Pi g=q^\Lambda(\gamma).  
  \end{equation*}
  We shall see in \autoref{cor:uvxy} below that, if
  $\Phi\in\Lambda^4W^*$ is a positive Cayley-form and
  $g=g_\Phi\in S^2\Lambda^2W^*$ is given by
  \begin{equation*}
    g_\Phi(u,v;x,y)
    := 
      \frac{\iota(v)\iota(u)\Phi\wedge\iota(y)\iota(x)\Phi\wedge\Phi}
      {\dvol},\qquad
    \dvol
    :=
      \frac{\Phi\wedge\Phi}{14},
  \end{equation*}
  then 
  \begin{equation*}
    g_\Phi=6q^\Lambda(\gamma)+7\Phi
  \end{equation*}
  for a unique inner product $\gamma\in S^2W^*$, 
  and the volume form of $\gamma$ is indeed $\dvol$. %
  Thus, in particular, we have $\Pi g_\Phi=7\Phi$. 
\end{remark}

\begin{remark}
  The space $S^2S^2W^*$ of symmetric bilinear forms on $S^2W$ can be identified with the space of multi-linear maps $\sigma\co W^4\to\R$ that satisfy
  \begin{equation}
    \label{eq:g3}
    \sigma(u,v;x,y)
    = \sigma(x,y;u,v)
    = \sigma(v,u;x,y).
  \end{equation}
  Denote by $S^2_0S^2W^*$ the subspace of all $\sigma\in S^2S^2W^*$
  that satisfy the algebraic Bianchi identity~\eqref{eq:bianchi}.
  Then 
  \begin{equation*}
    S^2S^2W^*=S^4W^*\oplus S^2_0S^2W^*,
  \end{equation*}
  where
  \begin{gather}
    \dim\,S^2W = 36,\qquad \dim\,S^2S^2W = 666, \\
    \dim\,S^4W = 330,\qquad \dim\,S^2_0S^2W = 336.
  \end{gather}
  The projection $\Pi\co S^2S^2 W^*\to S^4W^*$ is given by 
  the same formula as above. %
  Thus 
  \begin{equation*}
    (\sigma-\Pi\sigma)(u,v;x,y)
    = \tfrac{2}{3}\sigma(u,v;x,y) - \tfrac{1}{3}\sigma(v,x;u,y) - \tfrac{1}{3}\sigma(x,u;v,y).
  \end{equation*}
  There is a natural quadratic map $q^S\co S^2W^*\to S^2S^2W^*$ given by
  \begin{equation*}
    \bigl(q^S(\gamma)\bigr)(u,v;x,y):=\gamma(u,v)\gamma(x,y).
  \end{equation*}
  Polarizing the quadratic map
  $q^\Lambda\co S^2W^*\to S^2\Lambda^2W^*$ one obtains a linear map
  $T\co S^2S^2W^*\to S^2\Lambda^2W^*$ given by
  \begin{equation*}
    \bigl(T\sigma\bigr)(u,v;x,y) := \sigma(u,x;v,y)-\sigma(u,y;v,x)
  \end{equation*}
  such that $q^\Lambda=T\circ q^S$. %
  The image of $T$ is the subspace $S^2_0\Lambda^2W^*$ of solutions of
  the algebraic Bianchi identity~\eqref{eq:bianchi} and its kernel is
  the subspace $S^4W^*$. %
  A pseudo-inverse of $T$ is the map $S\co S^2\Lambda^2W^*\to S^2S^2W^*$
  given by
  \begin{equation*}
    \bigl(Sg\bigr)(u,v;x,y) := \tfrac{1}{3}\bigl(g(u,x;v,y)+g(u,y;v,x)\bigr)
  \end{equation*}
  whose kernel is $\Lambda^4W^*$ and whose image is $S^2_0S^2W^*$. %
  Thus
  \begin{equation*}
    TSg = g-\Pi g,\qquad ST\sigma = \sigma-\Pi\sigma
  \end{equation*}
  for $g\in S^2\Lambda^2W^*$ and $\sigma\in S^2S^2W^*$. %
  Given $g\in S^2\Lambda^2W^*$ and $\gamma\in S^2W^*$, we have
  \begin{equation*}
    g-\Pi g=q^\Lambda(\gamma)\qquad\iff\qquad Sg=(\one-\Pi)q^S(\gamma).
  \end{equation*}
  Namely, if $q^\Lambda(\gamma)=g-\Pi g$, then
  $Sg= S(g-\Pi g)=Sq^\Lambda(\gamma)=q^S(\gamma)-\Pi q^S(\gamma)$, and
  if $(\one-\Pi)q^S(\gamma)=Sg$, then
  $(\one-\Pi)g=TSg=T(\one-\Pi)q^S(\gamma) =q^\Lambda(\gamma)$.
\end{remark}


    
\section{The group \texorpdfstring{$\Gtwo$}{G2}}
\label{sec:G2}  

Let $V$ be a $7$--dimensional real Hilbert space equipped with a cross
product and let $\phi\in\Lambda^3V^*$ be the associative calibration
defined by~\eqref{eq:phi}. %
We orient $V$ as in \autoref{le:imO} and denote by
$*\co \Lambda^kV^*\to\Lambda^{7-k}V^*$ the associated Hodge
$*$--operator and by $\psi:=*\phi\in\Lambda^4V^*$ the coassociative
calibration. %
Recall that $V$ is equipped with an associator bracket
via~\eqref{eq:associator}, related to $\psi$ via~\eqref{eq:psi}, and
with a coassociator bracket~\eqref{eq:coass}.

The group of automorphisms of $\phi$ 
will be denoted by 
\begin{equation*}
  \rG(V,\phi) 
  := \left\{
    g\in\GL(V) : g^*\phi=\phi
  \right\}.
\end{equation*}
By \autoref{le:Vphi}, we have $\rG(V,\phi)\subset\SO(V)$
and hence, by~\eqref{eq:phi}, 
\begin{equation*}
  \rG(V,\phi) 
  = \left\{
    g\in\SO(V) : 
    gu\times gv=g(u\times v)\;\forall\,u,v\in V
  \right\}.
\end{equation*}
For the standard structure $\phi_0$ on $\R^7$ in \autoref{ex:cross7}
we denote the structure group by $\Gtwo:=\rG(\R^7,\phi_0)$. %
By \autoref{thm:imO}, the group $\rG(V,\phi)$ is isomorphic to $\Gtwo$
for every nondegenerate $3$--form on a $7$--dimensional vector space.

\begin{theorem}
  \label{thm:G2}
  The group $\rG(V,\phi)$ is a $14$--dimensional simple, connected,
  simply connected Lie group. %
  It acts transitively on the unit sphere and, for every unit vector
  $u\in V$, the isotropy subgroup
  $ \rG_u:=\left\{g\in\rG(V,\phi) : gu=u\right\} $ is isomorphic to
  $\SU(3)$. %
  Thus there is a fibration
  \begin{equation*}
    \SU(3)\hookrightarrow \Gtwo \longrightarrow S^6. 
  \end{equation*}
\end{theorem}

\begin{proof}
  As we have observed in \autoref{Pf_ImO4} in the proof of
  \autoref{le:imO}, the group $\rG=\rG(V,\phi)$ has dimension at least
  $14$, as it is an isotropy subgroup of the action of the
  $49$--dimensional group $\GL(V)$ on the $35$--dimensional
  space~$\Lambda^3V^*$. %
  Since $\rG\subset\SO(V)$, by \autoref{le:Vphi}, the group acts on
  the unit sphere
  \begin{equation*}
    S:=\left\{u\in V : \abs{u}=1\right\}.
  \end{equation*}
  Thus, for every $u\in S$, the isotropy 
  subgroup $\rG_u$ has dimension at least~$8$. %
  By \autoref{le:phisymp}, the group $\rG_u$ preserves 
  the subspace $W_u:=u^\perp$, the symplectic form 
  $\om_u$, and the complex structure $J_u$ on $W_u$ given 
  by $\om_u(v,w)=\inner{u}{v\times w}$ and $J_uv=u\times v$. %
  Hence, $\rG_u$ is isomorphic to a subgroup 
  of $\U(W_u,\om_u,J_u)\cong\U(3)$. %
  Now consider the complex valued $3$--form
  $\theta_u\in\Lambda^{3,0}W_u^*$ given by 
  \begin{equation*}
    \theta_u(x,y,z)
    :=
      \phi(x,y,z)-\i\phi(u\times x,y,z)
    =
      \phi(x,y,z)-\i\psi(u,x,y,z)
  \end{equation*}
  for $x,y,z\in W_u$. %
  (See~\eqref{eq:associator} and~\eqref{eq:psi} for the last
  equality.)  %
  This form is nonzero and is preserved by $\rG_u$. %
  Hence, $\rG_u$ is isomorphic to a subgroup of
  $\SU(W_u,\om_u,J_u)$. %
  Since $\SU(W_u,\om_u,J_u)\cong\SU(3)$ is a connected Lie group of
  dimension $8$ and $\rG_u$ has dimension at least $8$, it follows
  that
  \begin{equation*}
    \rG_u\cong \SU(W_u,\om_u,J_u)\cong\SU(3).
  \end{equation*}
  In particular, $\dim\rG_u=8$ and so
  $\dim\rG\le\dim\rG_u+\dim\,S=14$. %
  This implies $\dim\rG=14$ and, since $S$ is connected, $\rG$ acts
  transitively on $S$. %
  Thus we have proved that there is a fibration
  $\SU(3)\hookrightarrow\rG\to S$. %
  It follows from the homotopy exact sequence of this fibration that
  $\rG$ is connected and simply connected and that
  $\pi_3(\rG)\cong\Z$. %
  Hence, $\rG$ is simple.

  Here is another proof that $\rG$ is simple. %
  Let $\g:=\Lie(\rG)$ denote its Lie algebra and, for every $u\in S$,
  let $\g_u:=\Lie(\rG_u)$ denote the Lie algebra of the isotropy
  subgroup. %
  Then, for every $\xi\in\g$, we have $\xi\in\g_u$ if and only if
  $u\in\ker\xi$.  %
  Since every $\xi\in\g$ is skew-adjoint, it has a nontrivial kernel
  and hence belongs to $\g_u$ for some $u\in S$.

  Now let $I\subset\g$ be a nonzero ideal. %
  Then, by what we have just observed, there is an element $u\in S$
  such that $I\cap\g_u\ne\{0\}$. %
  Thus $I\cap\g_u$ is a nonzero ideal in $\g_u$ and, since $\g_u$ is
  simple, this implies $\g_u\subset I$.  %
  Next we claim that, for every $v\in u^\perp$, there is an element
  $\xi\in I$ such that $\xi u=v$. %
  To see this, choose any element $\eta\in\g_u\subset I$ such that
  $\ker\eta=\langle u\rangle$. %
  Then there is a unique element $w\in u^\perp$ such that $\eta
  w=v$. %
  Since $\rG$ acts transitively on~$S$ there is an element
  $\zeta\in\g$ such that $\zeta u=w$. %
  Hence, $\xi=[\eta,\zeta]\in I$ and $\xi u=\eta\zeta u=\eta w=v$.
  This proves that $\dim(I/\g_u)\ge 6$; hence, $\dim\,I\ge 14$, and
  hence $I=\g$. %
  This proves \autoref{thm:G2}.
\end{proof}

We examine the action of the group $\rG(V,\phi)$ on the space
\begin{equation*}
  \sS:=\left\{(u,v,w)\in V\,:\,
    \begin{array}{l}
      \abs{u}=\abs{v}=\abs{w}=1,\\
      \inner{u}{v}=\inner{u}{w}=\inner{v}{w}=\inner{u\times v}{w}=0
    \end{array}
  \right\}.
\end{equation*}
Let $S \subset V$ denote the unit sphere. %
Then each tangent space $T_u S = u^\perp$ carries a natural complex
structure $v \mapsto u \times v$. %
The space $\sS$ is a bundle over $S$ whose fiber over $u$ is the space
of Hermitian orthonormal pairs in $T_u S$. %
Hence, $\sS$ is a bundle of $3$--spheres over a bundle of $5$--spheres
over a $6$--sphere and therefore is a compact connected simply
connected $14$--dimensional manifold.

\begin{theorem}
  \label{thm:G2S}
  The group $\rG(V,\phi)$ acts freely and transitively on $\sS$.
\end{theorem}

\begin{proof}
  We give two proofs of this result. %
  The first proof uses the fact that the isotropy subgroup
  $\rG_u\subset\rG:=\rG(V,\phi)$ of a unit vector $u\in V$ is
  isomorphic to $\SU(3)$ and the isotropy subgroup in $\SU(3)$ of a
  Hermitian orthonormal pair is the identity. %
  Hence, $\rG$ acts freely on $\sS$. %
  Since $\rG$ and $\sS$ are compact connected manifolds of the same
  dimension, this implies that $\rG$ acts transitively on $\sS$.

  For the second proof we assume that $\phi = \phi_0$ is the standard
  structure on $V=\R^7$. %
  Given $(u,v,w) \in \sS$, define $g\co \R^7\to\R^7$ by
  \begin{gather*}
    g e_1 = u, \quad
    g e_2 = v, \quad
    g e_3 = u \times v, \quad
    g e_4 = w \\
    g e_5 = w \times u, \quad
    g e_6 = w \times v, \quad
    g e_7 = w \times (u \times v).
  \end{gather*}
  By construction $g$ preserves the cross product and the inner
  product. %
  Hence, $g \in \Gtwo$. %
  Moreover, $g$ is the unique element of $\Gtwo$ that maps the triple
  $(e_1,e_2,e_4)$ to $(u,v,w)$. %
  This proves \autoref{thm:G2S}.
\end{proof}

\begin{cor}
  \label{cor:assoctrans}
  The group $\rG(V,\phi)$ acts transitively on the space of
  associative subspaces of $V$ and on the space of coassociative
  subspaces of $V$.
\end{cor}

\begin{proof}
  This follows from \autoref{thm:G2S}, \autoref{le:assoc}, and
  \autoref{le:coassoc}.
\end{proof}

\begin{remark}
  \label{rmk:assocgrass}
  Let $\Lambda\subset V$ be an associative subspace and define
  $H:=\Lambda^\perp$ and
  ${\rG_\Lambda:=\{g\in\rG(V,\phi):g\Lambda=\Lambda\}}$.  %
  Then every $h\in\SO(H)$ extends uniquely to an element
  $g\in\rG_\Lambda$ (choose $(u,v,w)\in\sS$ such that $u,v,w\in H$)
  and the action of $g$ on $\Lambda$ is induced by the action of $h$
  on $\Lambda^+H^*$ under the isomorphism in \autoref{rmk:LaH}.  %
  Hence the map $\rG_\Lambda\to\SO(H):g\mapsto g|_H$ is an isomorphism
  and so the \defined{associative Grassmannian}
  $\sL := \left\{\Lambda\subset V:\Lambda\mbox{ is an associative
      subspace}\right\}$
  is diffeomorphic to the homogeneous space
  $\rG(V,\phi)/\SO(H)\cong\rG_2/\SO(4)$, by
  \autoref{cor:assoctrans}. %
  Since $\Lambda \subset V$ is associative if and only if
  $H:= \Lambda^\perp$ is coassociative (see \autoref{le:coassoc}),
  $\sL$ also is the \defined{coassociative Grassmannian}.
\end{remark}

\begin{theorem}
  \label{thm:form7}
  There are orthogonal splittings
  \begin{align*}
    \Lambda^2V^* 
    &=
      \Lambda^2_7\oplus\Lambda^2_{14},\\
    \Lambda^3V^*
    &=
      \Lambda^3_1\oplus\Lambda^3_7\oplus\Lambda^3_{27},
  \end{align*}
  where $\dim\Lambda^k_d=d$ and
  \begin{align*}
    \Lambda^2_7 
    &:= 
      \left\{\iota(u)\phi : u\in V\right\} 
       = 
      \left\{\omega\in\Lambda^2V^* : 
      *(\phi\wedge\omega)=2\om\right\},\\
    \Lambda^2_{14} 
    &:= 
      \left\{\omega\in\Lambda^2V^* : \psi\wedge\omega=0\right\} 
       = 
      \left\{\omega\in\Lambda^2V^* : 
      *(\phi\wedge\omega)=-\om\right\},\\
    \Lambda^3_1
    &:=
      \langle\phi\rangle, \\
    \Lambda^3_7
    &:= 
      \left\{\iota(u)\psi : u\in V\right\},\\
    \Lambda^3_{27} 
    &:= 
      \left\{\omega\in\Lambda^3V^* : 
      \phi\wedge\omega=0,\;\psi\wedge\omega=0\right\}.
  \end{align*}
  Each of the spaces $\Lambda^k_d$ is an irreducible 
  representation of $\rG(V,\phi)$ and the representations 
  $\Lambda^2_7$ and $\Lambda^3_7$ are both isomorphic to $V$,
  $\Lambda^2_{14}$ is isomorphic to the Lie algebra 
  $\g(V,\phi):=\Lie(\rG(V,\phi))\cong\g_2$, and $\Lambda^3_{27}$
  is isomorphic to the space of traceless symmetric endomorphisms of
  $V$. %
  The orthogonal projections $\pi_7\co \Lambda^2V^*\to\Lambda^2_7$
  and $\pi_{14}\co \Lambda^2V^*\to\Lambda^2_{14}$ are given by 
  \begin{align}
    \label{eq:pi7}
    \pi_7(\om)
    &=
      \tfrac{1}{3}\om + \tfrac{1}{3}*(\phi\wedge\om) 
    =
      \tfrac{1}{3}*\bigl(\psi\wedge*(\psi\wedge\om)\bigr),\\
    \label{eq:pi14}
    \pi_{14}(\om) 
    &=
      \tfrac{2}{3}\om - \tfrac{1}{3}*(\phi\wedge\om)
    =
      \om - \tfrac{1}{3}*\bigl(\psi\wedge*(\psi\wedge\om)\bigr).
  \end{align}
\end{theorem}

\begin{proof}
  For $u\in V$ denote by $A_u\in\so(V)$ the endomorphism
  $A_uv := u\times v$. %
  Then the Lie algebra $\g :=\Lie(\rG)$ of $\rG=\rG(V,\phi)$ is given
  by
  \begin{equation*}
    \g = \left\{\xi\in\End(V) : 
      \xi+\xi^*=0,\,A_{\xi u}+[A_u,\xi]=0
      \;\forall\,u\in V\right\}.
  \end{equation*}

  \setcounter{step}{0}
  \begin{step}
    \label{Pf_form71}
    There is an orthogonal decomposition
    \begin{equation*}
      \so(V)
      =
        \g\oplus\h,\qquad
      \h
      :=
        \left\{A_u : u\in V\right\}
    \end{equation*}
    with respect to the inner product 
    $\inner{\xi}{\eta}:=-\tfrac{1}{2}\tr(\xi\eta)$ on $\so(V)$.
  \end{step}

  The group $\rG$ acts on the space $\so(V)$ of skew-adjoint
  endomorphisms by conjugation and this action preserves the inner
  product. %
  Both subspaces $\g$ and $\h$ are invariant under this action,
  because $gA_ug^{-1}=A_{gu}$ for all $u\in V$ and $g\in\rG$. %
  If $\xi=A_u\in\g\cap\h$, then $0=\cL_{A_u}\phi=3\iota(u)\psi$ (see
  equation~\eqref{eq:upsi}) and hence $u=0$. %
  This shows that $\g\cap\h=\{0\}$. %
  Since $\dim\,\g=14$, $\dim\,\h=7$, and $\dim\,\so(V)=21$, we have
  $\so(V)=\g\oplus\h$. %
  Moreover, $\g^\perp$ is another $\rG$--invariant complement of
  $\g$. %
  Hence $\h$ is the graph of a $\rG$--equivariant linear map
  $\g^\perp\to\g$. %
  The image of this map is an ideal in $\g$ and hence must be zero. %
  This shows that $\h=\g^\perp$.

  \begin{step}
    \label{Pf_form72}
    $\Lambda^2_{14}$ is the orthogonal complement of $\Lambda^2_7$
  \end{step}

  By equation~\eqref{eq:phipsi4} in \autoref{le:phipsi} we have
  $u^*\wedge\psi = *\iota(u)\phi$ for all $u\in V$. %
  Hence, $u^*\wedge\om\wedge\psi=\om\wedge*\iota(u)\phi$ and this
  proves \autoref{Pf_form72}.

  \begin{step}
    \label{Pf_form73}
    The isomorphism
    $\so(V)\to\Lambda^2V^*\co
    \xi\mapsto\om_\xi:=\inner{\cdot}{\xi\cdot}$
    is an $\SO(V)$--equivariant isometry and maps $\g$ onto
    $\Lambda^2_{14}$
  \end{step}

  That the isomorphism $\xi\mapsto\om_\xi$ is an $\SO(V)$--equivariant
  isometry follows directly from the definitions. %
  The image of $\h$ under this isomorphism is obviously the subspace
  $\Lambda^2_7$.  %
  Hence, by \autoref{Pf_form71}, the orthogonal complement of
  $\Lambda^2_7$ is the image of $\g$ under this isomorphism. %
  Hence, the assertion follows from \autoref{Pf_form72}.

  \begin{step}
    \label{Pf_form74}
    Let $\om\in\Lambda^2V^*$. %
    Then $\psi\wedge\om=0$ if and only if $*(\phi\wedge\om)=-\om$.
  \end{step}

  Define the operators $Q\co \Lambda^2V^*\to\Lambda^2V^*$ and
  $R\co \Lambda^2V^*\to\Lambda^1V^*$ by
  \begin{equation*}
    Q\om := *(\phi\wedge\om),\qquad
    R\om := *(\psi\wedge\om)
  \end{equation*}
  for $\om\in\Lambda^2V^*$. %
  Then $Q$ is self-adjoint and $R^*\co \Lambda^1V^*\to\Lambda^2V^*$ is
  given by the same formula $R^*\alpha=*(\psi\wedge\alpha)$ for
  $\alpha\in\Lambda^1V^*$. %
  Both operators are $\rG$--equivariant. %
  Moreover, $R^*R=Q+\id$ by equation~\eqref{eq:phipsi23} in
  \autoref{le:phipsi}.  %
  Hence, $R\om=0$ if and only if $Q\om=-\om$. %
  (Note also that the operator $R^*R$ vanishes on $\Lambda^2_{14}$ by
  equation~\eqref{eq:phipsi23} and has eigenvalue $3$ on $\Lambda^2_7$
  by~\eqref{eq:phipsi15}.) %
  This proves \autoref{Pf_form74}.

  One can rephrase this argument more geometrically as follows.  The
  action of $\rG$ on $\Lambda^2_{14}$ is irreducible by \autoref{Pf_form73}. %
  Hence, $\Lambda^2_{14}$ is (contained in) an eigenspace of the
  operator $Q$. %
  Moreover, the operator $Q$ is traceless. %
  To see this, let $e_1,\dots,e_7$ be an orthonormal basis of $V$ and
  denote by $e^1,\dots,e^7$ the dual basis of $V^*$. %
  Then the $2$--forms $e^{ij}:=e^i\wedge e^j$ with $i<j$ form an
  orthonormal basis of $\Lambda^2V^*$ and we have
  \begin{equation*}
    \sum_{i<j}\inner{e^{ij}}{*(\phi\wedge e^{ij})}
    = \sum_{i<j} (e^{ij}\wedge e^{ij}\wedge\phi)(e_1,\dots,e_7)
    = 0.
  \end{equation*}
  By equation~\eqref{eq:phipsi12} in \autoref{le:phipsi}, the operator
  $Q$ has eigenvalue $2$ on the $7$--dimensional subspace
  $\Lambda^2_7$. %
  Since $\dim\,\Lambda^2V^*=21$, it follows that $Q$ has eigenvalue
  $-1$ on the $14$--dimensional subspace $\Lambda^2_{14}$. %
  This gives rise to another proof of equation~\eqref{eq:phipsi23} and
  completes the second proof of \autoref{Pf_form74}.

  \begin{step}
    \label{Pf_form75}
    The subspaces $\Lambda^3_1$, $\Lambda^3_7$, and $\Lambda^3_{27}$
    form an orthogonal decomposition of $\Lambda^3V^*$ and
    $\dim\,\Lambda^3_d=d$.
  \end{step}

  That $\dim\Lambda^3_d=d$ for $d=1,7$ is obvious. %
  Since $*\iota(u)\psi = - u^*\wedge\phi$, it follows that
  $\Lambda^3_1$ is orthogonal to $\Lambda^3_7$. %
  Moreover, for every $\om\in\Lambda^3V^*$, we have
  \begin{equation*}
    \phi\wedge\om=0
      \quad\iff\quad
    u^*\wedge\phi\wedge\om=0\;\forall u\in V
      \quad\iff\quad
    \om\perp\Lambda^3_7
  \end{equation*}
  and 
  \begin{equation*}
    \psi\wedge\om=0
    \quad\iff\quad
    \om\perp\Lambda^3_1.
  \end{equation*}
  Hence, $\Lambda^3_{27}$ is the orthogonal complement of
  $\Lambda^3_1\oplus\Lambda^3_7$. %
  Since $\dim\,\Lambda^3V^*=35$, this proves \autoref{Pf_form75}.

  \begin{step}
    \label{Pf_form76}
    The subspaces $\Lambda^2_7$, $\Lambda^2_{14}$, $\Lambda^3_1$, 
    $\Lambda^3_7$, $\Lambda^3_{27}$ are irreducible representations 
    of the group $\rG=\rG(V,\phi)$.
  \end{step}

  The irreducibility of $\Lambda^3_1$ and
  $\Lambda^2_7\cong\Lambda^3_7$ is obvious and for $\Lambda^2_{14}$ it
  follows from \autoref{Pf_form73}. %
  We also point out that $\Lambda^3_7$ is the tangent space of the
  orbit of~$\phi$ under the action of~$\SO(V)$. %
  The space $\Lambda^3_{27}$ can be identified with the space of
  traceless symmetric endomorphisms ${S\co V\to V}$ via
  $S \mapsto\cL_S\phi$ by \autoref{thm:27} below. %
  That it is an irreducible representation of $\rG(V,\phi)$
  is shown in~\cite{Bryant1987}. %
  This proves \autoref{Pf_form76}. %
  Equations~\eqref{eq:pi7} and~\eqref{eq:pi14} follow 
  directly from the definitions and~\eqref{eq:phipsi23}. %
  This proves \autoref{thm:form7}.
\end{proof}

\begin{theorem}
  \label{thm:27}
  The linear map 
  \begin{equation*}
    \End(V)\to\Lambda^3V^*\co
    A\mapsto\cL_A\phi
  \end{equation*}
  (see \autoref{rmk:uphipsi}) restricts to a $\rG(V,\phi)$--equivariant 
  isomorphism from the space of traceless symmetric 
  endomorphisms of $V$ onto $\Lambda^3_{27}$.  
\end{theorem}

\begin{proof}
  We follow the exposition of Karigiannis in~\cite{Karigiannis2009a}*{Section~2}. %
  Define the linear map $\Lambda^3V^*\to\End(V)\co \eta\mapsto S_\eta$ by
  \begin{equation}\label{eq:Seta}
    \inner{u}{S_\eta v}
    := \frac{\iota(u)\phi\wedge\iota(v)\phi\wedge\eta}{4\dvol}
  \end{equation}
  for $\eta\in\Lambda^3V^*$ and $u,v\in V$. %
  This map has the following properties.

  \setcounter{step}{0}
  \begin{step}
    \label{Pf_271}
    Let $A\in\End(V)$. %
    Then
    \begin{equation}
      \label{eq:SLA}
      S_{\cL_A\phi} = \tfrac12(A^*+A) + \tfrac12\tr(A)\one.
    \end{equation}
    In particular, $S_\phi=\tfrac{3}{2}\one$.
  \end{step}

  For $t\in\R$ define 
  $g_t := e^{At}$ and $\phi_t :=g_t^*\phi$.
  Then $\phi_t\in\Lambda^3V^*$ is a nondegenerate $3$--form 
  compatible with the inner product
  \begin{equation*}
    \inner{u}{v}_t:=\inner{g_tu}{g_tv}
  \end{equation*}
  on $V$ and the volume form $\dvol_t\in\Lambda^7V^*$ given by
  \begin{equation*}
    \dvol_t := g_t^*\dvol = \det(g_t)\dvol.
  \end{equation*}
  Hence,
  \begin{equation*}
    \iota(u)\phi_t\wedge\iota(u)\phi_t\wedge\phi_t 
    = 6\abs{u}_t^2\dvol_t
  \end{equation*}
  for all $u\in V$ and all $t\in\R$. %
  Differentiate this equation with respect to $t$ at $t=0$ 
  and use the identity 
  $0=\iota(u)(\iota(u)\phi\wedge\phi\wedge\eta)
  = \iota(u)\phi\wedge\iota(u)\phi\wedge\eta
  - \iota(u)\phi\wedge\phi\wedge\iota(u)\eta$ 
  for $\eta\in\Lambda^3V^*$ 
  to obtain
  \begin{equation*}
    3\iota(u)\phi\wedge\iota(u)\phi\wedge\cL_A\phi 
    = 12\inner{u}{Au}\dvol + 6\abs{u}^2\tr(A)\dvol.
  \end{equation*}
  Divide this equation by $12\dvol$ and use the definition of $S_{\cL_A\phi}$
  in equation~\eqref{eq:Seta} to obtain 
  \begin{equation*}
    \inner{u}{S_{\cL_A\phi}u} = \inner{u}{Au} + \tfrac{1}{2}\tr(A)\abs{u}^2.
  \end{equation*}
  Since $S_{\cL_A\phi}$ is a symmetric endomorphism,
  this proves equation~\eqref{eq:SLA}. %
  Now take $A=\one$ and use the identities
  $\cL_\one\phi=3\phi$ and $\tr(\one)=7$ to obtain
  $S_{3\phi}=S_{\cL_\one\phi}=\tfrac{9}{2}\one$. 
  This proves \autoref{Pf_271}.

  \begin{step}
    \label{Pf_272}
    Let $v\in V$. %
    Then $S_{\iota(v)\psi}=0$.
  \end{step}

  It follows from equation~\eqref{eq:phipsi10} in \autoref{le:phipsi}
  that
  \begin{equation}
    \label{eq:phipsialu}
    \frac{\iota(u)\phi\wedge\alpha\wedge\psi}{\dvol} 
    = \frac{\alpha\wedge *u^*}{\dvol}
    = 3\alpha(u)
  \end{equation}
  for all $u\in V$ and all $\alpha\in V^*$. %
  Take $\alpha:=\iota(w)\iota(v)\phi=\phi(v,w,\cdot)$ to obtain
  \begin{equation}
    \label{eq:uvw1}
    3\phi(u,v,w)
    = \frac{\iota(u)\phi\wedge\iota(w)\iota(v)\phi\wedge\psi}{\dvol}.
  \end{equation}
  Interchange $u$ and $v$ to obtain 
  \begin{equation}
    \label{eq:uvw2}
    -3\phi(u,v,w)
    = \frac{\iota(w)\iota(u)\phi\wedge\iota(v)\phi\wedge\psi}{\dvol}.
  \end{equation}
  Now contract the vector $w$ with the $8$--form
  $\iota(u)\phi\wedge\iota(v)\phi\wedge\psi=0$ to obtain
  \begin{align*}
    0
    &=
      \iota(w)\bigl(\iota(u)\phi\wedge\iota(v)\phi\wedge\psi\bigr) \\
    &=
      \iota(w)\iota(u)\phi\wedge\iota(v)\phi\wedge\psi \\ &\quad
       + \iota(u)\phi\wedge\iota(w)\iota(v)\phi\wedge\psi \\ &\quad
       + \iota(u)\phi\wedge\iota(v)\phi\wedge\iota(w)\psi \\
    &=
      \iota(u)\phi\wedge\iota(v)\phi\wedge\iota(w)\psi.
  \end{align*}
  Here the last step follows from~\eqref{eq:uvw1}
  and~\eqref{eq:uvw2}. %
  Thus we have proved that
  \begin{equation}
    \label{eq:phipsiuvw}
    \iota(u)\phi\wedge\iota(v)\phi\wedge\iota(w)\psi
    = 0\qquad\mbox{for all }u,v,w\in V.
  \end{equation}
  Hence, $S_{\iota(w)\psi}=0$ for all $w\in V$ by definition of
  $S_\eta$. %
  This proves \autoref{Pf_272}.

  \begin{step}
    \label{Pf_273}
    Let $S=S^*\in\End(V)$ be a self-adjoint endomorphism. %
    Then
    \begin{equation}
      \label{eq:*LS}
      *\cL_S\phi = \tr(S)\psi-\cL_S\psi.
    \end{equation}
  \end{step}

  It suffices to prove this for self-adjoint rank $1$ endomorphisms. %
  Let $u\in V$ and define $S:=uu^*$. %
  Then $\tr(S)=\abs{u}^2$ and $\cL_S\phi=u^*\wedge\iota(u)\phi$. %
  Hence,
  \begin{align*}
    *\cL_S\phi 
    &=
      *\bigl(u^*\wedge\iota(u)\phi) \\
    &=
      *\big(u^*\wedge *(u^*\wedge\psi)\bigr) \\
    &=
      \iota(u)(u^*\wedge\psi) \\
    &=
      \abs{u}^2\psi - u^*\wedge\iota(u)\psi \\
    &=
      \tr(S)\psi - \cL_S\psi.
  \end{align*}
  Here the third step uses the identity
  $u^*\wedge*\alpha=(-1)^{k-1}*\iota(u)\alpha$ in \autoref{rmk:star}
  with $k=5$ and $\alpha=u^*\wedge\psi$. %
  This proves \autoref{Pf_273}.

  \begin{step}
    \label{Pf_274}
    Let $S=S^*\in\End(V)$ and $T=T^*\in\End(V)$ be self-adjoint
    endomorphisms. %
    Then
    \begin{equation}
      \label{eq:LSLT}
      \inner{\cL_S\phi}{\cL_T\phi} 
      = 2\tr(ST) + \tr(S)\tr(T).
    \end{equation}
  \end{step}

  It suffices to prove this for self-adjoint rank $1$ endomorphisms. %
  Let $u,v\in V$ and define $S := uu^*$ and $T:=vv^*$. %
  Then $\tr(S) = \abs{u}^2$, $\tr(T) = \abs{v}^2$,
  $\tr(ST) = \inner{u}{v}^2$, $\cL_S\phi=u^*\wedge\iota(u)\phi$,
  $\cL_T\phi=v^*\wedge\iota(v)\phi$. %
  Hence, by \autoref{Pf_273},
  \begin{align*}
    \inner{\cL_S\phi}{\cL_T\phi}\dvol
    &=
      \cL_S\phi\wedge *\cL_T\phi \\
    &=
      \cL_S\phi\wedge \bigl(\tr(T)\psi-\cL_T\psi\bigr) \\
    &=
      \abs{v}^2u^*\wedge\iota(u)\phi\wedge\psi
        - u^*\wedge\iota(u)\phi\wedge v^*\wedge\iota(v)\psi \\
    &=
      \abs{v}^2\iota(u)\phi\wedge*\iota(u)\phi
        - u^*\wedge v^*\wedge\iota(u)\phi\wedge\iota(v)\psi \\ 
    &=
      \bigl(3\abs{u}^2\abs{v}^2 - 2\abs{u\times v}^2\bigr)\dvol \\
    &=
      \bigl(\abs{u}^2\abs{v}^2 + 2\inner{u}{v}^2\bigr)\dvol. 
  \end{align*}
  Here the fourth step follows from~\eqref{eq:phipsi4} and the
  fifth step follows from~\eqref{eq:phipsi6}
  and~\eqref{eq:phipsi19}. %
  This proves \autoref{Pf_274}.

  \begin{step}
    \label{Pf_275}
    Let $S=S^*\in\End(V)$ be a self-adjoint endomorphism and let
    $u\in V$. %
    Then $\inner{\iota(u)\psi}{\cL_S\phi} = 0$.
  \end{step}

  It suffices to prove this for rank $1$ endomorphisms. %
  Let $v\in V$ and define $S:=vv^*$. %
  Then $*\cL_S\phi = \tr(S)\psi-\cL_S\psi 
  = \abs{v}^2\psi-v^*\wedge\iota(v)\psi$ by~\autoref{Pf_273},
  so
  \begin{equation*}
    \iota(u)\psi\wedge*\cL_S\phi 
    = \abs{v}^2\iota(u)\psi\wedge\psi 
    - \iota(u)\psi\wedge v^*\wedge\iota(v)\psi
    = 0.
  \end{equation*}
  Here the last equation follows from~\eqref{eq:phipsi3}
  and~\eqref{eq:phipsi11}.

  \begin{step}
    \label{Pf_276}
    Define 
    \begin{equation*}
      \End^\sym_0(V) := \left\{S\in\End(V) : S=S^*,\,\tr(S)=0\right\}.
    \end{equation*}
    Then the map $A\mapsto\cL_A\phi$ restricts to 
    $\rG(V,\phi)$--equivariant isomorphism
    \begin{equation*}
      \End^\sym_0(V)\to\Lambda^3_{27}\co S\mapsto\cL_S\phi.
    \end{equation*}
  \end{step}

  That the map $A\mapsto\cL_A\phi$ is $\rG(V,\phi)$--equivariant
  follows directly from the definitions. %
  Now let $S\in\End^\sym_0(V)$. %
  Then by \autoref{Pf_274}
  \begin{equation*}
    \frac{\cL_S\phi\wedge\psi}{\dvol}
    = \inner{\cL_S\phi}{\phi} 
    = \tfrac13\inner{\cL_S\phi}{\cL_\one\phi}
    = \tr(S)
    = 0.
  \end{equation*}
  Moreover, $*\iota(u)\psi=-u^*\wedge\phi$ by~\eqref{eq:phipsi5} and
  so $u^*\wedge\cL_S\phi\wedge\phi=-\inner{\cL_S\phi}{\iota(u)\psi}=0$
  for all $u\in V$ by \autoref{Pf_275}. %
  This shows that $\cL_S\phi\wedge\phi=0$ and $\cL_S\phi\wedge\psi=0$,
  and so $\cL_S\phi\in\Lambda^3_{27}$. %
  Moreover, $S_{\cL_S\phi}=S$ for all $S\in\End^\sym_0(V)$ by
  \autoref{Pf_271}. %
  Thus the map $\End^\sym_0(V)\to\Lambda^3_{27}\co S\mapsto\cL_S\phi$
  is injective. %
  Since $\End^\sym_0(V)$ and $\Lambda^3_{27}$ both have dimension
  $27$, this proves \autoref{Pf_276} and \autoref{thm:27}.
\end{proof}

The above proof of~\autoref{thm:27} does not use the fact 
that the $\rG(V,\phi)$--repre\-sen\-ta\-tion $\End^\sym_0(V)$,
and hence also $\Lambda^3_{27}$, is irreducible. %
Moreover, we have not included a proof of this fact in these 
notes (although it is stated in~\autoref{thm:form7}).  %
Assuming irreducibility, the proof of~\autoref{thm:27} 
can be simplified as follows. %

\begin{proof}[Proof of~\autoref{thm:27} assuming $\End^\sym_0(V)$ is irreducible]
  Since
  \begin{equation*}
    \cL_A \phi = \left.\frac{\rd}{\rd t}\right|_{t=0} \exp(t A)^*\phi,
  \end{equation*}
  it is clear that the map $\End(V)\to\Lambda^3V^*:A\mapsto\cL_A\phi$ 
  is $\rG(V,\phi)$--equivariant. %
  Its kernel is $\Lie(\rG(V,\phi))$ and hence its restriction to 
  $\End^\sym_0(V)$ is injective. %
  Now the composition of the map $\End^\sym_0(V)\to\Lambda^3V^*\co A\to\cL_A\phi$ 
  with the orthogonal projection onto $\Lambda^3_1$, respectively $\Lambda^3_7$, 
  is $\rG(V,\phi)$--equivariant by~\autoref{Pf_form75} in the proof of~\autoref{thm:form7}.
  This composition cannot be an isomorphism for dimensional 
  reasons, and hence must vanish by Schur's Lemma, 
  because the $\rG(V,\phi)$--representations $\End^\sym_0(V)$, 
  $\Lambda^3_1$, and $\Lambda^3_7$ are all irreducible.
  Thus the image of $\End^\sym_0(V)$ under the map $A\mapsto\cL_A\phi$
  is perpendicular to $\Lambda^3_1$ and $\Lambda^3_7$,
  and hence is equal to $\Lambda^3_{27}$.
\end{proof}

We close this section with the proof of a well-known formula for the
differential of the map that assigns to a nondegenerate $3$--form its
coassociative calibration. %
Let $V$ be a seven-dimensional real vector space, abbreviate
$\Lambda^k:=\Lambda^kV^*$ for $k=0,1,\dots,7$, and define
\begin{equation*}
  \cP
  = \cP(V)
  := \left\{\phi\in\Lambda^3\,\big|\,
    \phi\mbox{ is nondegenerate}\right\}.
\end{equation*}
This is an open subset of $\Lambda^3$ and it is diffeomorphic to the
homogeneous space $\GL(7,\R)/\Gtwo$. %
Namely, if $\phi_0\in\cP$ is any nondegenerate $3$--form then the map
$\GL(V)\to\cP\co g\mapsto(g^{-1})^*\phi_0$ descends to a
diffeomorphism from the quotient space $\GL(V)/\rG(V,\phi_0)$ to
$\cP$. %
Define the map $\Theta\co \cP\to\Lambda^4$ by
\begin{equation}
  \label{eq:Theta}
  \Theta(\phi) := *_\phi\phi.
\end{equation}
Here $*_\phi\co \Lambda^3\to\Lambda^4$ denotes the Hodge $*$--operator
associated to the inner product and orientation determined by $\phi$.

\begin{theorem}
  \label{thm:Theta}
  The map $\Theta\co \cP\to\Lambda^4$ in~\eqref{eq:Theta} is a
  $\GL(V)$--equivariant local diffeomorphism, it restricts to a
  diffeomorphism onto its image on each connected component of $\cP$,
  and its derivative at $\phi\in\cP$ is given by
  \begin{equation}
    \label{eq:dTheta}
    d\Theta(\phi)\eta
    = *_\phi\left(
      \tfrac{4}{3}\pi_1(\eta)
      +\pi_7(\eta)
      -\pi_{27}(\eta)\right)
  \end{equation}
  for $\eta\in\Lambda^3$. %
  Here $\pi_d\co \Lambda^3\to\Lambda^3_d$ denotes the projection
  associated to the orthogonal splitting
  $\Lambda^3=\Lambda^3_1\oplus\Lambda^3_7\oplus\Lambda^3_{27}$ in
  \autoref{thm:form7} determined by $\phi$.
\end{theorem}

\begin{proof}
  That $\cP$ has two connected components distinguished by the 
  orientation of $V$ follows from the fact that $\GL(V)$ 
  has two connected components.  That the restriction of 
  $\Theta$ to each connected component of $\cP$ 
  is bijective follows from \autoref{thm:phipsi}
  and that it is a diffeomorphism then follows from 
  equation~\eqref{eq:dTheta} and the inverse function theorem.

  Thus it remains to prove~\eqref{eq:dTheta}.
  Since $\Theta$ is $\GL(V)$--equivariant, it satisfies 
  \begin{equation}
    \label{eq:gTheta}
    \Theta(g^*\phi)=g^*\Theta(\phi)
  \end{equation}
  for $\phi\in\cP$ and $g\in\GL(V)$. %
  Fix a nondegenerate $3$--form $\phi\in\cP$, denote by
  $\psi:=\Theta(\phi)=*_\phi\phi$ its coassociative calibration, and
  differentiate equation~\eqref{eq:gTheta} at $g=\one$ in the
  direction $A\in\End(V)$ to obtain
  \begin{equation}
    \label{eq:ATheta}
    d\Theta(\phi)\cL_A\phi=\cL_A\psi.
  \end{equation}
  Now let $\eta\in\Lambda^3$ and denote $\eta_d:=\pi_d(\eta)$ for
  $d=1,7,27$. %
  By \autoref{thm:form7} and \autoref{thm:27} there exists a real
  number $\lambda$, a vector $u\in V$, and a traceless symmetric
  endomorphism $S\co V\to V$ such that
  \begin{equation*}
    \eta_1 = 3\lambda\phi,\qquad
    \eta_7 = 3\iota(u)\psi,\qquad
    \eta_{27} = \cL_S\phi.
  \end{equation*}
  Since $\cL_\one\phi=3\phi$ and $\cL_\one\psi=4\psi$,
  it follows from equation~\eqref{eq:ATheta} that
  \begin{equation}
    \label{eq:dTheta1}
    d\Theta(\phi)\eta_1
    = \lambda d\Theta(\phi)\cL_\one\phi
    = \lambda \cL_\one\psi
    = 4\lambda\psi
    = \tfrac{4}{3}*_\phi(3\lambda\phi)
    = \tfrac{4}{3}*_\phi\eta_1.
  \end{equation}
  Now define $A_u\in\End(V)$ by $A_uv:=u\times v$ for $v\in V$. %
  Then
  \begin{equation*}
    \cL_{A_u}\phi=3\iota(u)\psi = \eta_7,\qquad
    \cL_{A_u}\psi=*_\phi(3\iota(u)\psi) = *_\phi\eta_7
  \end{equation*}
  by~\eqref{eq:uphi} and~\eqref{eq:upsi}. %
  Hence, it follows from equation~\eqref{eq:ATheta} that
  \begin{equation}
    \label{eq:dTheta7}
    d\Theta(\phi)\eta_7
    = d\Theta(\phi)\cL_{A_u}\phi
    = \cL_{A_u}\psi
    = *_\phi\eta_7.
  \end{equation}
  Moreover it follows from equations~\eqref{eq:*LS} 
  and~\eqref{eq:ATheta}
  \begin{equation}
    \label{eq:dTheta27}
    d\Theta(\phi)\eta_{27}
    = d\Theta(\phi)\cL_S\phi
    = \cL_S\psi
    = -*_\phi\cL_S\phi
    = -*_\phi\eta_{27}.
  \end{equation}
  With this understood, equation~\eqref{eq:dTheta} follows
  from~\eqref{eq:dTheta1}, \eqref{eq:dTheta7},
  and~\eqref{eq:dTheta27}. %
  This proves~\autoref{thm:Theta}.
\end{proof}


    
\section{The group \texorpdfstring{$\Spin(7)$}{Spin(7)}}
\label{sec:Spin7}  

Let $W$ be an $8$--dimensional real Hilbert space equipped with a
positive triple cross product and let $\Phi\in\Lambda^4W^*$ be the
Cayley calibration defined by~\eqref{eq:Phi}. %
We orient $W$ so that 
$$
\Phi\wedge\Phi>0
$$
and denote by $*\co\Lambda^kW^*\to\Lambda^{8-k}W^*$ 
the associated Hodge $*$--operator. %
Then $\Phi$ is self-dual, by \autoref{rmk:Worient}. %
Recall that, for every unit vector $e\in W$, the subspace
$$
V_e:=e^\perp
$$ 
is equipped with a cross product
$$ 
u\times_ev:=u\times e\times v 
$$ 
and that
\begin{equation*}
  \Phi=e^*\times\phi_e+\psi_e,\quad
  \phi_e:=\iota(e)\Phi\in\Lambda^3W^*,\quad
  \psi_e:=*(e^*\wedge\phi_e)\in\Lambda^4W^*,
\end{equation*} 
(see \autoref{thm:cayley}). %
The orientation of $W$ is compatible with the decomposition
$W=\langle e\rangle\oplus V_e$ (see \autoref{rmk:Worient}).

The group of automorphisms of $\Phi$ will be denoted by
\begin{equation*}
  \rG(W,\Phi) 
  := \left\{g\in\GL(W) : g^*\Phi=\Phi\right\}.
\end{equation*}
By \autoref{thm:CAYLEY}, we have $\rG(W,\Phi)\subset\SO(W)$ and hence
\begin{equation*}
  \rG(W,\Phi) 
  = \left\{g\in\SO(W) : 
    gu\times gv\times gw=g(u\times v\times w)\;
    \forall\,u,v,w\in W\right\}.
\end{equation*}
For the standard structure $\Phi_0$ on $\R^8$ in \autoref{ex:O} we
denote the structure group by $\Spin(7):=\rG(\R^8,\Phi_0)$. %
By \autoref{thm:CAYLEY1}, the group $\rG(W,\Phi)$ is isomorphic to
$\Spin(7)$ for every positive Cayley-form on an $8$--dimensional
vector space.

\begin{theorem}
  \label{thm:Spin7}
  The group $\rG(W,\Phi)$ is a $21$--dimensional simple, connected,
  simply connected Lie group. %
  It acts transitively on the unit tangent bundle of the unit sphere
  and, for every unit vector $e\in W$, the isotropy subgroup
  $\rG_e:=\left\{g\in\rG(W,\Phi) : ge=e\right\}$ is isomorphic to
  $\rG_2$. %
  Thus there is a fibration
  \begin{equation*}
    \rG_2\hookrightarrow \Spin(7) \longrightarrow S^7. 
  \end{equation*}
\end{theorem}

\begin{proof}
  The isotropy subgroup $\rG_e$ is obviously isomorphic to
  $\rG(V_e,\phi_e)$ and hence to $\rG_2$. %
  We prove that $\rG(W,\Phi)$ acts transitively on the unit sphere. %
  Let $u,v\in W$ be two unit vectors and choose a unit vector $e\in W$
  which is orthogonal to $u$ and $v$. %
  By \autoref{thm:G2}, the isotropy subgroup $\rG_e$ acts transitively
  on the unit sphere in $V_e$. %
  Hence, there is an element $g\in\rG_e$ such that $gu=v$. That
  $\rG(W,\Phi)$ acts transitively on the set of pairs of orthonormal
  vectors now follows immediately from \autoref{thm:G2}. %
  In particular, there is a fibration
  $\rG_2\hookrightarrow \Spin(7) \longrightarrow S^7$. %
  It follows from the homotopy exact sequence of this fibration and
  \autoref{thm:G2} that $\Spin(7)$ is connected and simply connected,
  and that $\pi_3(\Spin(7))\cong\Z$. %
  Hence, $\Spin(7)$ is simple. %
  This proves \autoref{thm:Spin7}.
\end{proof}

\begin{lemma}
  \label{le:spin7}
  Abbreviate 
  \begin{equation*}
    \rG := \rG(W,\Phi),\qquad
    \g := \Lie(\rG)\subset\so(W).
  \end{equation*}
  The homomorphism $\rho\co \rG(W,\Phi)\to\SO(\g^\perp)$ is a
  nontrivial double cover. %
  Hence, $\Spin(7)$ is isomorphic to the universal cover of $\SO(7)$.
\end{lemma}

\begin{proof}
  Define 
  $$
  I := \left\{\xi\in\g : [\xi,\so(W)]\subset\g\right\}.
  $$
  If $\xi\in I$ and $\eta\in\g$, then
  $[[\xi,\eta],\zeta]=-[[\eta,\zeta],\xi]-[[\zeta,\xi],\eta]\in\g$
  for all $\zeta\in\so(W)$, and so $[\xi,\eta]\in I$. %
  Thus $I$ is an ideal in $\g$. %
  Since $\so(W)$ is simple, we have $I\subsetneq\g$. %
  Since $\g$ is simple, we have $I=\{0\}$. %
  This implies $\im\,\ad(\xi)\not\subset\g$ for $0\ne\xi\in\g$. %
  Since $\ad(\xi)\co \so(W)\to\so(W)$ is skew-adjoint, this implies
  $\g^\perp\not\subset\ker\,\ad(\xi)$ for $0\ne\xi\in\g$. %
  This means that the infinitesimal adjoint action defines an
  isomorphism $\g\to\so(\g^\perp)$. %
  Hence, the adjoint action gives rise to a covering map
  $\rG\to\SO(\g^\perp)$. %
  Since $\rG$ is connected and simply connected, this implies that
  $\rG$ is the universal cover of $\SO(\g^\perp)\cong\SO(7)$ and this
  proves \autoref{le:spin7}.
\end{proof}

We examine the action of the group $\rG(W,\Phi)$ on the space
\begin{equation*}
  \sS:=\left\{(u,v,w,x)\in W\,\big|\,
  u,v,w,u\times v\times w,x\mbox{ are orthonormal}\right\}.
\end{equation*}
The space $\sS$ is a bundle of $3$--spheres over a bundle of
$5$--spheres over a bundle of $6$--spheres over a $7$--sphere. %
Hence, it is a compact connected simply connected $21$--dimensional
manifold.

\begin{theorem}
  \label{thm:Spin7S}
  The group $\rG(W,\Phi)$ acts freely and transitively on $\sS$.
\end{theorem}

\begin{proof}
  Since $\Spin(7)$ acts transitively on $S^7$ with isotropy subgroup
  $\rG_2$, the result follows immediately from \autoref{thm:G2S}.
\end{proof}

\begin{cor}
\label{cor:cayley}
  The group $\rG(W,\Phi)$ acts transitively on the space of Cayley
  subspaces of $W$.  
\end{cor}

\begin{proof}
  This follows directly from \autoref{le:cayley-sub} and \autoref{thm:Spin7S}.
\end{proof}

\begin{remark}
  \label{rmk:cayley}
  For each Cayley subspace $H\subset W$ choose the orientation such that 
  $$
  \dvol_H:=\Phi|_H
  $$ 
  is a positive volume form and denote by $\Lambda^+H^*$ 
  the space of self-dual $2$-forms (as in  \autoref{rmk:LaH}),
  by $\pi_H:W\to H$ the orthogonal projection, 
  and by 
  \begin{equation*}
  \rG_H:=\left\{g\in\rG(W,\Phi):gH=H\right\}
  \end{equation*} 
  the isotropy subgroup. %
  Fix a Cayley subspace ${H\subset W}$.  %
  Then there is a unique orientation preserving $\rG_H$-equivariant 
  isometric isomorphism 
$$
  T_H:\Lambda^+H^*\to\Lambda^+(H^\perp)^*.
$$ 
  It is given by 
  \begin{equation}\label{eq:TH}
  T_H\om := -\tfrac12\bigl(*(\Phi\wedge\pi_H^*\om)\bigr)|_{H^\perp}
  \qquad\mbox{for }\om\in \Lambda^+H^*
  \end{equation}
  and its inverse is $(T_H)^{-1}=T_{H^\perp}$.  %
  If $\om_1,\om_2,\om_3$ is a standard basis of $\Lambda^+H^*$
  and $\tau_i\in\Lambda^+(H^\perp)^*$
  is defined by $\tau_i:=T_H\om_i$ for $i=1,2,3$, 
  then the Cayley calibration $\Phi$ can be expressed in the form 
  \begin{equation}\label{eq:PhiTH}
  \Phi
  = \pi_H^*\dvol_H+\pi_{H^\perp}^*\dvol_{H^\perp}
  - \sum_{i=1}^3\pi_H^*\om_i\wedge\pi_{H^\perp}^*\tau_i.
  \end{equation}
  To see this, choose a standard basis of $W$ as in \autoref{ex:cayley}
  such that the vectors $e_0,e_1,e_2,e_3$ form a basis of $H$,
  the vectors $e_4,e_5,e_6,e_7$ form a basis of $H^\perp$, and 
  \begin{gather*}
    \omega_1 = e^{01} + e^{23}, \quad
    \omega_2 = e^{02} - e^{13}, \quad
    \omega_3 = e^{03} + e^{12}, \\
    \tau_1 = e^{45} + e^{67}, \quad
    \tau_2 = e^{46} - e^{57}, \quad
    \tau_3 = e^{47} + e^{56}.
  \end{gather*}
That such a basis exists follows 
from \autoref{thm:CAYLEY1} and \autoref{thm:Spin7S}.   
It follows also from  \autoref{thm:Spin7S} that
a pair $(h,h')\in\SO(H)\times\SO(H^\perp)$ belongs 
to the image of the homomorphism $\rG_H\to\SO(H)\times\SO(H^\perp)$
if and only if the induced automorphisms of $\Lambda^+H^*$ 
and $\Lambda^+(H^\perp)^*$ are conjugate under $T_H$. Hence
the map 
\begin{equation*}
  \rG_H\to \SO(H)\times_{\SO(\Lambda^+H^*)}\SO(H^\perp):
  g\mapsto[g|_H,g|_{H^\perp}]
\end{equation*}
is a Lie group isomorphism. %
Hence, $\dim\rG_H=9$ and so the \defined{Cayley Grassmannian}
$$
\sH:=\left\{H\subset W:H\mbox{ is a Cayley subspace}\right\},
$$ 
which is diffeomorphic to the homogeneous 
space $\rG(W,\Phi)/\rG_H$, has dimension~$12$.
\end{remark}

\bigbreak

\begin{theorem}
  \label{thm:form8}
  There are orthogonal splittings
  \begin{align*}
    \Lambda^2W^* 
    &=
      \Lambda^2_7\oplus\Lambda^2_{21},\\
    \Lambda^3W^* 
    &= 
      \Lambda^3_8\oplus\Lambda^3_{48},\\
    \Lambda^4W^* 
    &= 
      \Lambda^4_1\oplus\Lambda^4_7
      \oplus\Lambda^4_{27}\oplus\Lambda^4_{35},
  \end{align*}
  where $\dim\Lambda^k_d=d$ and
  \begin{align*}
    \Lambda^2_7 
    &:= 
      \left\{\om\in\Lambda^2W^* : 
      *(\Phi\wedge\om)=3\om\right\} \\
    &\,= 
      \left\{u^*\wedge v^*-\iota(u)\iota(v)\Phi : 
      u,v\in W\right\},\\
    \Lambda^2_{21} 
    &:=
      \left\{\om_\xi : \xi\in\g\right\} \\
     &\,= 
      \left\{\om\in\Lambda^2W^* : 
      *(\Phi\wedge\om)=-\om\right\}\\
    &\;= 
      \left\{\om\in\Lambda^2W^* : 
      \inner{\om}{\iota(u)\iota(v)\Phi}=\om(u,v)\;
      \forall\,u,v\in W\right\},\\
    \Lambda^3_8
    &:= 
      \left\{\iota(u)\Phi : u\in W\right\},\\
    \Lambda^3_{48} 
    &:= 
      \left\{\omega\in\Lambda^3W^* : 
      \Phi\wedge\omega=0\right\},\\
    \Lambda^4_1 
    &:= 
      \langle\Phi\rangle,\\
    \Lambda^4_7
    &:= 
      \left\{\cL_\xi\Phi : \xi\in\so(W)\right\},\\
    \Lambda^4_{27} 
    &:= 
      \left\{\om\in\Lambda^4W^* : 
      *\om=\om,\,\om\wedge\Phi=0,\,
      \om\wedge\cL_\xi\Phi=0\,\forall\xi\in\so(W)\right\},\\
    \Lambda^4_{35} 
    &:=
      \left\{\om\in\Lambda^4W^* : *\om=-\om\right\}.
  \end{align*}
  Here $\g:=\Lie(\rG(W,\Phi))$ and, for $\xi\in\so(W)$, 
  the $4$--form $\cL_\xi\Phi\in\Lambda^4W^*$ and 
  the $2$--form $\om_\xi\in\Lambda^2W^*$  are defined by
  $\cL_\xi\Phi:=\left.\frac{d}{dt}\right|_{t=0}\exp(t\xi)^*\Phi$
  and $\om_\xi := \inner{\cdot}{\xi\cdot}$.
  Each of the spaces $\Lambda^k_d$ is an irreducible 
  representation of $\rG(W,\Phi)$.
\end{theorem}

\begin{proof}
  By \autoref{thm:Spin7}, $\rG:=\rG(W,\Phi)$ is simple and so the
  action of $\rG$ on $\g$ by conjugation is irreducible. %
  Hence, the $21$--dimensional subspace $\Lambda^2_{21}$ must be
  contained in an eigenspace of the operator
  ${\om\mapsto*(\Phi\wedge\om)}$ on $\Lambda^2W^*$. %
  We prove that the eigenvalue is $-1$. %
  To see this, we choose a unit vector $e\in W$ and an element
  $\xi\in\g$ with $\xi e=0$. %
  Let
  \begin{equation*}
    V_e:=e^\perp
  \end{equation*} 
  and denote by ${\iota_e\co V_e\to W}$ and ${\pi_e\co W\to V_e}$ the
  inclusion and orthogonal projection and by
  $*_e\co \Lambda^kV_e^*\to\Lambda^{7-k}V_e^*$ the Hodge $*$--operator
  on the subspace. %
  Then
  \begin{equation*}
    *(e^*\wedge\pi_e^*\alpha_e) = \pi_e^**_e\alpha_e\qquad
    \forall\;\alpha_e\in\Lambda^kV_e^*.
  \end{equation*}  
  Moreover, the alternating forms
  \begin{equation*}
    \phi_e:=\iota_e^*(\iota(e)\Phi),\qquad
    \psi_e:=\iota_e^*\Phi
  \end{equation*}
  are the associative and coassociative calibrations of $V_e$. %
  Since $\xi e=0$, we have $\om_\xi=\pi_e^*\iota_e^*\om_\xi$ and, by
  \autoref{thm:form7},
  \begin{equation*}
    \psi_e\wedge\iota_e^*\om_\xi=0,\qquad
    *_e(\phi_e\wedge\iota_e^*\om_\xi)=-\iota_e^*\om_\xi.
  \end{equation*}
  Since $\Phi=e^*\wedge\pi_e^*\phi_e+\pi_e^*\psi_e$, this gives
  \begin{align*}
    *\bigl(\Phi\wedge\om_\xi\bigr) 
    &=
      *\bigl(\left(e^*\wedge\pi_e^*\phi_e+\pi_e^*\psi_e\right)
      \wedge\pi_e^*\iota_e^*\om_\xi\bigr) \\
    &= 
      *\bigl(e^*\wedge\pi_e^*\left(\phi_e\wedge\iota_e^*\om_\xi\right)\bigr)
      + *\pi_e^*\bigl(\psi_e\wedge\iota_e^*\om_\xi\bigr) \\
    &= 
      \pi_e^**_e\left(\phi_e\wedge\iota_e^*\om_\xi\right) 
      =
      -\pi_e^*\iota_e^*\om_\xi 
      =
      -\om_\xi.
  \end{align*}
  By \autoref{le:spin7} the adjoint action of $\rG$ on
  $\g^\perp\subset\so(W)$ is irreducible, and $\g^\perp$ is mapped
  under $\xi\mapsto\om_\xi$ onto the orthogonal complement of
  $\Lambda^2_{21}$. %
  Hence, the $7$--dimensional orthogonal complement of
  $\Lambda^2_{21}$ is also contained in an eigenspace of the operator
  $\om\mapsto*(\Phi\wedge\om)$. 
  Since this operator is self-adjoint and has trace zero, its
  eigenvalue on the orthogonal complement of $\Lambda^2_{21}$ must be
  $3$ and therefore this orthogonal complement is equal to
  $\Lambda^2_7$. %
  It follows that the orthogonal projection of $\om\in\Lambda^2W^*$
  onto $\Lambda^2_7$ is given by
  $
    \pi_7(\om)=\tfrac{1}{4}\left(\om+*(\Phi\wedge\om)\right).
  $
  Hence, for every nonzero vector $e\in W$, we have 
  \begin{align*}
    \Lambda^2_7
    &=
    \left\{e^*\wedge u^*-\iota(e)\iota(u)\Phi : 
      u\in e^\perp\right\}, \\
    \Lambda^2_{21}
    &=
    \left\{\om\in\Lambda^2W^* : 
      \inner{\om}{\iota(e)\iota(u)\Phi}=\om(e,u)\,\forall\,u\in e^\perp
    \right\}.
  \end{align*}
  This proves the decomposition result for $\Lambda^2W^*$. 

\bigbreak

  We verify the decomposition of $\Lambda^3W^*$. %
  For $u\in W$ and $\omega\in\Lambda^3W^*$ we have
  the equation
  \begin{equation*}
    u^*\wedge\Phi\wedge\omega = -\omega\wedge*\iota(u)\Phi.
  \end{equation*}
  Hence, $\Phi\wedge\omega=0$ if and only if $\omega$ is orthogonal to
  $\iota(u)\Phi$ for all ${u\in W}$. %
  This shows that $\Lambda^3_{48}$ is the orthogonal complement of
  $\Lambda^3_8$. %
  Since $\Phi$ is nondegenerate, we have $\dim\,\Lambda^3_8=8$ and,
  since $\dim\,\Lambda^3W^*=56$, it follows that
  $\dim\,\Lambda^3_{48}=48$.

  We verify the decomposition of $\Lambda^4W^*$. %
  The $4$--form $g^*\Phi$ is self-dual for every
  $g\in\rG=\rG(W,\Phi)$, because $\Phi$ is self-dual and
  $\rG\subset\SO(W)$. %
  This implies that $\cL_\xi\Phi$ is self-dual for every
  $\xi\in\g=\Lie(\rG)$. %
  Since $\SO(W)$ has dimension $28$ and the isotropy subgroup $\rG$ of
  $\Phi$ has dimension $21$, it follows that the tangent space
  $\Lambda^4_7$ to the orbit of $\Phi$ under the action of $\rG$ has
  dimension $7$. %
  As $\Lambda^4_1$ has dimension $1$ and the space of self-dual
  $4$--forms has dimension $35$, the orthogonal complement of
  $\Lambda^4_1\oplus\Lambda^4_7$ in the space of self-dual $4$--forms
  has dimension $27$. %
  This proves the dimension and decomposition statements.

  That the action of $\rG$ on $\Lambda^2_{21}\cong\g$ is irreducible
  follows from the fact that $\rG$ is simple. %
  Irreducibility of the action on $\Lambda^4_1$ is obvious. %
  For $\Lambda^3_8\cong W$ it follows from the fact that $\rG$ acts
  transitively on the unit sphere in $W$, and for
  $\Lambda^2_7\cong\g^\perp\cong\Lambda^4_7$ it follows from the fact
  that the istropy subgroup $\rG_e$ of a unit vector $e\in W$ acts
  transitively on the unit sphere in $V_e=e^\perp$. %
  For $\Lambda^4_{27}$, $\Lambda^4_{35}$, and $\Lambda^3_{48}$ we
  refer to~\cite{Bryant1987}. %
  This proves \autoref{thm:form8}.
\end{proof}

\begin{cor}
  \label{cor:uvxy}
  For $u,v\in W$ denote
  $\om_{u,v}:=\iota(v)\iota(u)\Phi=\Phi(u,v,\cdot,\cdot)$. %
  Then, for all $u,v,x,y\in W$ we have
  \begin{gather}
    \label{eq:uvxy1}
    *\left(\Phi\wedge u^*\wedge v^*\right)
      = \om_{u,v},\qquad
    *\left(\Phi\wedge \om_{u,v}\right)
      = 3u^*\wedge v^*+2\om_{u,v},
    \\
    \label{eq:uvxy2}
    \inner{\om_{u,v}}{\om_{x,y}} 
      = 3\bigl(
      \inner{u}{x}\inner{v}{y}-\inner{u}{y}\inner{v}{x}
      \bigr)
    + 2\Phi(u,v,x,y),
    \\
    \label{eq:uvxy3}
    \frac{\om_{u,v}\wedge\om_{x,y}\wedge\Phi}{\dvol}
      = 6\bigl(\inner{u}{x}\inner{v}{y}-\inner{u}{y}\inner{v}{x}\bigr)
      + 7\Phi(u,v,x,y).
  \end{gather}
\end{cor}

\begin{proof}
  The first equation in~\eqref{eq:uvxy1} is a general statement about
  the Hodge $*$--opera\-tor in any dimension. %
  Moreover, by \autoref{thm:form8}, the $2$--form
  $u^*\wedge v^*+\om_{u,v}$ is an eigenvector of the operator
  $\om\mapsto*(\Phi\wedge\om)$ with eigenvalue $3$. %
  Hence, the second equation in~\eqref{eq:uvxy1} follows from the
  first. %
  To prove~\eqref{eq:uvxy2}, take the inner product of the second
  equation in~\eqref{eq:uvxy1} with $x^*\wedge y^*$ and use the
  identities
  \begin{gather}
    \label{eq:uvxy4}
    \inner{\om_{u,v}}{x^*\wedge y^*} 
    =\Phi(u,v,x,y),
    \\
    \label{eq:uvxy5}
    \inner{u^*\wedge v^*}{x^*\wedge y^*} 
    =\inner{u}{x}\inner{v}{y}-\inner{u}{y}\inner{v}{x},
  \end{gather}
  and the fact that the operator $\om\mapsto *(\Phi\wedge\om)$ is
  self-adjoint. %
  To prove~\eqref{eq:uvxy3}, we observe that
  \begin{align*}
    \frac{\om_{u,v}\wedge\om_{x,y}\wedge\Phi}{\dvol}
    &=
      \inner{\om_{u,v}}{*(\Phi\wedge\om_{x,y})} \\
    &=
      \inner{\om_{u,v}}{3x^*\wedge y^*+2\om_{x,y}} \\
    &= 
      6\bigl(\inner{u}{x}\inner{v}{y}-\inner{u}{y}\inner{v}{x}\bigr)
      + 7\Phi(u,v,x,y),
  \end{align*}
  where the second equation follows from~\eqref{eq:uvxy1}
  and the last follows from~\eqref{eq:uvxy2} and~\eqref{eq:uvxy5}. %
  This proves \autoref{cor:uvxy}.
\end{proof}



\section{Spin structures}
\label{sec:spin}  

This section explains how a cross products in dimension seven,
respectively a triple cross products in dimension eight,
gives rise to a spin structure and a unit spinor and how,
conversely, the cross product or triple cross product
can be recovered from these data. %
We begin the discussion with 
spin structures and triple cross products in~\autoref{sec:spintriplecross}
and then move on to cross products in~\autoref{sec:spincross}.


\subsection{Spin structures and triple cross products}
\label{sec:spintriplecross}

Let $W$ be an $8$--dimensional oriented real Hilbert space. %
A spin structure on~$W$ is a pair of $8$--dimensional real Hilbert
spaces $S^\pm$ equipped with a vector space homomorphism
$\gamma\co W\to\Hom(S^+,S^-)$ that satisfies the condition
\begin{equation}
  \label{eq:spindim8-1}
  \gamma(u)^*\gamma(u) = \abs{u}^2\one
\end{equation}
for all $u\in W$ (see~\cite{Salamon1999a}*{Proposition 4.13,
  Definition 4.32, Example 4.48}). %
The sign in $S^\pm$ is determined by the condition
\begin{equation}
  \label{eq:spindim8-2} 
  \gamma(e_7)^*\gamma(e_6)\cdots\gamma(e_1)^*\gamma(e_0)=\one_{S^+}
\end{equation}
for some, and hence every, positively oriented orthonormal basis
$e_0,\dots,e_7$ of $W$ (see~\cite{Salamon1999a}*{page~132}). %
More precisely, consider the $16$--dimensional real Hilbert space
$S := S^+\oplus S^-$ and define the homomorphism
$\Gamma\co W\to\End(S)$ by
\begin{equation*}
  \Gamma(u) := 
  \begin{pmatrix}
    0 & \gamma(u) \\
    -\gamma(u)^* & 0
  \end{pmatrix}
  \qquad\mbox{for }u\in W.
\end{equation*}
Then equation~\eqref{eq:spindim8-1} guarantees that $\Gamma$ extends
uniquely to an algebra isomorphism from the Clifford algebra $\Cl(W)$
to $\End(S)$, still denoted by $\Gamma$. %
The complexification of $S$ gives rise an algebra isomorphism
$\Gamma^c\co\Cl^c(W)\to\End(S^c)$ from the complexified Clifford
algebra $\Cl^c(W):=\Cl(W)\otimes_\R\C$ to the complex endomorphisms of
$S^c:=S\otimes_\R\C$ (see~\cite{Salamon1999a}*{Proposition~4.33}). %

\begin{theorem}
  \label{thm:spintriplecross}
  Let $W$ be an oriented $8$--dimensional real Hilbert space
  and abbreviate $\Lambda^k:=\Lambda^kW^*$ for $k=0,1,\dots,8$. %
  \begin{enumerate}[(i)]
  \item
    \label{spintriplecross1}
    Suppose $W$ is equipped with a positive triple cross
    product~\eqref{eq:tc}, let ${\Phi\in\Lambda^4}$ be the Cayley
    calibration defined by~\eqref{eq:Phi}, and assume that
    ${\Phi\wedge\Phi>0}$. %
    Define the homomorphism $\gamma\co W\to\Hom(S^+,S^-)$ by
    \begin{equation}
      \label{eq:Spm}
      S^+ := \Lambda^0\oplus\Lambda^2_7,\qquad S^- := \Lambda^1
    \end{equation}
    and
    \begin{equation}
      \label{eq:gamma2}
      \gamma(u)(\lambda,\om) := \lambda u^*+2\iota(u)\om
    \end{equation}
    for $u\in W$, $\lambda\in\R$, and $\om\in\Lambda^2_7$. %
    Then $\gamma$ is a spin structure on~$W$, i.e., it
    satisfies~\eqref{eq:spindim8-1} and~\eqref{eq:spindim8-2}. %
    Moreover, the space $S^+=\Lambda^0\oplus\Lambda^2_7$ of positive
    spinors contains a canonical unit vector $s=(1,0)$ and the triple
    cross product can be recovered from the spin structure and the
    unit spinor via the formula
    \begin{equation}
      \label{eq:triplecrossgamma}
      \begin{split}
        \gamma(u\times v\times w)s 
        &= \inner{v}{w}\gamma(u)s - \inner{w}{u}\gamma(v)s + \inner{u}{v}\gamma(w)s \\
        &\quad- \gamma(u)\gamma(v)^*\gamma(w)s
      \end{split}
    \end{equation}
    for $u,v,w\in W$.

  \item
    \label{spintriplecross2}
    Let $\gamma\co W\to\Hom(S^+,S^-)$ be a spin structure 
    and let $s\in S^+$ be a unit vector. %
    Then equation~\eqref{eq:triplecrossgamma}
    defines a positive triple cross product on~$W$ and the associated
    Cayley calibration $\Phi$ satisfies $\Phi\wedge\Phi>0$. %
    Since any two spin structures on $W$ are isomorphic, 
    this shows that there is a one-to-one correspondence between positive 
    unit spinors and positive triple cross products on $W$ 
    that are compatible with the inner product and orientation. %
  \end{enumerate}
\end{theorem}

\begin{proof}
  See page~\pageref{proof:spintriplecross}.
\end{proof}

Assume $W$ is equipped with a positive triple cross
product~\eqref{eq:tc} and that its Cayley calibration
$\Phi\in\Lambda^4W^*$ in~\eqref{eq:Phi} satisfies
$\Phi\wedge\Phi>0$. %
Recall that, for every unit vector $e\in W$, there is a normed algebra
structure on~$W$, defined by~\eqref{eq:NDA}. %
This normed algebra structure can be recovered from an intrinsic
product map
\begin{equation*}
  m\co W\times W\to \Lambda^0\oplus\Lambda^2_7
\end{equation*}
(which does not depend on $e$) and an isomorphism 
$\gamma(e)\co \Lambda^0\oplus\Lambda^2_7\to\Lambda^1$
(which does depend on $e$). %
The product map is given by 
\begin{equation}
  \label{eq:m}
  m(u,v) = \left(\inner{u}{v},
    \tfrac{1}{2}(u^*\wedge v^*+\om_{u,v})\right)
\end{equation}
for $u,v\in W$ and the isomorphism $\gamma(e)$ is given by~\eqref{eq:gamma2}
with $u$ replaced by $e$. %
Here $\om_{u,v}:=\iota(v)\iota(u)\Phi$ as in~\autoref{cor:uvxy}.

\begin{lemma}
  \label{le:m}
  Let $I_W\co W\to W^*$ be the isomorphism induced by the inner
  product, so that $I_W(u)=\inner{u}{\cdot}=u^*$ for $u\in W$. %
  Let $\gamma\co W\to\Hom(S^+,S^-)$ and $m\co W\times W\to S^+$ be
  defined by~\eqref{eq:gamma2} and~\eqref{eq:m}. %
  Then, for all $u,v,e\in W$, we have
  \begin{gather}
    \label{eq:gammam}
    I_W^{-1}(\gamma(e)m(u,v))
    = \inner{u}{v}e + \inner{u}{e}v-\inner{v}{e}u + u\times e\times v, \\
    \label{eq:mnorm}
    \abs{m(u,v)} = \abs{u}\abs{v},\qquad
    \abs{\gamma(e)(\lambda,\om)}^2 
    = \abs{e}^2\left(\abs{\lambda}^2+\abs{\om}^2\right).
  \end{gather}
\end{lemma}

\begin{proof}
  Equation~\eqref{eq:gammam} follows directly from the definitions. %
  Moreover, it follows from~\eqref{eq:uvxy2} that
  \begin{equation*}
    \abs{m(u,v)}^2
    = \inner{u}{v}^2 + \tfrac{1}{4}\abs{u^*\wedge v^*}^2
    + \tfrac{1}{4}\abs{\om_{u,v}}^2
    = \inner{u}{v}^2 + \abs{u\wedge v}^2
    = \abs{u}^2\abs{v}^2.
  \end{equation*}
  This proves the first equation in~\eqref{eq:mnorm}. %
  To prove the second equation in~\eqref{eq:mnorm} we observe that
  $\gamma(e)m(e,v)=v$ and, hence,
  $\abs{\gamma(e)m(e,v)} = \abs{v} = \abs{m(e,v)}$ whenever
  $\abs{e}=1$. %
  Since the map $W\to\Lambda^0\oplus\Lambda^2_7\co v\mapsto m(e,v)$ is
  bijective, this proves \autoref{le:m}.
\end{proof}

\begin{remark}
  \label{rmk:prodm}
  If we fix a unit vector $e\in W$ and denote $\vbar:=2\inner{e}{v}e-v$, 
  then the product in~\eqref{eq:NDA} is given by
  \begin{equation*}
    uv
    = -\inner{u}{v}e + \inner{u}{e}v+\inner{v}{e}u + u\times e\times v
    = \overline{I_W^{-1}(\gamma(e)m(u,\vbar))}
  \end{equation*}
  for $u,v\in W$.
\end{remark}

The next lemma shows that the linear map
${\gamma(u)\co \Lambda^0\oplus\Lambda^2_7\to\Lambda^1}$ is dual to the
map $m(u,\cdot)\co W\to\Lambda^0\oplus\Lambda^2_7$ for every $u\in W$
and that it satisfies equation~\eqref{eq:spindim8-1}.

\begin{lemma}
  \label{le:spinW}
  Let $\gamma\co W\to\Hom(S^+,S^-)$ be the homomorphism
  in~\eqref{eq:Spm} and~\eqref{eq:gamma}. %
  Then $\gamma$ satisfies~\eqref{eq:spindim8-1} and
  \begin{equation}
    \label{eq:gammaW1}
    \gamma(u)^*v^*
    = m(u,v)
    = \left(
      \inner{u}{v},
      \tfrac{1}{2}(u^*\wedge v^*+\om_{u,v})
    \right)
  \end{equation}
  for all $u,v\in W$. 
\end{lemma}

\begin{proof}
  For $u\in W$, $\lambda\in\R$, $\om\in\Lambda^2_7$, and $v\in W$ we
  compute
  \begin{align*}
    \inner{\gamma(u)(\lambda,\om)}{v^*}
    &=
      \inner{\lambda u^*+2\iota(u)\om}{v^*} \\
    &=
      \lambda\inner{u}{v} + 2\inner{\om}{u^*\wedge v^*} \\
    &=
      \lambda\inner{u}{v} + \tfrac{1}{2}\Inner{\om}{\om_{u,v}+u^*\wedge v^*}.
  \end{align*}
  The last equation follows from the fact that
  \begin{equation*}
    {\pi_7(u^*\wedge v^*)
      = \tfrac{1}{4}(u^*\wedge v^*+\om_{u,v})}.
  \end{equation*}
  This proves~\eqref{eq:gammaW1}. %
  With this understood, the formula
  $\gamma(u)^*\gamma(u) = \abs{u}^2\one$ follows directly
  from~\eqref{eq:mnorm}. %
  This proves \autoref{le:spinW}.
\end{proof}

Combining the product map $m$ with the triple cross product 
we obtain an alternating multi-linear map 
$\tau\co W^4\to\Lambda^0\oplus\Lambda^2_7$
defined by 
\begin{equation}
  \label{eq:tau}
  \begin{split}
    \tau(x,u,v,w)
    &=\, \tfrac{1}{4}
    \Bigl(m(u \times v \times w,x)
    - m(v \times w \times x,u) \\
    &\quad\quad
    + m(w \times x \times u,v)
    - m(x \times u \times v, w)\Bigr).
  \end{split}
\end{equation}
This map corresponds to the four-fold cross product (see
\autoref{def:four}) and has the following properties (see
\autoref{thm:four}).

\begin{lemma}
  \label{le:chi}
  Let $\chi\co W^4\to\Lambda^2_7$ denote the second component of
  $\tau$. %
  Then, for all $u,v,w,x\in W$, we have 
  \begin{gather*}
    \tau(x,u,v,w) = \left(\Phi(x,u,v,w),\chi(x,u,v,w)\right), \\
    \Phi(x,u,v,w)^2 + \abs{\chi(x,u,v,w)}^2 = \abs{x\wedge u\wedge v\wedge w}^2.
  \end{gather*}
\end{lemma}

\begin{proof}
  That the first component of $\tau$ is equal to $\Phi$ follows
  directly from the definitions. %
  Moreover, for $u,v,w,x\in W$, we have
  \begin{equation}
    \label{eq:chi}
    \begin{split}
      2\chi(x,u,v,w)
      =&\,(u \times v \times w)^*\wedge x^*+\om_{u\times v\times w,x} \\
      & - (v \times w \times x)^*\wedge u^*-\om_{v\times w\times x,u} \\
      & + (w \times x \times u)^*\wedge v^*+\om_{w\times x\times u,v} \\
      & - (x \times u \times v)^*\wedge w^*-\om_{x\times u\times v,w}.
    \end{split}
  \end{equation}
  We claim that the four rows on the right agree whenever $u,v,w,x$
  are pairwise orthogonal. %
  Under this assumption the first two rows remain unchanged if we add
  to $x$ a multiple of $u\times v\times w$. %
  Thus we may assume that $x$ is orthogonal to $u$, $v$, $w$, and
  $u\times v\times w$. %
  By \autoref{thm:Spin7S}, we may therefore assume that $W=\R^8$ with
  the standard triple cross product and
  \begin{equation*}
    u = e_0,\qquad v=e_1,\qquad w=e_2,\qquad x=e_4.
  \end{equation*}  
  In this case a direct computation proves that the first two rows
  agree. %
  Thus we have proved that, if $u,v,w,x\in W$ are pairwise orthogonal,
  then
  \begin{equation*}
    \tau(x,u,v,w) = m(u\times v\times w,x).
  \end{equation*}
  In this case it follows from~\eqref{eq:mnorm} that
  \begin{align*}
    \abs{\tau(x,u,v,w)}
    &= \abs{m(u \times v \times w,x)} \\
    &= \abs{x}\abs{u\times v\times w} \\
    &= \abs{x}\abs{u}\abs{v}\abs{w} \\
    &= \abs{x\wedge u\wedge v\wedge w}.
  \end{align*}
  Since $\tau$ is alternating, this proves \autoref{le:chi}.
\end{proof}

\begin{lemma}\label{le:spinWcross}
  Let $\gamma\co W\to\Hom(S^+,S^-)$ be the homomorphism
  in~\eqref{eq:Spm} and~\eqref{eq:gamma2}. %
  Then $\gamma$ satisfies~\eqref{eq:spindim8-2}
  and~\eqref{eq:triplecrossgamma}.
\end{lemma}

\begin{proof}
  It follows from~\eqref{eq:spindim8-1} that
  $\inner{\gamma(u)s}{\gamma(v)s}=\inner{u}{v}$ for all $u,v\in W$. %
  Hence, equation~\eqref{eq:triplecrossgamma} is equivalent to
  \begin{equation}\label{eq:triplecrossPhi}
    \begin{split}
      \Phi(x,u,v,w) 
      &= \inner{x}{u}\inner{v}{w} - \inner{x}{v}\inner{w}{u} + \inner{x}{w}\inner{u}{v} \\
      &\quad - \inner{\gamma(u)^*\gamma(x)s}{\gamma(v)^*\gamma(w)s}
    \end{split}
  \end{equation}
  for all $x,u,v,w\in W$.  %
  Since $s=(1,0)\in S^+=\Lambda^0\oplus\Lambda^2_7$, we have
  \begin{equation*}
    \gamma(u)^*\gamma(x)s
    =\gamma(u)^*x^*
    =\left(\inner{u}{x},\tfrac12(u^*\wedge x^*+\om_{u,x})\right)
  \end{equation*} 
  for all $u,x\in W$ by \autoref{le:spinW}. %
  Hence,
  \begin{equation*}
    \begin{split}
      &\inner{\gamma(u)^*\gamma(x)s}{\gamma(v)^*\gamma(w)s} \\
      &\quad= 
      \inner{u}{x}\inner{v}{w} 
      + \tfrac14\inner{u^*\wedge x^*+\om_{u,x}}{v^*\wedge w^*+\om_{v,w}} \\
      &\quad= 
      \inner{u}{x}\inner{v}{w} + \inner{u}{v}\inner{x}{w} - \inner{u}{w}\inner{x}{v} 
      + \Phi(u,x,v,w).
    \end{split}
  \end{equation*}
  Here the last equation follows from \autoref{cor:uvxy}. %
  This shows that 
  the homomorphism
  $\gamma$ satisfies~\eqref{eq:triplecrossPhi} 
  and hence also~\eqref{eq:triplecrossgamma}. %
  
  We prove that $\gamma$ satisfies~\eqref{eq:spindim8-2}. %
  Choose an orthonormal basis $e_0,\dots,e_7$ 
  of $W$ 
  in which $\Phi$ has the standard form of \autoref{ex:cayley}. %
  Such a basis exists by \autoref{thm:CAYLEY1} because $\Phi$ is a
  positive Cayley form, and it is positive because
  $\Phi\wedge\Phi>0$. %
  Moreover, for any quadruple of integers $0\le i<j<k<\ell\le7$, the
  following are equivalent.
  \begin{enumerate}[(a)]
  \item
    \label{It_SWCa}
    The term $\pm e^{ijk\ell}$ appears in the standard basis.
  \item
    \label{It_SWCb}
    $\Phi(e_i,e_j,e_k,e_\ell)=\pm1$.
  \item
    \label{It_SWCc}
    $e_k\times e_j\times e_i=\pm e_\ell$. 
  \item
    \label{It_SWCd}
    $-\gamma(e_k)\gamma(e_j)^*\gamma(e_i)s=\pm\gamma(e_\ell)s$.
  \end{enumerate}
  Here the equivalence of~\itref{It_SWCa} and~\itref{It_SWCb} is
  obvious, the equivalence of~\itref{It_SWCb} and~\itref{It_SWCc}
  follows from the fact that
  $$
    \Phi(e_i,e_j,e_k,e_\ell)=\Phi(e_\ell,e_k,e_j,e_i)=\inner{e_k\times
    e_j\times e_i}{e_\ell}
  $$
  by~\eqref{eq:crossPhi}, and the equivalence of~\itref{It_SWCc}
  and~\itref{It_SWCd} follows from
  equation~\eqref{eq:triplecrossgamma}. %
  Examining the relevant terms in \autoref{ex:cayley} we find that
  $$
  \gamma(e_2)\gamma(e_1)^*\gamma(e_0)s=-\gamma(e_3)s, 
  $$
  hence
  $$
  \gamma(e_4)\gamma(e_3)^*\gamma(e_2)\gamma(e_1)^*\gamma(e_0)s=-\gamma(e_4)s,
  $$
  hence
  \begin{equation*}
    \gamma(e_6)\gamma(e_5)^*\gamma(e_4)\gamma(e_3)^*\gamma(e_2)\gamma(e_1)^*\gamma(e_0)s
    =-\gamma(e_6)\gamma(e_5)^*\gamma(e_4)s
    =\gamma(e_7)s, 
  \end{equation*}
  and hence 
  \begin{equation*}
    \gamma(e_7)^*\gamma(e_6)\gamma(e_5)^*\gamma(e_4)
    \gamma(e_3)^*\gamma(e_2)\gamma(e_1)^*\gamma(e_0)s
    =s.  
  \end{equation*}
  Hence, $\gamma$ satisfies~\eqref{eq:spindim8-2} and this proves \autoref{le:spinWcross}.
\end{proof}

\begin{proof}[Proof of~\autoref{thm:spintriplecross}]
  \label{proof:spintriplecross}
  Part~\itref{spintriplecross1} follows 
  from \autoref{le:spinW} and \autoref{le:spinWcross}. %
  To prove part~\itref{spintriplecross2}
  assume~${\gamma\co W\to\Hom(S^+,S^-)}$ is a spin structure,
  let~${s\in S^+}$ be a unit vector, and define the multilinear map
  \begin{equation}\label{eq:triplecross-spin}
    W^3\to W\co (u,v,w)\mapsto u\times v\times w
  \end{equation}
  by~\eqref{eq:triplecrossgamma}. %
  Then $u\times v\times w=0$ whenever two of the three vectors agree.
  Hence, it suffices to verify~\eqref{eq:tc1} and~\eqref{eq:tc2} under
  the assumption that $u,v,w$ are pairwise orthogonal. %
  In this case we compute
  \begin{align*}
    \inner{u\times v\times w}{u}
    &=
      \inner{\gamma(u\times v\times w)s}{\gamma(u)s} \\
    &=
      - \inner{\gamma(u)\gamma(v)^*\gamma(w)s}{\gamma(u)s} \\
    &=
      - \abs{u}^2\inner{\gamma(v)^*\gamma(w)s}{s} \\
    &=
      - \abs{u}^2\inner{v}{w} = 0.
  \end{align*}
  and
  \begin{align*}
    \abs{u\times v\times w}^2
    &=
      \abs{\gamma(u\times v\times w)s}^2 \\
    &= 
      \abs{\gamma(u)\gamma(v)^*\gamma(w)s}^2 \\
    &=
      \abs{u}^2\abs{v}^2\abs{w}^2 \\
    &=
      \abs{u\wedge v\wedge w}^2.
  \end{align*}
  This shows that the map~\eqref{eq:triplecross-spin} 
  is a triple cross product.  %
  To prove that it is positive, choose a quadruple of pairwise 
  orthogonal vectors $e,u,v,w\in W$ such that $w$ is also
  orthogonal to $e\times u\times v$.  %
  Then 
  \begin{align*}
    \gamma(e\times u\times(e\times v\times w))s
    &= - \gamma(e)\gamma(u)^*\gamma(e\times v\times w)s \\
    &= \gamma(e)\gamma(u)^*\gamma(e)\gamma(v)^*\gamma(w)s \\
    &= - \gamma(e)\gamma(e)^*\gamma(u)\gamma(v)^*\gamma(w)s \\
    &= - \abs{e}^2\gamma(u)\gamma(v)^*\gamma(w)s \\
    &= \abs{e}^2\gamma(u\times v\times w)s.
  \end{align*}
  Here the first, second, and fifth equalities follow
  from~\eqref{eq:triplecrossgamma} and the third and fourth equalities
  follow from~\eqref{eq:spindim8-1}.  %
  Thus we have proved
  that the triple cross product~\eqref{eq:triplecross-spin} is positive. %
  That the associated Cayley calibration $\Phi$ satisfies
  ${\Phi\wedge\Phi>0}$ follows by using a standard basis and reversing
  the argument in the proof of \autoref{le:spinWcross}. %
  This proves \autoref{thm:spintriplecross}.
\end{proof}


\subsection{Spin structures and cross products}
\label{sec:spincross}

Let $V$ be a $7$--dimensional oriented real Hilbert space. %
A spin structure on $V$ is an $8$-dimensional real Hilbert space $S$
equipped with a vector space homomorphism $\gamma:V\to\End(S)$
that satisfies the conditions 
\begin{equation}
  \label{eq:spindim7-1}
  \gamma(u)^*+\gamma(u)=0,\qquad
  \gamma(u)^*\gamma(u) = \abs{u}^2\one
\end{equation}
for all $u\in V$ (see~\cite{Salamon1999a}*{Definition~4.32})
and 
\begin{equation}
  \label{eq:spindim7-2}
  \gamma(e_7)\gamma(e_6)\cdots\gamma(e_1) 
  = -\one.
\end{equation}
for some, and hence every, positive orthonormal basis $e_1,\dots,e_7$
of $V$. %
Equation~\eqref{eq:spindim7-1} guarantees that the 
linear map $\gamma\co V\to\End(S)$ extends uniquely to an 
algebra homomorphism $\gamma\co \Cl(V)\to\End(S)$
(see~\cite{Salamon1999a}*{Proposition~4.33}). %
It follows from~\eqref{eq:spindim7-2} that the kernel of this extended
homomorphism is given by $\set{x\in\Cl(V):\eps x=x}$, where
$\eps:=e_7\cdots e_1\in \Cl_7(V)$ for a positive orthonormal basis
$e_1,\dots,e_7$ of~$V$ (see~\cite{Salamon1999a}*{Proposition~3.34}). %
Since $\eps$ is an odd element of $\Cl(V)$, this implies that the
restrictions of $\gamma$ to both $\Cl^\ev(V)$ and $\Cl^\odd(V)$ are
injective. %
Since $\dim\,\Cl^\ev(V)=\dim\,\Cl^\odd(V)=\dim\End(S)=64$, it follows
that $\gamma$ restricts to an algebra isomorphism from $\Cl^\ev(V)$ to
$\End(S)$ and to a vector space isomorphism from $\Cl^\odd(V)$ to
$\End(S)$. %

\begin{theorem}\label{thm:spincross}
  Let $V$ be an oriented $7$--dimensional real Hilbert space.
  \begin{enumerate}[(i)]
  \item\label{spincross1}
    Suppose $V$ is equipped with a cross product and define the 
    homomorphism $\gamma\co V\to\End(S)$ by 
    \begin{equation}
      \label{eq:spin7gamma}
      S:=\R\times V,\qquad 
      \gamma(u)(\lambda,v) := \left(-\inner{u}{v},\lambda u+u\times v\right)
    \end{equation}
    for $\lambda\in\R$ and $u,v\in V$. %
    Then $\gamma$ is a spin structure on~$V$, i.e., it
    satisfies~\eqref{eq:spindim7-1} and~\eqref{eq:spindim7-2}. %
    Moreover, the space $S=\R\times V$ contains a canonical unit
    vector $s=(1,0)$ and the cross product can be recovered from the
    spin structure and the unit spinor via the formula
    \begin{equation}
      \label{eq:spin7cross}
      \gamma(u\times v)s =  \gamma(u)\gamma(v)s + \inner{u}{v}s
      \qquad\mbox{for }u,v\in V.
    \end{equation}

  \item\label{spincross2} Let $\gamma\co V\to\End(S)$ be a spin
    structure and let $s\in S$ be a unit vector. %
    Then equation~\eqref{eq:spin7cross} defines a cross product on~$V$
    that is compatible with the inner product and orientation. %
    Since any two spin structures on $V$ are isomorphic, this shows
    that there is a one-to-one correspondence between unit spinors and
    cross products on $V$ that are compatible with the
    inner product and orientation.
  \end{enumerate}
\end{theorem}

\begin{proof}
  We prove part~\itref{spincross1}. %
  Thus assume $V$ is equipped with a cross product that is compatible
  with the inner product and orientation, and let
  $\gamma\co V\to\End(S)$ be given by~\eqref{eq:spin7gamma}. %
  Then, for $u,v,w\in V$ and $\lambda,\mu\in\R$, we have
  \begin{equation*}
    \inner{(\lambda,v)}{\gamma(u)(\mu,w)}
    = \mu\inner{u}{v}-\lambda\inner{u}{w} +\phi(v,u,w).
  \end{equation*}
  This expression is skew-symmetric in $(\lambda,v)$ and $(\mu,w)$
  and so $\gamma(u)$ is skew-adjoint. %
  Moreover, for $u,v,w\in V$ and $\mu\in\R$, we have
  \begin{align*}
    &
      \gamma(u)\gamma(v)(\mu,w) 
      + \inner{u}{v}(\mu,w)\\
    &=
      \left(-\inner{u}{\mu v+v\times w},
      -\inner{v}{w}u+u\times (\mu v+ v\times w)\right) 
      + \inner{u}{v}(\mu,w) \\
    &= 
      \left(-\inner{u\times v}{w},\mu(u\times v) 
      + u\times(v\times w) - \inner{v}{w}u+\inner{u}{v}w\right)  \\
    &= 
      \gamma(u\times v)(\mu,w)
      +\left(0,-(u\times v)\times w - \inner{v}{w}u +\inner{u}{w}v\right) \\
    &\quad
      +\left(0,-(v\times w)\times u - \inner{u}{w}v + \inner{u}{v}w \right)  \\
    &= 
      \gamma(u\times v)(\mu,w) - 2(0,[u,v,w]).
  \end{align*}
  Here the last equation follows from~\eqref{eq:associator}. %
  This proves~\eqref{eq:spindim7-1} by taking $v=-u$ 
  and~\eqref{eq:spin7cross} by taking $\mu=1$ and $w=0$. %
  For the proof of~\eqref{eq:spindim7-2} it is convenient to use 
  the standard basis for the standard cross product 
  on $V=\R^7$ in \autoref{ex:cross7}. %
  The left hand side of~\eqref{eq:spindim7-2} is independent of the
  choice of the positive orthonormal basis and we know from general
  principles that the composition $\gamma(e_7)\cdots\gamma(e_1)$ must
  equal $\pm\one$ (see~\cite{Salamon1999a}*{Prop~4.34}). %
  The sign can thus be determined by evaluating the composition of the
  $\gamma(e_j)$ on a single nonzero vector. %
  We leave the verification to the reader. %
  This proves part~\itref{spincross1}.

  We prove part~\itref{spincross2}. %
  Thus assume that $\gamma\co V\to\End(S)$ is a spin structure compatible
  with the orientation and let $s\in S$ be a unit vector. %
  Then the map
  \begin{equation}
    \label{eq:Xi}
    \R\times V\to S\co (\lambda,v)\mapsto \Xi(\lambda,v) := \lambda s+\gamma(v)s
  \end{equation}
  is an isometric isomorphism, because
  $\abs{\lambda s+\gamma(v)s}^2=\abs{\lambda}^2+\abs{v}^2$
  by~\eqref{eq:spindim7-1} and both spaces have the same dimension. %
  For $u,v\in V$ the first coordinate of $\Xi^{-1}\gamma(u)\gamma(v)s$
  is $\inner{s}{\gamma(u)\gamma(v)s}=-\inner{u}{v}$ and so the second
  coordinate is the vector $u\times v\in V$ that
  satisfies~\eqref{eq:spin7cross}. %
  The map $V\times V\to V\co (u,v)\mapsto u\times v$ is obviously
  bilinear and it is skew symmetric because
  $\gamma(u)\gamma(v)+\gamma(v)\gamma(u) = - 2\inner{u}{v}\one$
  by~\eqref{eq:spindim7-1}. %
  It satisfies~\eqref{eq:orthogonal} and ~\eqref{eq:crosscomplex}
  because
  \begin{equation*}
    \begin{split}
      \inner{u}{u\times v}
      &= \inner{\gamma(u)s}{\gamma(u\times v)s}  
      = \inner{\gamma(u)s}{\gamma(u)\gamma(v)s+\inner{u}{v}s}  
      = 0, \\
      \gamma(u\times(u\times v))s
      &= \gamma(u)\gamma(u\times v)s 
      = \gamma(u)\bigl(\gamma(u)\gamma(v)s+\inner{u}{v}s\bigr) \\
      &= \gamma\bigl(\inner{u}{v}u - \abs{u}^2v\bigr)s.
    \end{split}
  \end{equation*}
  for all $u,v\in V$. %
  Hence, it is a cross product by \autoref{le:area}. %
  That it is compatible with the orientation can be proved 
  by choosing a standard basis as in \autoref{ex:cross7}. %
  This proves \autoref{thm:spincross}.
\end{proof}

We close this section with some useful identities.

\begin{lemma}\label{le:spincross}
  Fix a spin structure $\gamma\co V\to\End(S)$ that is compatible with
  the orientation and a unit vector $s\in S$, 
  let $V\times V\to V\co (u,v)\mapsto u\times v$ be the
  cross product determined by~\eqref{eq:spin7cross}, and let
  $\Xi\co \R\times V\to S$ be the isomorphism in~\eqref{eq:Xi}.  Then
  the following hold:
  \begin{enumerate}[(i)]
  \item
    \label{le:spincross1}
    The spin structure $\gamma$ is isomorphic 
    to the spin structure in~\eqref{eq:spin7gamma} via $\Xi$, 
    i.e., for all $\lambda\in\R$ and all $u,v\in V$, we have
    \begin{equation}\label{eq:spin7Xi}
      \Xi^{-1}\gamma(u)\Xi(\lambda,v) = (-\inner{u}{v},\lambda u+u\times v)
    \end{equation}
  \item
    \label{le:spincross2}
    For all $u,v,w\in V$ we have
    \begin{equation}\label{eq:spin7assocbrack}
      \begin{split}
        &\gamma([u,v,w])s + \phi(u,v,w)s + \gamma(u)\gamma(v)\gamma(w)s \\
        &\quad= - \inner{v}{w}\gamma(u)s + \inner{w}{u}\gamma(v)s - \inner{u}{v}\gamma(w)s.
      \end{split}
    \end{equation}
  \item
    \label{le:spincross3}
    The associative calibration $\phi\in\Lambda^3V^*$ is given by 
    \begin{equation}\label{eq:spin7assoc}
      \phi(u,v,w) = -\inner{s}{\gamma(u)\gamma(v)\gamma(w)s}
    \end{equation}
    and the coassociative calibration $\psi=*\phi\in\Lambda^4V^*$ is given by
    \begin{equation}\label{eq:spin7coassoc}
      \begin{split}
        \psi(u,v,w,x)
        &= - \inner{s}{\gamma(u)\gamma(v)\gamma(w)\gamma(x)s} \\
        &\quad + \inner{v}{w}\inner{u}{x} - \inner{w}{u}\inner{v}{x} + \inner{u}{v}\inner{w}{x}.
      \end{split}
    \end{equation}
  \end{enumerate}  
\end{lemma}

\begin{proof}
  Part~\itref{le:spincross1} follows from~\eqref{eq:spin7cross} by
  direct calculation. %
  By~\itref{le:spincross1} the second displayed formula in the proof
  of \autoref{thm:spincross} with $\mu=0$ can be expressed as
  \begin{equation*}
    \begin{split}
      &\gamma(u)\gamma(v)\gamma(w)s + \inner{u}{v}\gamma(w)s \\
      &= 
      \gamma(u\times v)\gamma(w)s - 2\gamma([u,v,w])s \\
      &= 
      - 2\inner{u\times v}{w}s - 2\gamma([u,v,w])s - \gamma(w)\gamma(u\times v)s \\
      &= 
      - 2\phi(u,v,w)s - 2\gamma([u,v,w])s
      - \gamma(w)\gamma(u)\gamma(v)s - \inner{u}{v}\gamma(w)s \\
      &= 
      - 2\phi(u,v,w)s - 2\gamma([u,v,w])s \\
      &\quad
      + \gamma(u)\gamma(w)\gamma(v)s + 2\inner{w}{u}\gamma(v)s
      - \inner{u}{v}\gamma(w)s \\
      &= 
      - 2\phi(u,v,w)s - 2\gamma([u,v,w])s \\
      &\quad
      - \gamma(u)\gamma(v)\gamma(w)s 
      - 2\inner{v}{w}\gamma(u)s + 2\inner{w}{u}\gamma(v)s
      - \inner{u}{v}\gamma(w)s 
    \end{split}
  \end{equation*}
  for all $u,v,w\in V$ and this proves~\itref{le:spincross2}.
  Part~\itref{le:spincross3} follows from~\itref{le:spincross2} by
  taking the inner product with $s$, respectively with $\gamma(x)s$
  (see \autoref{le:psi}). %
  This proves Lemma~\ref{le:spincross}.
\end{proof}



\section{Octonions and complex linear algebra}
\label{sec:SU}

Let $W$ be a $2n$--dimensional real vector space. %
An \defined{$\SU(n)$--structure on $W$} is a triple $(\om,J,\theta)$
consisting of a nondegenerate $2$--form $\om\in\Lambda^2W^*$, an
$\om$--compatible complex structure $J\co W\to W$ (so that
$\inner{\cdot}{\cdot}:=\om(\cdot,J\cdot)$ is an inner product), and a
complex multi-linear map $\theta\co W^n\to\C$ which has norm $2^{n/2}$
with respect to the metric determined by $\om$ and $J$. %
The archetypal example is $W=\C^n$ with the standard symplectic form
\begin{equation*}
  \om:=\sum_jdx_j\wedge dy_j,
\end{equation*}
the standard complex structure $J:=i$, and the standard
$(n,0)$--form
\begin{equation*}
  \theta:=dz_1\wedge\cdots\wedge dz_n.
\end{equation*} 
In this section we examine the relation between 
$\SU(3)$--structures and cross products and between 
$\SU(4)$--structures and triple cross products. %
We also explain the decompositions of \autoref{thm:form7} 
and \autoref{thm:form8} in this setting.

\begin{theorem}
  \label{thm:form7SU3}
  Let $W$ be 
  a
  $6$--dimensional real vector space equipped with an
  $\SU(3)$--structure $(\omega,J,\theta)$. %
  Then the space $V:=\R\oplus W$ carries a natural cross product
  defined by
  \begin{equation}
    \label{eq:crossSU3}
    v\times w:=\left(\om(v_1,w_1),
      v_0Jw_1-w_0Jv_1+v_1\times_\theta w_1\right)
  \end{equation}
  for $u=(u_0,u_1),v=(v_0,v_1)\in\R\oplus W$, where
  $v_1\times_\theta w_1\in V$ is defined by
  $\inner{u_1}{v_1\times_\theta w_1}:=\RE\,\theta(u_1,v_1,w_1)$ for
  all $u_1\in W$. %
  The associative calibration of this cross product is
  \begin{equation}
    \label{eq:phiSU3}
    \phi := e^0 \wedge \omega + \RE\,\theta\in\Lambda^3V^*
  \end{equation}
  and the coassociative calibration is
  \begin{equation}
    \label{eq:psiSU3}
    \psi := *\phi = \tfrac{1}{2} \omega\wedge\omega - e^0 \wedge \IM\,\theta
    \in\Lambda^4V^*.
  \end{equation}
  Moreover, the subspaces $\Lambda^k_d\subset\Lambda^kV^*$ in 
  \autoref{thm:form7} are given by
  \begin{align*}
    \Lambda^2_7 
    & = \R \omega
      \oplus \{ e^0\wedge u^* - \iota(u) \IM\,\theta  :  u \in W \}, \\
    \Lambda^2_{14} 
    &= \left\{\tau- e^0 \wedge *_W(\tau \wedge \RE\,\theta) : 
      \tau \in \Lambda^2 W^*,\;\tau\wedge\omega\wedge\omega = 0 \right\}, \\
    \Lambda^3_7 
    &= \R\cdot\IM\,\theta\oplus 
      \left\{u^*\wedge\om-e^0\wedge\iota(u)\RE\,\theta : u\in W\right\}, \\
    \Lambda^3_{27} 
    &= \R\cdot\left(3\RE\,\theta-4e^0\wedge\om\right)  \\
    &\quad\oplus
      \left\{e^0\wedge\tau : \tau\in\Lambda^{1,1}W^*,\,
      \tau\wedge\om\wedge\om=0\right\} \\
    &\quad\oplus
      \left\{\beta\in\Lambda^{2,1}W^*+\Lambda^{1,2}W^* : 
      \beta\wedge\om=0\right\} \\
    &\quad\oplus
      \left\{u^*\wedge\om+e^0\wedge\iota(u)\RE\,\theta : u\in W\right\}.
  \end{align*}
\end{theorem}

\begin{proof}
  For $v,w\in W$ we define $\alpha_{v,w}\in\Lambda^1W^*$ by
  $\alpha_{v,w}:=\RE\,\theta(\cdot,v,w)$. %
  Then $\abs{\alpha_{v,w}}=\abs{\theta(u,v,w)}=\abs{v}\abs{w}$
  whenever $u,Ju,v,Jv,w,Jw$ are pairwise orthogonal and $\abs{u}=1$. %
  This implies
  \begin{equation}
    \label{eq:normSU3}
    \abs{\alpha_{v,w}}^2 + \om(v,w)^2 
    + \inner{v}{w}^2 = \abs{v}^2\abs{w}^2
  \end{equation}
  for all $v,w\in W$. %
  (Add to $w$ a suitable linear combination of $v$ and $Jv$.) %
  It follows from~\eqref{eq:normSU3} by direct computation that the
  formula~\eqref{eq:crossSU3} defines a cross product on $\R\times
  W$. %
  By~\eqref{eq:crossSU3} and~\eqref{eq:phiSU3}, we have
  ${\phi(u,v,w)=\inner{u}{v\times w}}$ so that $\phi$ is the
  associative calibration of~\eqref{eq:crossSU3} as claimed. %
  That $\phi$ is compatible with the orientation of $\R\oplus W$
  follows from the fact that $\iota(e_0)\phi=\om$ and
  $\om\wedge \RE\,\theta=0$ so that
  $\iota(e_0)\phi\wedge\iota(e_0)\phi\wedge\phi =
  e^0\wedge\om^3=6\dvol$. %
  The formula~\eqref{eq:psiSU3} for $\psi:=*\phi$ follows from the
  fact that $\om\wedge\theta=0$ and $\IM\,\theta=*\RE\,\theta$ so that
  $\RE\,\theta\wedge\IM\,\theta=4\dvol_W$. %
  It remains to examine the subspaces $\Lambda^k_d\subset\Lambda^kV^*$
  introduced in \autoref{thm:form7}.

  The formula for $\Lambda^2_7$ follows directly from the formula for
  $\phi$ in~\eqref{eq:phiSU3} and the fact that $\Lambda^2_7$ consists
  of all $2$--forms $\iota(v)\phi$ for $v\in\R\oplus W$. %
  With $v=(1,0)$ we obtain $\iota(v)\phi=\om$ and with $v=(0,Ju)$ we
  obtain
  \begin{equation*}
    \iota(v)\phi
    = -e^0\wedge\iota(Ju)\om + \iota(Ju)\RE\,\theta
    = e^0\wedge u^* - \iota(u)\IM\,\theta.
  \end{equation*} 
  Similarly, the formula for $\Lambda^3_7$ follows directly from the
  formula for $\psi$ in~\eqref{eq:phiSU3} and the fact that
  $\Lambda^3_7$ consists of all $3$--forms $\iota(v)\psi$ for
  $v\in\R\oplus W$. %
  With $v=(-1,0)$ we obtain $\iota(v)\psi=\IM\,\theta$ and with
  $v=(0,-Ju)$ we obtain
  \begin{equation*}
    \iota(v)\psi 
    = -(\iota(Ju)\om)\wedge\om
    - e^0\wedge\iota(Ju)\IM\,\theta 
    = u^*\wedge\om - e^0\wedge\iota(u)\RE\,\theta.
  \end{equation*}
  To prove the formula for $\Lambda^2_{14}$ we choose 
  $\alpha\in\Lambda^1W^*$ and $\tau\in\Lambda^2W^*$.
  Then $\tau+e^0\wedge\alpha\in\Lambda^2_{14}$ if and only 
  if $(\tau+e^0\wedge\alpha)\wedge\psi=0$. %
  By~\eqref{eq:psiSU3}, we have
  \begin{align*}
    (e^0 \wedge \alpha + \tau) \wedge \psi
    &= \left(e^0 \wedge \alpha + \tau\right)
      \wedge \left(\tfrac{1}{2} \omega\wedge\omega 
      - e^0\wedge \IM\,\theta\right) \\
    &= e^0 \wedge 
      \left(
      \tfrac{1}{2}\om\wedge\om\wedge\alpha 
      - \tau\wedge\IM\,\theta
      \right)
      + \tfrac{1}{2}\tau\wedge\omega\wedge\omega.
  \end{align*}
  The expression on the right vanishes if and only if
  $\tau\wedge\om\wedge\om=0$ and 
  $\om\wedge\om\wedge\alpha=2\IM\,\theta\wedge\tau$. %
  Since $\alpha\circ J = \tfrac{1}{2}*_W(\om\wedge\om\wedge\alpha)$, 
  the last equation is equivalent to 
  $\alpha=-\left(*_W(\IM\,\theta\wedge\tau)\right)\circ J
  = - *_W(\RE\,\theta\wedge\tau)$. 

  To prove the formula for $\Lambda^3_{27}$ we choose 
  $\tau \in \Lambda^2 W^*$ and $\beta \in \Lambda^3 W^*$. %
  Then
  \begin{align*}
    \left(\beta+e^0 \wedge \tau\right)\wedge \phi
    &= e^0 \wedge \left(\tau\wedge\RE\,\theta-\beta\wedge\om\right)
      + \beta \wedge \RE\,\theta, \\
    \left(\beta + e^0 \wedge \tau\right)\wedge \psi
    &= e^0 \wedge \left(\tfrac{1}{2}\tau \wedge \omega \wedge \omega
      + \beta \wedge \IM\,\theta\right).
  \end{align*}
  Both terms vanish simultaneously if and only if
  \begin{gather*}
    \tau \wedge \RE\,\theta
    = \beta \wedge \omega,\quad
    \beta\wedge\RE\,\theta
    = 0, \quad
    \beta\wedge\IM\,\theta 
    = -\frac12 \tau \wedge \omega \wedge \omega.
  \end{gather*}
  These equations hold in the following four cases.

  \begin{enumerate}[(a)]
  \item
    \label{Pf_form7SU3a}
    $\beta=3\lambda\RE\,\theta$ and $\tau=-4\lambda\om$ with
    $\lambda\in\R$.

  \item
    \label{Pf_form7SU3b}
    $\beta=0$ and $\tau\in\Lambda^{1,1}W^*$ with
    $\tau\wedge\om\wedge\om=0$.

  \item
    \label{Pf_form7SU3c}
    $\beta\in\Lambda^{1,2}W^*+\Lambda^{2,1}W^*$ with
    $\beta\wedge\om=0$ and $\tau=0$.

  \item
    \label{Pf_form7SU3d}
    $\beta = u^*\wedge\om$ and $\tau = \iota(u)\RE\,\theta$ with
    $u\in W$.
  \end{enumerate}

  In case~\itref{Pf_form7SU3d} this follows from
  $(\iota(u)\Re\,\theta)\wedge\RE\,\theta = 2 * (Ju)^* =
  u^*\wedge\om\wedge\om $. %
  The subspaces determined by these conditions are pairwise orthogonal
  and have dimensions $1$ in case~(\itref{Pf_form7SU3a}, $8$ in
  case~\itref{Pf_form7SU3b}, $12$ in case~\itref{Pf_form7SU3c}, and
  $6$ in case~\itref{Pf_form7SU3d}. %
  Thus, for dimensional reasons, their direct sum is the space
  $\Lambda^3_{27}$. %
  This proves \autoref{thm:form7SU3}.
\end{proof}

\begin{theorem}
  \label{thm:form8SU4}
  Let $W$ be an $8$--dimensional real vector space equipped with an
  $\SU(4)$--structure $(\Omega,J,\Theta)$. %
  Then the alternating multi-linear map
  \begin{equation*}
    \Phi := \tfrac{1}{2}\Omega\wedge\Omega + \RE\,\Theta\in\Lambda^4W^*
  \end{equation*}
  is a positive Cayley calibration, compatible with the complex
  orientation and the inner product. %
  Moreover, in the notation of \autoref{thm:form8}, we have
  \begin{align*}
    \Lambda^2_7
    &= \R \Omega
      \oplus\left\{\tau\in\Lambda^{2,0}+\Lambda^{0,2} : 
      *\left(\RE\,\Theta\wedge\tau\right)=2\tau\right\}, \\
    \Lambda^2_{21} 
    &= \left\{\tau\in\Lambda^{1,1} : \tau\wedge\Om^3=0\right\}
      \oplus 
      \left\{\tau\in\Lambda^{2,0}+\Lambda^{0,2} : 
      *\left(\RE\,\Theta\wedge\tau\right) = -2\tau\right\}.
  \end{align*}
\end{theorem}

\begin{proof}
  We prove that $\Phi$ is compatible with the inner product
  $ \inner{\cdot}{\cdot}:=\Om(\cdot,J\cdot) $ and the complex
  orientation on $W$. %
  The associated volume form is $\tfrac{1}{24}\Om^4$.
  Hence, by
  \autoref{le:cayley}, we must show that
  \begin{equation}
    \label{eq:PhiSU4}
    \om_{u,v}\wedge\om_{u,v}\wedge\Phi
    = \tfrac{1}{4}\abs{u\wedge v}^2\Om^4
  \end{equation}
  for all $u,v\in W$, where
  \begin{equation*}
    \om_{u,v}:=\iota(v)\iota(u)\Phi
    = \Om(u,v)\Om - \iota(u)\Om\wedge\iota(v)\Om
    + \iota(v)\iota(u)\RE\,\Theta.
  \end{equation*}
  To see this, we observe that
  \begin{equation}
    \label{eq:ThOm}
    \iota(v)\iota(u)\RE\,\Theta 
    \wedge\iota(u)\Om\wedge\iota(v)\Om
    \wedge\Om^2 
    = \left(\iota(v)\iota(u)\RE\,\Theta\right)^2 
    \wedge\RE\,\Theta = 0.
  \end{equation}
  If $v=Ju$, then~\eqref{eq:ThOm} follows from the fact that
  $\iota(u)\Om\wedge\iota(Ju)\Om$ is a $(1,1)$--form and
  $\iota(Ju)\iota(u)\RE\,\Theta=0$. %
  If $v$ is orthogonal to $u$ and $Ju$, then~\eqref{eq:ThOm} follows
  from the explicit formulas in \autoref{rmk:standardSU4} below. %
  The general case follows from the special cases by adding to $v$ a
  linear combination of $u$ and $Ju$. %
  Using~\eqref{eq:ThOm} and the identity
  $ \iota(u)\Om\wedge\iota(v)\Om\wedge\Om^3 =\tfrac{1}{4}\Om(u,v)\Om^4
  $ we obtain
  \begin{align*}
    \om_{u,v}\wedge\om_{u,v}\wedge\Phi
    &=
      \tfrac{1}{2}\Om(u,v)^2\Om^4 
      + \tfrac{1}{2}\iota(v)\iota(u)\RE\,\Theta
        \wedge\iota(v)\iota(u)\RE\,\Theta\wedge\Om^2 \\ &\quad
      - \Om(u,v)\iota(u)\Om\wedge\iota(v)\Om\wedge\Om^3 \\ &\quad
      - 2\iota(v)\iota(u)\RE\,\Theta \wedge
        \iota(u)\Om\wedge\iota(v)\Om\wedge\RE\,\Theta \\
    &=
      \tfrac{1}{4}\Om(u,v)^2\Om^4
      + \tfrac{1}{2}\iota(v)\iota(u)\RE\,\Theta
        \wedge\iota(v)\iota(u)\RE\,\Theta\wedge\Om^2 \\ &\quad
      - 2\iota(v)\iota(u)\RE\,\Theta \wedge
        \iota(u)\Om\wedge\iota(v)\Om\wedge\RE\,\Theta.
  \end{align*}
  One can now verify equation~\eqref{eq:PhiSU4} by first 
  considering the case $v=Ju$ and using $\iota(Ju)\iota(u)\RE\,\Theta=0$
  (here the last two terms on the right vanish). %
  Next one can verify~\eqref{eq:PhiSU4} in the case where 
  $v$ is orthogonal to $u$ and $Ju$ by using the $\SU(4)$--symmetry
  and the explicit formulas in \autoref{rmk:standardSU4} below
  (here the first term on the right vanishes). %
  Finally, one can reduce the general case to the special cases
  by adding to $v$ a linear combination of $u$ and $Ju$.

  Now recall from \autoref{thm:form8} that, 
  for every $\tau\in\Lambda^2W^*$, we have
  \begin{equation*}
    \tau\in\Lambda^2_7
    \qquad\iff\qquad
    *(\Phi\wedge\tau)=3\tau,
  \end{equation*}
  \begin{equation*}
    \tau\in\Lambda^2_{21}
    \qquad\iff\qquad
    *(\Phi\wedge\tau)=-\tau.
  \end{equation*}
  Since $\RE\,\Theta\wedge\Om=0$, we have 
  \begin{equation*}
    *\left(\Phi\wedge\Om\right) 
    = \tfrac{1}{2}*\left(\Om\wedge\Om\wedge\Om\right) 
    = 3\Om
  \end{equation*}
  and, hence, $\R\Om\subset\Lambda^2_7$. %
  Moreover, $\Lambda^2_{21}$ is the image of the Lie algebra $\g$ of
  $\rG(W,\Phi)$ under the isomorphism
  \begin{equation*}
    \so(W)\to\Lambda^2W^*\co \xi\mapsto\om_\xi
  \end{equation*} 
  given by $\om_\xi(u,v):=\inner{u}{\xi v}$. %
  The image of $\su(W)$ under this inclusion is the subspace
  $\left\{\tau\in\Lambda^{1,1}W^* : \tau\wedge\Om^3=0\right\}$ and,
  since $\SU(W)\subset\rG(W,\Phi)$, this space is contained in
  $\Lambda^2_{21}$. %
  By considering the standard structure on $\C^4$ we obtain
  $$
  *(\Omega\wedge\Omega\wedge\tau)=2\tau
  $$ 
  for $\tau\in \Lambda^{2,0}+\Lambda^{0,2}$. %
  Hence,
  \begin{align*}
    *(\Phi\wedge\tau) 
    = \tfrac{1}{2} *(\Omega\wedge\Omega\wedge\tau)
    + *(\RE\,\Theta\wedge\tau)
    = \tau + *(\RE\,\Theta\wedge\tau).
  \end{align*}
  for $\tau\in \Lambda^{2,0}+\Lambda^{0,2}$. %
  Since the operator $\tau\mapsto *(\RE\,\Theta\wedge\tau)$
  has eigenvalues $\pm2$ on the subspace $\Lambda^{2,0}+\Lambda^{0,2}$ 
  the result follows.
\end{proof}

\begin{remark}
  \label{rmk:standardSU4}
  If $(\Om,J,\Theta)$ is the standard $\SU(4)$--structure
  on $W=\C^4$ with coordinates $(x_1+iy_1,\dots,x_4+iy_4)$,
  then
  \begin{align*}
    \RE\,\Theta
    &= 
      dx_1\wedge dx_2\wedge dx_3\wedge dx_4 
      + dy_1\wedge dy_2\wedge dy_3\wedge dy_4 \\ &\quad
      - dx_1\wedge dx_2\wedge dy_3\wedge dy_4 
      - dy_1\wedge dy_2\wedge dx_3\wedge dx_4 \\ &\quad
      - dx_1\wedge dy_2\wedge dx_3\wedge dy_4 
      - dy_1\wedge dx_2\wedge dy_3\wedge dx_4 \\ &\quad
      - dx_1\wedge dy_2\wedge dy_3\wedge dx_4 
      - dy_1\wedge dx_2\wedge dx_3\wedge dy_4
  \intertext{and}
    \tfrac{1}{2}\Om\wedge\Om
    &=
      dx_1\wedge dy_1\wedge dx_2\wedge dy_2 
      + dx_3\wedge dy_3\wedge dx_4\wedge dy_4 \\ &\quad
      + dx_1\wedge dy_1\wedge dx_3\wedge dy_3 
      + dx_2\wedge dy_2\wedge dx_4\wedge dy_4 \\ &\quad
      + dx_1\wedge dy_1\wedge dx_4\wedge dy_4 
      + dx_2\wedge dy_2\wedge dx_3\wedge dy_3.
  \end{align*}
  These forms are self-dual. %
  The first assertion in \autoref{thm:form8SU4} also follows from the
  fact that the isomorphism $\R^8\to\C^4$ which sends $e_0,\dots,e_7$
  to
  \begin{equation*}
    \p/\p x_1,\p/\p y_1,\p/\p x_2,\p/\p y_2, \p/\p x_3,-\p/\p
    y_3,-\p/\p x_4,\p/\p y_4
  \end{equation*}
  pulls back $\Phi$ to the standard form $\Phi_0$ in \autoref{ex:O}.
\end{remark}

\begin{theorem}
  \label{thm:form8G2}
  Let $V$ be a $7$--dimensional real Hilbert space equipped with a
  cross product and its induced orientation. %
  Let $\phi\in\Lambda^3V^*$ be the associative calibration and
  $\psi:=*_V\phi\in\Lambda^4V^*$ the coassociative calibration. %
  Denote $W:=\R\oplus V$ and define $\Phi\in\Lambda^4 W^*$ by
  \begin{equation*}
    \Phi := e^0 \wedge \phi + \psi.
  \end{equation*}
  Then $\Phi$ is a positive Cayley-form on $W$ and,
  in the notation of \autoref{thm:form7} and \autoref{thm:form8},
  we have
  \begin{align*}
    \Lambda^2_7 W^*
    &= 
      \left\{e^0\wedge *_V(\psi\wedge\tau)+3\tau : 
      \tau\in\Lambda^2_7V^*\right\},
    \\
    \Lambda^2_{21}W^*
    &= 
      \left\{e^0\wedge *_V(\psi\wedge\tau)-\tau :  
      \tau\in\Lambda^2 V^* \right\}, 
    \\
    \Lambda^3_8 W^*
    &= 
      \R\phi\oplus\left\{
      \iota(u)\psi-e^0\wedge\iota(u)\phi : u\in V
      \right\},
    \\
    \Lambda^3_{48} W^*
    &= 
      \Lambda^3_{27} V^*
      \oplus \left\{e^0\wedge\tau : \tau \in \Lambda^2_{14} V^*\right\}
      \oplus \left\{3\iota(u)\psi+4e^0\wedge\iota(u)\phi : u\in V\right\},
    \\
    \Lambda^4_7 W^*
    &= 
      \left\{e^0\wedge\iota(u)\psi-u^*\wedge\phi : u\in V\right\},
    \\
    \Lambda^4_{27} W^*
    &= 
      \left\{e^0\wedge\beta+*_V\beta : \beta\in\Lambda^3_{27}V^*\right\}, 
    \\
    \Lambda^4_{35}W^*
    &=
      \left\{e^0\wedge\beta-*_V\beta :  
      \beta\in\Lambda^3V^*\right\}.
  \end{align*}
\end{theorem}

\begin{proof}
  By \autoref{thm:NDG}, $W$ is a normed algebra with
  product~\eqref{eq:uv}. %
  Hence, by \autoref{thm:tcW}, $W$ carries a triple cross
  product~\eqref{eq:tcW4} and $\Phi$ is the associated Cayley
  calibration.  By \autoref{thm:CAYLEY}, $\Phi$ is a Cayley form. %
  By~\eqref{eq:tcW3} the triple cross product on $W$
  satisfies~\eqref{eq:tc6} with $\eps=+1$ and so is positive
  (\autoref{def:tc+}). %
  Thus, by \autoref{thm:CAYLEY1}, $\Phi$ is positive.

  Recall that, by \autoref{thm:form8}, $\Lambda^2_7W^*$ and
  $\Lambda^2_{21}W^*$ are the eigenspaces of the operator
  ${*_W(\Phi\wedge\cdot)}$ with eigenvalues $3$ and $-1$ and, by
  \autoref{thm:form7}, $\Lambda^2_7V^*$ and $\Lambda^2_{14}V^*$ are
  the eigenspaces of the operator ${*_V(\phi\wedge\cdot)}$ with
  eigenvalues $2$ and $-1$. %
  With $\alpha\in\Lambda^1 V^*$ and $\tau\in\Lambda^2V^*$ we have
  \begin{align*}
    *_W\left(\Phi\wedge\bigl(e^0\wedge\alpha+\tau\bigr)\right)
    &= 
      *_W\left(
      e^0\wedge\bigl(\psi\wedge\alpha + \phi\wedge\tau\bigr) 
      + \psi\wedge\tau
      \right) \\
    &= e^0 \wedge *_V(\psi\wedge\tau)
      + *_V(\phi\wedge\tau)
      + *_V(\psi\wedge\alpha)
  \end{align*}
  and, hence,
  \begin{equation*}
    e^0\wedge\alpha+\tau \in \Lambda^2_7 W^*
    \qquad\iff\qquad
    \begin{cases}
      *_V(\psi\wedge\tau) = 3\alpha,\\
      *_V(\phi\wedge\tau) + *_V(\psi\wedge\alpha) = 3\tau.
    \end{cases}
  \end{equation*}
  Since $*_V(\psi\wedge *_V(\psi\wedge\tau))=\tau+*_V(\phi\wedge\tau)$,
  by equation~\eqref{eq:phipsi23} in \autoref{le:phipsi}, 
  we deduce that $e^0\wedge\alpha+\tau\in\Lambda^2_7W^*$ 
  if and only if $*_V(\phi\wedge\tau)=2\tau$ 
  and $3\alpha=*_V(\psi\wedge\tau)$. %
  This proves the formula for $\Lambda^2_7W^*$. %
  Likewise, we have $e^0\wedge\alpha+\tau\in\Lambda^2_{21}W^*$
  if and only if $\alpha=-*_V(\psi\wedge\tau)$. %
  In this case the second equation  
  $*_V(\phi\wedge\tau)+*_V(\psi\wedge\alpha) = -\tau$
  is automatically satisfied.

  The formula for the subspace $\Lambda^3_8W^*$ follows from the fact
  that it consists of all $3$--forms of the form $\iota(u)\Phi$ for
  $u\in W$ (see \autoref{thm:form8}). %
  Now let $\tau\in\Lambda^2V^*$ and $\beta\in\Lambda^3V^*$. %
  Then $e^0\wedge\tau+\beta\in\Lambda^3_{48}W^*$ if and only if
  \begin{equation*}
    0
    = \Phi\wedge\left(e^0\wedge\tau+\beta\right)
    = e^0\wedge\left(\phi\wedge\beta+\psi\wedge\tau\right)
    + \psi\wedge\beta
  \end{equation*}
  (see again \autoref{thm:form8}). %
  Hence,
  \begin{equation*}
    e^0\wedge\tau+\beta\in\Lambda^3_{48}W^*
    \qquad\iff\qquad
    \begin{cases}
      \phi\wedge\beta+\psi\wedge\tau=0,\\
      \psi\wedge\beta=0.
    \end{cases}
  \end{equation*}
  These conditions are satisfied in the following three cases.
  \begin{enumerate}[(a)]
  \item
    \label{Pf_form8G2a}
    $\beta=0$ and $\psi\wedge\tau=0$ (or equivalently
    $\tau\in\Lambda^2_{14}V^*$).

  \item
    \label{Pf_form8G2b}
    $\tau=0$ and $\phi\wedge\beta=0$ and $\psi\wedge\beta=0$ (or
    equivalently $\beta\in\Lambda^3_{27}V^*$).

  \item
    \label{Pf_form8G2c}
    $\beta=3\iota(u)\psi$ and $\tau=4\iota(u)\phi$ with $u\in V$.
  \end{enumerate}
  In the case~\itref{Pf_form8G2a} this follows from the equations
  $\psi\wedge\iota(u)\psi=0$ and 
  \begin{equation}
    \label{eq:3phi4psi}
    3\phi\wedge\iota(u)\psi + 4\psi\wedge\iota(u)\phi = 0
  \end{equation}
  for $u\in V$. %
  This last identity can be verified by direct computation using the
  standard structure on $V=\R^7$ with
  \begin{align*}
    \phi_0 
    &=
      e^{123}- e^{145} - e^{167} - e^{246} + e^{257} - e^{347} - e^{356} \qand \\
    \psi_0
    &=
      - e^{1247} - e^{1256} + e^{1346} - e^{1357} - e^{2345} - e^{2367} + e^{4567},
  \end{align*}
  and $u:=e_1$ (see the proof of \autoref{le:psi}). %
  In this case 
  \begin{equation*}
    \iota(u)\phi_0
      =e^{23}-e^{45}-e^{67}, \qquad
    \iota(u)\psi_0
      =-e^{247}-e^{256}+e^{346}-e^{357}
  \end{equation*}
  and so
  \begin{equation*}
    \psi_0\wedge\iota(u)\phi_0
      = 3e^{234567},\qquad
    \phi_0\wedge\iota(u)\psi_0
      = -4e^{234567}.
  \end{equation*}
  This proves~\eqref{eq:3phi4psi}. %
  The subspaces determined by the above conditions are pairwise orthogonal 
  and have dimensions $14$ in case~\itref{Pf_form8G2a}, $27$ in
  case~\itref{Pf_form8G2b}, and $7$ in case~\itref{Pf_form8G2c}. %
  Thus, for dimensional reasons, their direct sum is $\Lambda^3_{48}W^*$. 

  Now $\Lambda^4_7W^*$ is the tangent space of the $\SO(W)$--orbit of
  $\Phi$. %
  For $u\in V$ define the endomorphism $A_u\in\so(V)$ by
  $A_uv:=u\times v$. %
  Then, by \autoref{rmk:uphipsi}, we have
  $\cL_{A_u}\phi=3\iota(u)\psi$ and $\cL_{A_u}\psi=-3u^*\wedge\phi$.
  Hence
  \begin{equation*}
    e^0\wedge\iota(u)\psi-u^*\wedge\phi\in\Lambda^4_7W^*
  \end{equation*}
  for all $u\in V$. %
  Since $\Lambda^4_7W^*$ has dimension $7$, each element of
  $\Lambda^4_7W^*$ has this form.

  Next we recall that $\Lambda^4_{27}W^*$
  is contained in the subspace of self-dual $4$--forms, 
  and every self-dual $4$--form can be written as 
  $e^0\wedge\beta+*_V\beta$ with $\beta\in\Lambda^3V^*$. %
  By \autoref{thm:form8} we have 
  \begin{align*}
    e^0\wedge\beta+*_V\beta\in\Lambda^4_{27}W^*
    &\iff
      \begin{cases}
        \beta\wedge*_V\phi+*_V\beta\wedge*_V\psi
        =0, \\
        \beta\wedge*_V(\iota(u)\psi)
        = *_V\beta\wedge*_V(u^*\wedge\phi)\;\forall u,
      \end{cases}\\
    &\iff
      \psi\wedge\beta=0,\;\;\phi\wedge\beta=0 \\
    &\iff
      \beta\in\Lambda^3_{27}V^*.
  \end{align*}
  Here the
  last equivalence follows from \autoref{thm:form7}. %
  This proves the formula for~$\Lambda^4_{27}W^*$. %
  The formula for $\Lambda^4_{35}W^*$ follows from the fact that this
  subspace consists of the anti-self-dual $4$--forms. %
  This proves \autoref{thm:form8G2}.
\end{proof}



\section{Donaldson--Thomas theory}
\label{sec:DT}

The motivation for the discussion in these notes came from our attempt
to understand Riemannian manifolds with special holonomy in dimensions
six, seven, and eight~\cites{Bryant1987,Harvey1982,Joyce2000} and the
basic setting of Donaldson--Thomas theory on such
manifolds~\cites{Donaldson1998,Donaldson2009}.


\subsection{Manifolds with special holonomy}

\begin{definition}
  \label{def:almost}
  Let $Y$ be a smooth $7$--manifold and $X$ a smooth
  $8$--mani\-fold. %
  A \defined{$\Gtwo$--structure} on $Y$ is a nondegenerate $3$--form
  $\phi\in\Om^3(Y)$; in this case the pair $(Y,\phi)$ is called an
  \defined{almost $\Gtwo$--manifold}. %
  An \defined{$\Spin(7)$--structure} on $X$ is a $4$--form
  $\Phi\in\Om^4(X)$ which restricts to a positive Cayley-form on each
  tangent space; in this case the pair $(X,\Phi)$ is called an
  \defined{almost $\Spin(7)$--manifold}.
\end{definition}

\begin{remark}
  \label{rmk:G2}
  An almost $\Gtwo$--manifold $(Y,\phi)$ admits a unique Riemannian
  metric and a unique orientation that, on each tangent space, are
  compatible with the nondegenerate $3$--form $\phi$ as in
  \autoref{def:nondegenerate} (see \autoref{thm:imO}). %
  Thus each tangent space of $Y$ carries a cross product
  \begin{equation*}
    T_yY\times T_yY\to T_yY \co
    (u,v)\mapsto u\times v
  \end{equation*}
  such that
  \begin{equation*}
    \phi(u,v,w)=\inner{u\times v}{w}
  \end{equation*} 
  for all $u,v,w\in T_yY$. %
  Moreover, \autoref{thm:form7} gives rise to a natural splitting of
  the space $\Om^k(Y)$ of $k$--forms on $Y$ for each $k$.
\end{remark}

\begin{remark}
  \label{rmk:Spin7}
  An almost $\Spin(7)$--manifold $(X,\Phi)$ admits a unique Riemannian
  metric that, on each tangent space, is compatible with the Cayley-form 
  $\Phi$ as in \autoref{def:nondeg} (see \autoref{thm:CAYLEY}). %
  Moreover, the positivity hypothesis asserts that the $8$--forms 
  \begin{equation*}
    \Phi\wedge\Phi,\qquad
    \iota(v)\iota(u)\Phi\wedge\iota(v)\iota(u)\Phi\wedge\Phi
  \end{equation*}
  induce the same orientation whenever $u,v\in T_xX$ are 
  linearly independent (see \autoref{def:cayleypos}). %
  Thus each tangent space of $X$ carries a positive 
  triple cross product 
  \begin{equation*}
    T_xX\times T_xX\times T_xX\to T_xX\co
    (u,v,w)\mapsto u\times v\times w
  \end{equation*}
  such that
  \begin{equation*}
    \Phi(\xi,u,v,w)=\inner{\xi}{u\times v\times w}
  \end{equation*}
  for all $\xi,u,v,w\in T_xX$. %
  Moreover, \autoref{thm:form8} gives rise to a natural splitting of
  the space $\Om^k(X)$ of $k$--forms on $X$ for each $k$.
\end{remark}

Every spin $7$--manifold admits a $\Gtwo$--structure
\cite{Lawson1989}*{Theorem 10.6}; concrete examples are $S^7$
(considered as unit sphere in the octonions), $S^1\times Z$ where $Z$
is a Calabi--Yau $3$--fold and various resolutions of $T^7\!/\Gamma$
where $\Gamma$ is an appropriate finite group, see
\cite{Joyce2000}.  %
A spin $8$--manifold $X$ admits a $\Spin(7)$--structure if and only if
either $\chi(\slS^+) = 0$ or $\chi(\slS^-) = 0$
\cite{Lawson1989}*{Theorem 10.7}; concrete examples can be obtained
from almost $\Gtwo$--manifolds, Calabi--Yau $4$--folds and various
resolutions of~$T^8\!/\Gamma$.

\begin{definition}
  \label{def:torsionfree}
  An almost $\Gtwo$--manifold $(Y,\phi)$ is called a
  \defined{$\Gtwo$--manifold} if $\phi$ is harmonic with respect to
  the Riemannian metric in \autoref{rmk:G2} and we say that $\phi$ is
  \defined{torsion-free}. %
  An almost $\Spin(7)$--manifold $(X,\Phi)$ is called a
  \defined{$\Spin(7)$--manifold} if $\Phi$ is closed (and, hence,
  harmonic with respect to the Riemannian metric in
  \autoref{rmk:Spin7}) and we say that $\Phi$ is
  \defined{torsion-free}.
\end{definition}

\begin{remark}
  \label{rmk:parallel}
  Let $(Y,\phi)$ be an almost $\Gtwo$--manifold equipped with the
  metric of \autoref{rmk:G2}. %
  Then $\phi$ is harmonic if and only if $\phi$ is parallel with
  respect to the Levi--Civita connection and hence is preserved by
  parallel transport. It follows that the holonomy of a
  $\Gtwo$--manifold is contained in the group $\Gtwo$
  \cite{Fernandez1982}. %
  It also follows that the splitting of \autoref{thm:form7} is
  preserved by the Hodge Laplace operator and hence passes on to the
  de Rham cohomology. %
  Exactly the same holds for an almost $\Spin(7)$--manifold $(X,\Phi)$
  equipped with the metric of \autoref{rmk:Spin7}. %
  The $4$--form $\Phi$ is closed (and hence harmonic) if and only if
  it is parallel with respect to the Levi--Civita connection
  \cite{Bryant1987}. %
  Thus the holonomy of a $\Spin(7)$ manifold is contained in
  $\Spin(7)$ and the splitting of its spaces of differential forms in
  \autoref{thm:form8} descends to the de Rham cohomology.
\end{remark}

\begin{remark}[Construction methods]
  \label{Rmk_ManifoldConstruction}
  Examples of manifolds with torsion-free $\Gtwo$-- or
  $\Spin(7)$--structures are much harder to construct.  There are
  however a number of construction techniques (all based on gluing
  methods): Joyce's generalized Kummer construction for $\Gtwo$-- and
  $\Spin(7)$--manifolds
  \cites{Joyce1996,Joyce1996a,Joyce1996b,Joyce2000} based on resolving
  orbifolds of the form $T^7\!/\Gamma$ and $T^8\!/\Gamma$; a method of
  Joyce's for constructing $\Spin(7)$--manifolds from real singular
  Calabi--Yau $4$--folds \cite{Joyce1999}; and the twisted connected 
  sum construction invented by Donaldson, pioneered by Kovalev~\cite{Kovalev2003}, 
  and extended and improved by Kovalev--Lee \cite{Kovalev2011} 
  and Corti--Haskins--Nordstr{\"o}m--Pacini \cites{Corti2012,Corti2012a}.
\end{remark}


\subsection{The gauge theory picture}
\label{sec:DT2}

We close these notes with a brief review of certain partial
differential equations arising in Donaldson--Thomas
theory~\cite{Donaldson1998}. %
We first discuss the gauge theoretic setting. %
Let $(Y,\phi)$ be a $\Gtwo$--manifold with coassociative calibration
$\psi:=*\phi$ and $E\to Y$ a $\rG$--bundle with compact semi-simple
structure group $\rG$. %
In~\cite{Donaldson1998} Donaldson and Thomas introduce a
\defined{$\Gtwo$--Chern--Simons functional}
\begin{equation*}
  \CS^\psi\co \sA(E)\to\R
\end{equation*}
on the space of connections on $E$. %
The functional depends on the choice of a reference connection
$A_0\in\sA(E)$ satisfying $ F_{A_0}\wedge\psi=0 $ and is given by
\begin{equation}
  \label{eq:CSDT}
  \CS^\psi(A_0+a) 
  := \frac{1}{2}\int_Y \left(\winner{d_{A_0} a}{a}
    + \frac{1}{3}\winner{a}{[a\wedge a]}\right)\wedge \psi
\end{equation}
for $a\in\Om^1(Y,\End(E))$. %
The differential of $\CS$ has the form
\begin{align*}
  \delta\CS^\psi(A)a = \int_N \winner{F_{A}}{a}\wedge\psi
\end{align*}
for $A\in\sA(E)$ and $a\in T_A\sA(E)=\Om^1(Y,\End(E))$. %
Thus a connection $A$ is a critical point of $\CS^\psi$ if and only if
\begin{equation}
  \label{eq:crit}
  F_A \wedge \psi = 0.
\end{equation}
By \autoref{thm:form7} this is equivalent to the equation
$*(F_A\wedge\phi)=-F_A$ and hence to $\pi_7(F_A) = 0$. %
A connection $A$ that satisfies equation~\eqref{eq:crit} is called a
\defined{$\Gtwo$--instanton}. %
As in the case of flat connections on $3$--manifolds
equation~\eqref{eq:crit} becomes elliptic with index zero after
augmenting by a suitable gauge fixing condition (which we do not
elaborate on here). %
The negative gradient flow lines of the $\Gtwo$--Chern--Simons
functional are the $1$--parameter families of connections
$\R\to\sA(E)\co t\mapsto A(t)$ satisfying the partial differential
equation
\begin{equation}
  \label{eq:floer}
  \p_t A = -*(F_A \wedge \psi),
\end{equation}
where $F_A=F_{A(t)}$ is understood as the curvature of the connection
$A(t)\in\sA(E)$ for a fixed value of $t$. %
For the study of the solutions of~\eqref{eq:floer} it is interesting
to observe that, by equation~\eqref{eq:phipsi23} in
\autoref{le:phipsi}, every connection $A$ on $Y$ satisfies the energy
identity
\begin{equation*}
  \int_Y\abs{F_A}^2\dvol_Y = \int_Y\abs{F_A\wedge\psi}^2\dvol_Y
  -\int_Y\winner{F_A}{F_A}\wedge\phi.
\end{equation*}

A smooth solution of~\eqref{eq:floer} can also be thought of as
connection $\A$ on the pullback bundle $\bE$ of $E$ over $\R\times
Y$. %
The curvature of this connection is given by
\begin{equation*}
  F_\A = F_A + dt\wedge \p_t A
  = F_A - dt\wedge*(F_A\wedge\psi).
\end{equation*}
Hence, it follows from \autoref{thm:form8} and~\autoref{thm:form8G2}
that $F_\A$ satisfies
\begin{equation}
  \label{eq:DT}
  *(F_\A \wedge \Phi) = -F_\A
\end{equation}
or, equivalently, $\pi_7(F_\A) = 0$. %
Conversely, a connection on $\bE$ satisfying equation~\eqref{eq:DT} can
be transformed into temporal gauge and hence corresponds to a solution
of~\eqref{eq:floer}. %
It is interesting to observe that equation~\eqref{eq:DT} makes sense
over any $\Spin(7)$--manifold. %
Solutions of~\eqref{eq:DT} are called
\defined{$\Spin(7)$--instantons}. %
This discussion is completely analogous to Floer--Donaldson theory in
${3+1}$ dimensions. %
The hope is that one can construct an analogous quantum field theory
in dimension ${7+1}$. %
Moreover, as is apparent from \autoref{thm:form7SU3}
and~\autoref{thm:form8SU4}, this theory will interact with theories in
complex dimensions $3$ and $4$. %
The ideas for the real and complex versions of this theory are
outlined in~\cites{Donaldson1998,Donaldson2009}.

\begin{remark}
  For construction methods and concrete examples 
  of $\Gtwo$--instantons and
  $\Spin(7)$--instantons we refer to
  \cites{Walpuski2011,SaEarp2013,Walpuski2015} and
  \cites{Tanaka2012,Walpuski2014}.
\end{remark}


\subsection{The submanifold picture}
\label{sec:DT3}

There is an analogue of the $\Gtwo$--Chern--Simons functional on the
space of $3$--di\-men\-sio\-nal submanifolds of $Y$, whose critical points
are the associative submanifolds of $Y$ and whose gradient flow lines
are Cayley submanifolds of $\R\times Y$~\cite{Donaldson1998}. %
This is the submanifold part of the conjectural Donaldson--Thomas
field theory.

More precisely, let $(Y,\phi)$ be a $\Gtwo$--manifold with
coassociative calibration $\psi=*\phi$ and let $S$ be a compact
oriented $3$--manifold without boundary. %
Denote by $\sF$ the space of smooth embeddings $f\co S\to Y$ such that
$f^*\phi$ vanishes nowhere. %
Then the group $\sG:=\Diff^+(S)$ of orientation preserving
diffeomorphism of $S$ acts on $\sF$ by composition. %
The quotient space
\begin{equation*}
  \sS:=\sF/\sG
\end{equation*}
can be identified with the space of oriented $3$--dimensional 
submanifolds of $Y$ that are diffeomorphic to $S$ 
and have the property that the restriction of $\phi$ 
to each tangent space is nonzero; the identification sends 
the equivalence class $[f]$ of an element $f\in\sF$ 
to its image $f(S)$. 

Given $f\in\sF$ the tangent space of $\sS$ at $[f]$ 
can be identified with the quotient
\begin{equation*}
  T_{[f]}\sS = \frac{\Om^0(S,f^*TY)}
  {\left\{df\circ\xi : \xi\in\Vect(S)\right\}}.
\end{equation*}
If $g\in\sG$ is an orientation preserving diffeomorphism of $S$, then
$g^*f:=f\circ g$ is another representative of the equivalence class
$[f]$ and the two quotient spaces can be naturally identified via
$[\hat f]\mapsto[\hat f\circ g]$.

Let us fix an element $f_0\in\sF$ and denote by $\tsF$ the universal
cover of $\sF$ based at $f_0$. %
Thus the elements of $\tsF$ are equivalence classes of smooth maps
$\tf\co [0,1]\times S\to Y$ such that $\tf(0,\cdot)=f_0$ and
$\tf(t,\cdot)=:f_t\in\sF$ for all~$t$. %
Thus we can think of $\tf=\{f_t\}_{0\le t\le1}$ as a smooth path in
$\sF$ starting at $f_0$, and two such paths are \emph{equivalent} iff
they are smoothly homotopic with fixed endpoints.
$\tsF\to\sF$ sends $\tf$ to $f:=\tf(1,\cdot)$. %
The universal cover of $\sS$ is the quotient
\begin{equation*}
  \tsS := \tsF/\tsG
\end{equation*}
where $\tsG$ denotes the group of smooth isotopies
$[0,1]\to\Diff(S)\co t\mapsto g_t$ starting at the identity. %
Now the space $\tsF$ carries a natural $\tsG$--invariant action
functional $\sA\co \tsF\to\R$ defined by
\begin{equation*}
  \sA(\tf) := - \int_{[0,1]\times S}\tf^*\psi
  = -\int_0^1\int_S f_t^*\left(\iota(\p_tf_t)\psi\right)\,dt.
\end{equation*}
This functional is well defined because $\psi$ is closed and it
evidently descends to $\tsS$. %
Its differential is the $1$--form $\delta\sA$ on $\sF$ given by
\begin{equation*}
  \delta\sA(f)\hat f = -\int_S f^*\left(\iota(\hat f)\psi\right)
\end{equation*}
This $1$--form is $\sG$--invariant in that
$\delta\sA(g^*f)g^*\hat f=\delta\sA(f)\hat f$ and horizontal in that
$\delta\sA(f)df\xi=0$ for $\xi\in\Vect(S)$. %
Hence, $\delta\sA$ descends to a $1$--form on $\sS$.

\begin{lemma}
  \label{le:critA}
  An element $[\tf]=[\{f_t\}]\in\tsF$ is a critical point of $\sA$ if
  and only if the image of $f:=f_1\co S\to Y$ is an
  \defined{associative submanifold} of $Y$ (that is, 
  each tangent space is an associative subspace).
\end{lemma}

\begin{proof}
  We have $\delta\sA(f)=0$ if and only if
  $\psi(\hat f(x),df(x)\xi,df(x)\eta,df(x)\zeta)=0$ 
  for all $\hat f\in\Om^0(S,f^*TY)$, 
  all $x\in S$, and all $\xi,\eta,\zeta\in T_xS$. %
  This means that $\psi(u,v,w,\cdot)=0$ for all $q\in f(S)$ and all
  $u,v,w\in T_qf(S)$. %
  By definition of the coassociative calibration $\psi$ in
  \autoref{le:psi} this means that $[u,v,w]=0$ for all
  ${u,v,w\in T_qf(S)}$ where
  $ T_qY\times T_qY\times T_qY\to T_qY\co (u,v,w)\mapsto[u,v,w] $
  denotes the associator bracket defined by~\eqref{eq:associator}. %
  By \autoref{def:assoc} this means that~$T_qf(S)$~is an associative
  subspace of $T_qY$ for all $q\in f(S)$. %
  This proves \autoref{le:critA}.
\end{proof}

The tangent space of $\sF$ at $f$ carries a natural $L^2$ inner
product given by
\begin{equation}
  \label{eq:L2}
  \Inner{\hat f_1}{\hat f_2}_{L^2}
  := \int_S\inner{\hat f_1}{\hat f_2}\,f^*\phi
\end{equation}
for $\hat f_1,\hat f_2\in\Om^0(S,f^*TY)$.
This can be viewed as a $\sG$--invariant metric on~$\sF$.

\begin{lemma}
  \label{le:gradA}
  The gradient of $\sA$ at an element $f\in\sF$ with respect
  to the inner product~\eqref{eq:L2} is given by 
  \begin{equation*}
    \grad\,\sA(f) = \frac{[df\wedge df\wedge df]}{f^*\phi}
    \in\Om^0(S,f^*TY),
  \end{equation*}
  where $[df\wedge df\wedge df]\in\Om^3(S,f^*TY)$ 
  denotes the $3$--form
  \begin{equation*}
    T_xS\times T_xS\times T_xS\to T_{f(x)}Y\co 
    (\xi,\eta,\zeta)\mapsto [df(x)\xi,df(x)\eta,df(x)\zeta].
  \end{equation*}
\end{lemma}

\begin{proof}
  The gradient of $\sA$ at an element $f\in\sF$ is the vector field
  $\grad\,\sA(f)$ along $f$ defined by
  \begin{equation*}
    \int_S\inner{\grad\,\sA(f)}{\hat f}f^*\phi
    = -\int_S f^*\left(\iota(\hat f)\psi\right)
    = \int_S \inner{[df\wedge df\wedge df]}{\hat f}.
  \end{equation*}
  Here the last equation follows from the identity
  \begin{equation*}
    -\psi(\hat f,u,v,w)=\psi(u,v,w,\hat f)=\inner{[u,v,w]}{\hat f}
  \end{equation*}
  (see equation~\eqref{eq:psi} in \autoref{le:psi}). %
  This proves \autoref{le:gradA}.
\end{proof}

We emphasize that the gradient of $\sA$ at $f$ 
is pointwise orthogonal to the image of $df$. %
This is of course a consequence of the fact that the $1$--form
$\delta\sA$ on $\sF$ and the inner product on $T\sF$ are
$\sG$--invariant. %
Now a negative gradient flow line of $\sA$ is a smooth map
\begin{equation*}
  \R\times S\to Y\co (t,x)\mapsto u_t(x)
\end{equation*}
that satisfies the partial differential equation
\begin{equation}
  \label{eq:gradflowA}
  \p_tu_t(x) + 
  \frac{[du_t(x)e_1,du_t(x)e_2,du_t(x)e_3]}
  {\phi(du_t(x)e_1,du_t(x)e_2,du_t(x)e_3)}
  = 0
\end{equation}
for all $(t,x)\in\R\times S$ and every frame $e_1,e_2,e_3$ of
$T_xS$. %
Moreover, we require of course that $u_t$ is an embedding for every
$t$ and that $u_t^*\phi$ vanishes nowhere.

\begin{lemma}
  \label{le:gradflowA}
  Let $\R\times S\to Y\co (t,x)\mapsto u_t(x)$ be a smooth map such
  that ${u_t\in\sF}$ for every $t$. %
  Let $\xi_t\in\Vect(S)$ be chosen such that
  \begin{equation}
    \label{eq:uxi}
    \p_tu_t(x) - du_t(x)\xi_t(x) \perp \im\,du_t(x)
    \qquad\forall (t,x)\in\R\times S.
  \end{equation}
  Then the set
  \begin{equation}
    \label{eq:Sigma}
    \Sigma := \left\{(t,u_t(x)) : t\in\R,\,x\in S\right\}
  \end{equation}
  is a \defined{Cayley submanifold} of $\R\times Y$ (that is, each
  tangent space is a Cayley subspace) with respect to the Cayley
  calibration $\Phi:=dt\wedge\phi+\psi $ if and only if
  \begin{equation}
    \label{eq:gradflowA1}
    \p_tu_t(x) - du_t(x)\xi_t(x)
    + \frac{[du_t(x)e_1,du_t(x)e_2,du_t(x)e_3]}
    {\phi(du_t(x)e_1,du_t(x)e_2,du_t(x)e_3)}
    =0
  \end{equation}
  for every pair $(t,x)\in\R\times S$ and every 
  frame $e_1,e_2,e_3$ of $T_xS$.
\end{lemma}

\begin{proof}
  Fix a pair $(t,x)\in\R\times S$ and choose 
  a basis $e_1,e_2,e_3$ of $T_xS$. %
  By \autoref{thm:tcW}~\itref{It_tcW3} the triple cross product of the three
  tangent vectors
  \begin{equation*}
    (0,du_t(x)e_1),\qquad
    (0,du_t(x)e_2),\qquad (0,du_t(x)e_3)
  \end{equation*}
  of $\Sigma$ is the pair
  \begin{equation*}
    \bigl(\phi(du_t(x)e_1,du_t(x)e_2,du_t(x)e_3),
    -[du_t(x)e_1,du_t(x)e_2,du_t(x)e_3]\bigr).
  \end{equation*}
  Since this pair is orthogonal to the three vectors $(0,du_t(x)e_i)$
  and its first component is nonzero, it follows that our pair is
  tangent to $\Sigma$ if and only if it is a scalar multiple of the
  pair $(1,\p_tu_t(x) - du_t(x)\xi_t(x))$. %
  This is the case if and only if~\eqref{eq:gradflowA1} holds. %
  Hence, it follows from \autoref{le:cayley-sub} that $\Sigma$ is a
  Cayley submanifold of $\R\times Y$ if and only if $u$ satisfies
  equation~\eqref{eq:gradflowA1}. %
  This proves \autoref{le:gradflowA}.
\end{proof}

\autoref{le:gradflowA} shows that every negative gradient flow line of
$\sA$ determines a Cayley submanifold $\Sigma\subset\R\times Y$
via~\eqref{eq:Sigma} and, conversely, every Cayley submanifold
$\Sigma\subset\R\times Y$, with the property that the projection
$\Sigma\to\R$ is a proper submersion, can be parametrized as a
negative gradient flow line of~$\sA$ (for some $S$). %
Thus the negative gradient trajectories of~$\sA$ are solutions of an
elliptic equation, after taking account of the action of the infinite
dimensional reparametrization group $\sG$. %
They minimize the energy
\begin{align*}
  E(u,\xi) 
  &:=
    \tfrac{1}{2}\int_{-\infty}^\infty\int_S
    \left(\abs{\p_tu_t-du_t\xi_t}^2
    +\abs*{\frac{[du_t\wedge du_t\wedge du_t]}{u_t^*\phi}}^2
    \right)u_t^*\phi\,dt  \\
  &\;= 
    \tfrac{1}{2}\int_{-\infty}^\infty\int_S
    \abs*{\p_tu_t-du_t\xi_t
    + \frac{[du_t\wedge du_t\wedge du_t]}{u_t^*\phi}}^2 
    u_t^*\phi\,dt  
    + \int_{\R\times S}u^*\psi.
\end{align*}
For studying the solutions of~\eqref{eq:gradflowA1} it will
be interesting to introduce the energy density $e_f\co S\to\R$ of 
an embedding $f\in\sF$ via 
\begin{equation*}
  e_f(x) := \frac{\det\bigl(\inner{df(x)e_i}{df(x)e_j}_{i,j=1,2,3}\bigr)}
  {\phi(df(x)e_1,df(x)e_2,df(x)e_3)^2}
\end{equation*}
for every $x\in S$ and every frame $e_1,e_2,e_3$ of $T_xS$. %
Then $e_{f\circ g}=e_f\circ g$ for every (orientation preserving)
diffeomorphism $g$ of~$S$ and so the energy 
\begin{equation}
  \label{eq:Ef}
  \sE(f) := \int_Se_f f^*\phi
\end{equation}
is a $\sG$--invariant function on $\sF$. %
Moreover, it follows from \autoref{le:assocphi} that
\begin{equation*}
  \sE(f) = \int_S
  \abs*{\frac{[df\wedge df\wedge df]}{f^*\phi}}^2
  f^*\phi + \int_Sf^*\phi.
\end{equation*}
If $\phi$ is closed, then the last term on the right is a topological
invariant. %
Moreover, the first term vanishes if and only if $f$ is a critical
point of the action functional $\sA$. %
Thus the critical points of $\sA$ are also the absolute minima of the
energy $\sE$ (in a given homology class).

\subsection{Outlook: difficulties and new phenomena}

These observations are the starting point of a conjectural
Floer--Donald\-son type theory in dimensions seven and eight, as
outlined in the paper by Donaldson and Thomas~\cite{Donaldson1998}. %
The analytical difficulties one encounters when making this precise
are formidable, including non-compactness phenomena in codimension
four~\cite{Tian2000}~and two in the gauge theory and submanifold
theory respectively.  %
The work of Donaldson and Segal~\cite{Donaldson2009} explains that
this leads to new geometric phenomena linking the gauge theory and the
submanifold theory. %
It is now understood that neither the naive approach to counting
$\Gtwo$--instantons~\cites{Donaldson2009,Walpuski2013a} nor that of
counting associative submanifolds~\cite{Nordstrom2013} can work on
their own. %
There are, however, ideas of how the theories outlined
in~\autoref{sec:DT2} and~\autoref{sec:DT3} have to be combined and
extended to obtain new invariants~\cites{Donaldson2009,Haydys2014}.


\bibliography{aux/refs}

\end{document}